\documentclass[a4paper]{article}
\usepackage{amsmath,amssymb,amscd,amsthm,amsfonts}
\usepackage{latexsym}
\usepackage{dsfont}
\usepackage{stmaryrd}
\usepackage{tikz-cd}
\usepackage{upgreek}
\usepackage[cal=boondox,calscaled=.96]{mathalfa}
\usepackage[debug]{hyperref}

\newcommand{\QQ}{\mathbb{Q}}
\newcommand{\ZZ}{\mathbb{Z}}
\newcommand{\N}{\mathbb{N}}

\newcommand{\fn}{\mathfrak{n}}
\newcommand{\ft}{\mathfrak{t}}

\newcommand{\DD}{\mathcal{D}}
\newcommand{\OO}{\mathcal{O}}

\newcommand{\sbullet}{{\hspace{.1em}\scriptstyle\bullet\hspace{.1em}}}

\newcommand{\dU}{\mathds{U}}
\newcommand{\dI}{\mathds{I}}
\newcommand{\dT}{\mathds{T}}

\DeclareMathOperator{\ord}{ord}

\DeclareMathOperator{\EXP}{\mathcal{E}}

\DeclareMathOperator{\Aut}{Aut}

\DeclareMathOperator{\supp}{\rm supp}

\DeclareMathOperator{\U}{\mathcal{U}}
\newcommand{\Ufil}{\U_{\text{\rm fil}}}
\newcommand{\Ugr}{\U_{\text{\rm gr}}}

\newcommand{\bupsigma}{{\boldsymbol\upsigma}}
\newcommand{\bupchi}{{\boldsymbol\upchi}}
\newcommand{\bupupsilon}{{\boldsymbol\upupsilon}}

\newcommand{\bupvartheta}{{\boldsymbol\upvartheta}}

\DeclareMathOperator{\Sim}{Sym}
\DeclareMathOperator{\Der}{Der}
\DeclareMathOperator{\fDer}{\mathcal{D}er}
\DeclareMathOperator{\Hom}{Hom}
\DeclareMathOperator{\diff}{\mathcal{D}}
\DeclareMathOperator{\End}{End}
\DeclareMathOperator{\gr}{\rm gr}
\DeclareMathOperator{\height}{\rm ht}

\DeclareMathOperator{\Lie}{\rm Lie}

\DeclareMathOperator{\fLie}{\mathcal{L\! ie}}

\newcommand{\pcirc}{{\scriptstyle \,\circ\,}}

\DeclareMathOperator{\init}{in}

\newcommand{\smallcirc}{{\scriptstyle \circ}}
\newcommand\Id{{\rm Id}}

\DeclareMathOperator{\Ad}{\rm Ad}

\DeclareMathOperator{\fAd}{\mathcal{A\! d}}

\newcommand{\bfu}{{\bf u}}
\newcommand{\bfv}{{\bf v}}
\newcommand{\bfs}{{\bf s}}
\newcommand{\bft}{{\bf t}}

\newcommand{\bvartheta}{{\boldsymbol\vartheta}}

\DeclareMathOperator{\HS}{HS}
\DeclareMathOperator{\Ider}{IDer}

\DeclareMathOperator{\Sub}{\mathcal{S}}

\DeclareMathOperator{\opvarphi}{\varphi}

\DeclareMathOperator{\coideal}{\mathcal{C\!I}}
\newcommand{\coide}[1]{\coideal\left(#1\right)}
\DeclareMathOperator{\Uppsib}{\overline{\Uppsi}}
\DeclareMathOperator{\Uppsibb}{\overline{\overline{\Uppsi}}}
\DeclareMathOperator{\Upgammab}{\overline{\Upgamma}}
\DeclareMathOperator{\Upgammabb}{\overline{\overline{\Upgamma}}}

\newtheorem{thm}{Theorem}[subsection]
\newtheorem{cor}[thm]{Corollary}
\newtheorem{prop}[thm]{Proposition}
\newtheorem{lemma}[thm]{Lemma}
\newtheorem{defi}[thm]{Definition}

\newtheorem{exam}[thm]{Example}

\newtheorem{nota}[thm]{Remark}
\newcommand{\numero}{\refstepcounter{thm}\noindent {\bf  \thethm\ }}
\newtheorem{notacion}[thm]{Notation}

\newtheorem{question}[thm]{Question}

\title{Rings of differential operators as enveloping algebras of Hasse--Schmidt derivations}

\author{Luis Narv\'{a}ez Macarro\thanks{Partially supported by MTM2016-75027-P, P12-FQM-2696
 and FEDER. }}

\date{}

\begin{document}

\maketitle

\mbox{} \hfill \parbox{6cm}{\scriptsize \it Il semble donc (et c'est le point de vue de H. Hasse, F.K. Schmidt et O. Teichm\"uller) que l'on ne puisse \'etudier les op\'erateurs $\Delta_k$ isolement, mais uniquement le syst\`eme  qu'ils forment avec les relations qui les relient.\\

\noindent Jean Dieudonn\'e \cite{dieu_ICM1954}}
\bigskip

\begin{abstract} 

Let $k$ be a commutative ring and $A$ a commutative $k$-algebra. In this paper we introduce the notion of  enveloping algebra of Hasse--Schmidt derivations of $A$ over $k$ and we prove that, under suitable smoothness hypotheses, the canonical map from the above enveloping algebra to the ring of differential operators $\diff_{A/k}$ is an isomorphism. This result generalizes the characteristic 0 case in which the ring $\diff_{A/k}$ appears as the enveloping algebra of the Lie-Rinehart algebra of the usual $k$-derivations of $A$ provided that $A$ is smooth over $k$.
\medskip


\end{abstract}

\section*{Introduction}

In classical $\DD$-module theory, left $\DD_X$-modules on a smooth space $X$ (e.g. a smooth algebraic variety over a field of characteristic $0$, or a complex smooth analytic manifold, or a smooth rigid analytic space over a complete ultrametric field of characteristic $0$, etc.) are the same as modules over the structure sheaf $\OO_X$ endowed with an integrable connection, which is equivalent to an $\OO_X$-linear action of the module of derivations $\fDer_k(\OO_X)$  satisfying Leibniz rule and compatible with Lie brackets. A similar result holds for right $\DD_X$-modules. This fact plays a basic role in classical $\DD$-module theory, for instance in the definition of various operations or in the canonical right $\DD_X$-module structure on top differential forms on $X$. It can be conceptually stated as saying that the sheaf $\DD_X$ is the {\em enveloping algebra} of the {\em Lie algebroid} $\fDer_k(\OO_X)$ and it is strongly related with the canonical isomorphism of graded $\OO_X$-algebras:
\begin{equation} \label{eq:sim-gr} \Sim_{\OO_X} \fDer_k(\OO_X) \xrightarrow{\sim} \gr \diff_{X/k}.
\end{equation}
The main motivation of this paper is the existence of a canonical isomorphism:
\begin{equation} \label{eq:gamma-gr} \Gamma_A \Der_k(A) \xrightarrow{\sim} \gr \diff_{A/k}
\end{equation}
for any commutative ring $k$ (of arbitrary characteristic) and any {\em HS-smooth} $k$-algebra $A$ (see Definition \ref{defi:HS_smooth}), where $\Gamma_A$ denotes the power divided algebra functor (remember that $\Gamma_A = \Sim_A$ if $\QQ\subset A$). The proof of (\ref{eq:gamma-gr}) in \cite{nar_2009} depends on the fact that for a HS-smooth $k$-algebra $A$, any $k$-derivation $\delta:A\to A$ is integrable in the sense of Hasse--Schmidt (see Definition \ref{def:HS-integ}). This result suggests that, under these hypotheses, the ring of differential operators $\diff_{A/k}$ should be recovered in some canonical way from Hasse--Schmidt derivations. This paper is devoted to answering this question.
\medskip

The main difficulty is that Hasse--Schmidt derivations have a much less transparent 
algebraic structure than usual derivations. The module of usual derivations $\Der_k(A)$ carries an $A$-module structure and a $k$-Lie algebra structure, and both are mixed on a {\em Lie-Rinehart algebra} structure, enough to recover the ring of differential operators as its enveloping algebra provided that $\QQ\subset k$ and $A$ is smooth over $k$ (see \cite{rine-63}),
although Hasse--Schmidt derivations were only known to carry a (non-commutative) group structure. In our previous paper \cite{nar_subst_2018}, we introduced and studied the action of substitution maps (between power series rings) on Hasse--Schmidt derivations, to be thought as a substitute of the $A$-module structure on usual derivations. 
\medskip

In this paper we prove that both the group structure and the action of substitution maps allow us to define the {\em enveloping algebra} of Hasse--Schmidt derivations and to prove that, under smoothness hypotheses, this enveloping algebra is canonically isomorphic to the ring of differential operators without any assumption on the characteristic of $k$. A key step in the proof is the existence of a canonical map of graded algebras from the power divided algebra of the module of integrable derivations (in the sense of Hasse--Schmidt) to the graded ring of the enveloping algebra of Hasse--Schmidt derivations. 
\medskip

Let us now comment on the content of this paper.
\medskip

In section 1 we recall and adapt, for the ease of the reader, the material in \cite[\S 1, \S 2, \S 3]{nar_subst_2018}. We will concentrate ourselves in the case of power series rings and modules in a finite number of variables, which will be enough for our main results in section \ref{sect:main}. In the last sub-section we recall the notions of exponential type series and power divided algebras.
\medskip

In section 2 first we recall the notion of Hasse--Schmidt derivation and its basic properties. As we already did in \cite[\S 4]{nar_subst_2018}, we need to study, not only uni-variate Hasse--Schmidt derivations, but also multivariate ones: a $(p,\Delta)$-variate Hasse--Schmidt derivation of our $k$-algebra $A$ is a family  $D= \left(D_\alpha\right)_{\alpha\in\Delta}$ of $k$-linear endomorphisms of $A$ such that $D_0$ is the identity map and 
$$ D_\alpha (x y) = \sum_{\scriptscriptstyle \beta + \gamma = \alpha} D_\beta(x) D_\gamma(y),\quad \forall \alpha\in \Delta, \forall x,y\in A,
$$
where $\Delta \subset \N^p$ is a non-empty {\em co-ideal}, i.e. a subset of $\N^p$ such that everytime $\alpha\in \Delta$ and $\alpha'\leq \alpha$ we have $\alpha'\in \Delta$. An important idea is to think of Hasse--Schmidt derivations as series $D=\sum_{\scriptscriptstyle \alpha\in\Delta} D_\alpha \bfs^\alpha$ in the quotient ring $R[[\bfs]]_\Delta$ of the power series ring $R[[\bfs]] = R[[s_1,\dots,s_p]]$, $R=\End_k(A)$, by the two-sided monomial ideal generated by all $\bfs^\alpha$ with $\alpha \in \N^p \setminus \Delta$. In the second sub-section we recall \cite[\S 5]{nar_subst_2018} on the action of substitution maps on Hasse--Schmidt derivations. The starting point is simple: given a substitution map $\varphi: A[[s_1,\dots,s_p]]_\Delta \to A[[t_1,\dots,t_q]]_\nabla $ and a $(p,\Delta)$-variate Hasse--Schmidt derivation $D=\sum_{\scriptscriptstyle \alpha\in\Delta} D_\alpha \bfs^\alpha$ we may consider a new ($q,\nabla)$-variate Hasse--Schmidt derivation given by:
$$\varphi\sbullet D := \sum_{\scriptscriptstyle \alpha\in\Delta} \varphi(\bfs^\alpha) D_\alpha.
$$
In the last sub-section, we first recall the notion of integrable derivation: a $k$-derivation $\delta:A\to A$ is said to be $m$-integrable if there is a uni-variate Hasse--Schmidt derivation $D= \left(D_i\right)_{i=0}^m$ such that $D_1=\delta$, and second we recall the main results in \cite{nar_2009}.
\medskip

Section 3 contains the original results of this paper. First, we introduce the notion of {\em HS-module}, as a generalization of the classical notion of module with an integrable connection. 
Roughly speaking, a left HS-module is a module $M$ over our $k$-algebra $A$ on which Hasse--Schmidt derivations act ``globally'', in a compatible way with the group structure and the action of substitution maps, and satisfying a Leibniz rule. More precisely, for each $(p,\Delta)$-variate Hasse--Schmidt derivation $D=\sum_{\scriptscriptstyle \alpha\in\Delta} D_\alpha \bfs^\alpha$ of $A$, $M$ is endowed with a $k[[\bfs]]_\Delta$-linear automorphism $\Psi^p_\Delta(D): M[[\bfs]]_\Delta \to M[[\bfs]]_\Delta$ congruent to the identity modulo $\langle \bfs \rangle$, in such a way that:
\begin{itemize}
\item[-)]
 The $\Psi^p_\Delta(-)$ are group homomorphism.
\item[-)] For each substitution map $\varphi: A[[\bfs]]_\Delta \to A[[\bft]]_\nabla $ we have $\Psi^q_\nabla(\varphi \sbullet D) = \varphi \sbullet \Psi^p_\Delta(D)$.
\item[-)] (Leibniz rule) For each $a\in A$ we have $\Psi^p_\Delta(D) a = D(a) \Psi^p_\Delta(D)$.
\end{itemize}
Any $\diff_{A/k}$-module is obviously a HS-module, since Hasse--Schmidt derivations act through their components, which are differential operators. Namely, if $M$ is a left $\diff_{A/k}$-module, for each $(p,\Delta)$-variate Hasse--Schmidt derivation $D=\sum_{\scriptscriptstyle \alpha\in\Delta} D_\alpha \bfs^\alpha$ of $A$ we define $\Psi^p_\Delta(D)$ as:
$$ \Psi^p_\Delta(D) (m) = \sum_{\scriptscriptstyle \alpha\in\Delta} (D_\alpha m) \bfs^\alpha,\quad \forall m\in M.
$$
The basic question is whether a HS-module structure can be lifted to a $\diff_{A/k}$-module structure or not.
\smallskip

To illustrate the notion of HS-module, or more precisely, the notion of {\em pre-HS-module structure} (i.e. the compatibility with substitution maps only holds for substitution maps with constant coefficients), we give natural actions of Hasse--Schmidt derivations on $\Omega_{A/k}$ and on $\Der_k(A)$ generalizing, respectively, the classical Lie derivative and the adjoint representation of classical derivations.
\smallskip

In the second sub-section we generalize the well known $\otimes$ and $\Hom$ operations on modules with an integrable connection
to the setting of HS-modules. In the last two sub-sections we define the enveloping algebra of Hasse--Schmidt derivations of a commutative algebra, and we prove, by imitating \cite{nar_2009}, that there is a canonical map of graded algebras from the power divided algebra of the module of integrable derivations to the graded ring of the enveloping algebra of Hasse--Schmidt derivations. We finally prove that, under the HS-smoothness hypothesis, 
the former map is an isomorphism and we deduce that the canonical map from the enveloping algebra of Hasse--Schmidt derivations to the ring of differential operators is an isomorphism. As a corollary, HS-modules coincide with $\diff$-modules for HS-smooth algebras.
\smallskip

I would like to thank the referee for the careful reading of the paper.

\section{Notations and preliminaries}

\subsection{Notations}

Throughout the paper we will use the following notations:
\smallskip

\noindent -) $k$ is a commutative ring and $A$ a commutative
$k$-algebra.
\smallskip

\noindent -) $\diff_{A/k}$ is the ring of $k$-linear differential operators of $A$ 
(see \cite{ega_iv_4}).
\smallskip

\noindent -) $\bfs = \{s_1,\dots,s_p\}$, $\bft =\{t_1,\dots,t_q\}$, \dots are sets of variables.
\smallskip

\noindent -) $k$-algebra over $A$: see Definition \ref{def:k-algebra-over-A}.
\smallskip

\noindent -) $\fn_\beta := \{\alpha \in \N^p\ |\ \alpha\leq \beta\}$ ) for $\beta \in \N^p$.
\smallskip

\noindent -) $\ft_m := \{\alpha \in \N^p\ |\ |\alpha|\leq m\}$ with $m\geq 0$.
\smallskip

\noindent -) $\coide{\N^p}$ is the set of all non-empty co-ideals of $\N^p$: see Notation 
\ref{notacion:co-ideal}.
\smallskip

\noindent -) $\tau_{\Delta' \Delta}$ is a truncation map: see (\ref{eq:truncations}).
\smallskip

\noindent -) $\U^p(R;\Delta), \Ufil^p(R;\Delta), \Ugr^p(R;\Delta)$: see Notation \ref{notacion:Ump}.
\smallskip

\noindent -) $r\boxtimes r'$: see Definition \ref{defi:external-x}.
\smallskip

\noindent -) $r\mapsto \widetilde{r}$: see (\ref{eq:tilde-map});\quad $g \mapsto g^e$: see (\ref{eq:g^e}).
\smallskip

\noindent -) $\Hom_k^\pcirc(-,-)$, $\Aut_{k[[\bfs]]_\Delta}^\pcirc(-)$: see Notation \ref{notacion:pcirc}.
\smallskip

\noindent -) $\Sub_A(p,q;\Delta,\nabla)$ is the set of substitution maps: see Definition \ref{def:substitution_maps}.
\smallskip

\noindent -) ${\bf C}_e(\varphi,\alpha)$: see (\ref{eq:exp-substi}).
\smallskip

\noindent -) $\varphi_M$, $\sideset{_M}{}\opvarphi$: see \ref{sec:action-substi};\quad  $\varphi \sbullet r$, $r\sbullet \varphi$: see \ref{nume:def-sbullet}.
\smallskip

\noindent -) $\varphi_*$, $\overline{\varphi_*}$: see (\ref{eq:varphi_star}) and (\ref{eq:overline_varphi_star}).
\smallskip

\noindent -)  $\EXP_m(B)$ is the set of exponential type series: see Definition \ref{defi:exponential-series}.
\smallskip

\noindent -) $\Sim_A M$ is the symmetric algebra of the $A$-module $M$.
\smallskip

\noindent -) $\Gamma_A M$ is the power divided algebra of the $A$-module $M$: see Definition \ref{defi:power-divided-algebra}. 
\smallskip

\noindent -) $\HS_k^p(A;\Delta)$ is the set of $(p,\Delta)$-variate Hasse--Schmidt derivations: see Definition \ref{defi:HS}.
\smallskip

\noindent -) $a \sbullet D$: see Definition \ref{defi:sbullet-0}.
\smallskip

\noindent -) $\varphi^D$, for $\varphi$ a substitution map and $D$ a Hasse--Schmidt derivation: see Proposition \ref{prop:varphi-D-main}.
\smallskip

\noindent -) $\dU_{A/k} = \dT_{A/k}/\dI$ is the enveloping algebra of the Hasse--Schmidt derivations of $A$ over $k$: see Definition \ref{defi:enveloping_HS}.

\subsection{Rings and modules of power series}

Throughout this section, $k$ will be a commutative ring, $A$ a commutative $k$-algebra and $R$ a ring, not-necessarily commutative.
\medskip

Let $p\geq 0$ be an integer and let us call $\bfs = \{s_1,\dots,s_p\}$ a set of $p$ variables. The support of each $\alpha\in \N^p$ is defined as  $\supp \alpha := \{i\ |\ \alpha_i\neq 0\}$. 
The monoid $\N^p$ is endowed with a natural partial ordering. Namely, for $\alpha,\beta\in \N^p$, we define
$$ \alpha \leq \beta\quad \stackrel{\text{def.}}{\Longleftrightarrow}\quad \exists \gamma \in \N^p\ \text{such that}\ \beta = \alpha + \gamma\quad \Longleftrightarrow\quad \alpha_i \leq \beta_i\quad \forall i=1\dots,p.$$
We denote $|\alpha| := \alpha_1+\cdots+\alpha_p$. If $\alpha \leq \beta$ then $|\alpha|\leq |\beta|$. Moreover, if $\alpha \leq \beta$ and $|\alpha| = |\beta|$, then $\alpha=\beta$.
\medskip

Let $M$ be an abelian group and $M[[\bfs]]$ the abelian group of power series with coefficients in $M$.
The {\em support}  of a series $m=\sum_\alpha m_\alpha \bfs^\alpha \in M[[\bfs]]$ is $\supp(m) := \{  \alpha \in \N^p\ |\  m_\alpha \neq 0\} \subset \N^p$. It is clear that $m=0\Leftrightarrow \supp(m) = \emptyset$. 
The {\em order}  of a non-zero series $m=\sum_\alpha m_\alpha \bfs^\alpha \in M[[\bfs]]$ is
$$ \ord (m) := \min \{ |\alpha|\ |\ \alpha \in \supp(m) \} \in \N.
$$
If $m=0$  we define $\ord (0) := \infty$. If $M$ is an $A$-module, then $M[[\bfs]]$ is naturally an $A[[\bfs]]$-module and for $a\in A[[\bfs]]$ and $m,m'\in M[[\bfs]]$ we have 
$\supp(m+m') \subset \supp(m) \cup \supp(m')$, $\supp (am), \supp(ma) \subset \supp(m) + \supp(a)$,
$\ord(m+m') \geq \min \{\ord(m),\ord(m')\}$ and 
$\ord (am), \ord (ma) \geq \ord(a) + \ord(m)$. Moreover, if $\ord(m') > \ord(m)$, then $\ord(m+m') = \ord(m)$.
\medskip

The abelian group $M[[\bfs]]$ is the completion of the abelian group $M[\bfs]$ of polynomials with coefficients in $\bfs$ with respect to the $\langle \bfs \rangle$-adic topology, and its natural topology is also the $\langle \bfs \rangle$-adic topology.
\medskip

When $M=R$ is a ring, $R[[\bfs]]$ is a topological ring. 
If $M$ is an $A$-module, there is a natural $A[[\bfs]]$-linear bicontinuous isomorphism:
\begin{equation}  \label{eq:A[[s]]wotM-new}
   A[[\bfs]] \widehat{\otimes}_A M \stackrel{\sim}{\longrightarrow} M[[\bfs]],
\end{equation}
where $\widehat{\otimes}_A$ indicates the completed tensor product with respect to the natural topology on $A[[\bfs]]$.

\begin{defi} \label{def:k-algebra-over-A}
A $k$-algebra over $A$ \index{algebra over a commutative ring} is a (not-necessarily commutative) $k$-algebra $R$ endowed with a map of $k$-algebras $\iota:A \to R$. A map between two $k$-algebras $\iota:A \to R$ and $\iota':A \to R'$ over $A$ is a map $g:R\to R'$ of $k$-algebras such that $\iota' = g \pcirc \iota$. 
A filtered $k$-algebra over $A$  is a $k$-algebra $(R,\iota)$ over $A$, endowed with a ring filtration $(R_k)_{k\geq 0}$ such that $\iota(A) \subset R_0$.
\end{defi}

A $k$-algebra over $A$ is obviously an $(A;A)$-bimodule. If $R$ is a $k$-algebra over $A$, then the power series ring $R[[\bfs]]$ is a $k[[\bfs]]$-algebra over $A[[\bfs]]$.

\begin{defi} We say that a subset $\Delta \subset \N^p$ is an {\em ideal} (resp. a
{\em co-ideal}) of $\N^p$  
if everytime $\alpha\in\Delta$ and $\alpha\leq \alpha'$ (resp. $\alpha'\leq \alpha$), then $\alpha'\in\Delta$.
\end{defi}
It is clear that $\Delta$ is an ideal if and only if its complement $\Delta^c$ is a co-ideal, and that the union and the intersection of any family of ideals (resp. of co-ideals)  of $\N^p$ is again an ideal (resp. a co-ideal) of $\N^p$. 
Examples of ideals (resp. of co-ideals) of $\N^p$ are the $\beta + \N^p$
(resp. the $\fn_\beta := \{\alpha \in \N^p\ |\ \alpha\leq \beta\}$ ) with $\beta \in \N^p$. The $\ft_m$ defined as $\ft_m := \{\alpha \in \N^p\ |\ |\alpha|\leq m\}$ with $m\geq 0$ are also co-ideals.
Notice that a co-ideal $\Delta\subset \N^p$ is non-empty if and only if $(\ft_0 = \fn_0=)\{0\}  \subset \Delta$.

\begin{notacion} \label{notacion:co-ideal}
The set of all non-empty co-ideals of $\N^p$ will be denoted by $\coide{\N^p}$. 
\end{notacion}

For a co-ideal $\Delta \subset \N^P$ and an integer $m\geq 0$, we denote $\Delta^m := \Delta \cap \ft_m$. If $\Delta\subset \N^P$ is a finite non-empty co-ideal, we define its {\em height} as $\height (\Delta) := \min \{m\in\N\ |\ \Delta \subset \ft_m\} = \max \{|\alpha|\ |\ \alpha\in\Delta\}$.
\medskip

Let $M$ be an $(A;A)$-bimodule central over $k$. 
For each co-ideal $\Delta \subset \N^p$, we denote by $\Delta_M$ the closed sub-$(A[[\bfs];A[[\bfs]])$-bimodule of $M[[\bfs]]$ whose elements are the formal power series $\sum_{\alpha\in\N^p} m_\alpha \bfs^\alpha$ such that $m_\alpha=0$ whenever $\alpha\in \Delta$, i.e. 
\begin{eqnarray*}
&\displaystyle
\Delta_M =
  \left\{m\in M[[\bfs]],\ \supp(m) \subset \Delta^c\right\} = 
\left\{m\in M[[\bfs]],\ \supp(m) \subset \bigcap_{\beta\in \Delta} \fn_\beta^c\right\}=
&\\
&\displaystyle
\bigcap _{\beta\in \Delta} \left\{m\in M[[\bfs]],\ \supp(m) \subset  \fn_\beta^c\right\} =
 \bigcap_{\beta\in \Delta} \left(\fn_\beta\right)_M.
\end{eqnarray*}
For $m\in\N$ we have 
$\left(\ft_m\right)_M = \langle \bfs \rangle^{m+1} M[[\bfs]]$. 
Let us denote by $M[[\bfs]]_\Delta := M[[\bfs]]/\Delta_M$ endowed with the quotient topology (it coincides with the $\langle \bfs \rangle$-adic topology regarded as a $k[[\bfs]]$-module), for which it is a topological bimodule over $(A[[\bfs]]_\Delta;A[[\bfs]]_\Delta)$.
\medskip

When $\Delta = \fn_\alpha$, for some $\alpha \in\N^p$, we will simply denote $M[[\bfs]]_\alpha := M[[\bfs]]_{\fn_\alpha}$. Similarly, when $\Delta = \ft_m$, for some $m\geq 0$, we will simply denote $M[[\bfs]]_m := M[[\bfs]]_{\ft_m}$. 
\medskip

The elements in $M[[\bfs]]_\Delta$ are power series of the form 
$$\sum_{\scriptscriptstyle \alpha\in\Delta} m_\alpha \bfs^\alpha,\quad m_\alpha \in M.
$$
The additive isomorphism
$$ \sum_{\scriptscriptstyle \alpha\in\Delta} m_\alpha \bfs^\alpha \in M[[\bfs]]_\Delta  \mapsto \{m_\alpha\}_{\alpha \in \Delta} \in M^\Delta    $$
is a homeomorphism, where $M^\Delta$ is endowed with the product of discrete topologies on each copy of $M$. 
\medskip

For $\Delta \subset \Delta'$ co-ideals of $\N^p$, we have natural 
 $(A[[\bfs]]_{\Delta'};A[[\bfs]]_{\Delta'})$-linear projections $\tau_{\Delta' \Delta}:M[[\bfs]]_{\Delta'} \longrightarrow M[[\bfs]]_\Delta$, that we call {\em truncations}:
 \begin{equation}  \label{eq:truncations}
\tau_{\Delta' \Delta} : \sum_{\scriptscriptstyle \alpha\in\Delta'} m_\alpha \bfs^\alpha \in M[[\bfs]]_{\Delta'} \longmapsto \sum_{\scriptscriptstyle \alpha\in\Delta} m_\alpha \bfs^\alpha \in M[[\bfs]]_{\Delta}.
\end{equation}
When $\Delta=\ft_m, \Delta'=\ft_{m'}$, $m\leq m'$, we will simply denote $\tau_{m'm}:= \tau_{\ft_{m'}\ft_{m}}$. We have $(A;A)$-linear scissions:
$$\sum_{\scriptscriptstyle \alpha\in\Delta} m_\alpha \bfs^\alpha \in M[[\bfs]]_\Delta \longmapsto \sum_{\scriptscriptstyle \alpha\in\Delta} m_\alpha \bfs^\alpha \in M[[\bfs]]_{\Delta'}
$$ 
which are topological immersions. 
In particular we have natural $(A;A)$-linear topological embeddings $M[[\bfs]]_\Delta \hookrightarrow M[[\bfs]]$ and we define the {\em support} (resp. the {\em order}) of any element in $M[[\bfs]]_\Delta$ as its support (resp. its order) as element of $M[[\bfs]]$. 
We have a bicontinuous isomorphism of $(A[[\bfs]]_\Delta;A[[\bfs]]_\Delta)$-bimodules
$$ M[[\bfs]]_\Delta = \lim_{\stackrel{\longleftarrow}{m\in\N}} M[[\bfs]]_{\Delta^m},
$$
where transition maps in the inverse system are given by truncations. 
For a ring $R$, the $\Delta_R$ are closed two-sided ideals of $R[[\bfs]]$
 and we have a bicontinuous ring isomorphism
$$ R[[\bfs]]_\Delta = \lim_{\stackrel{\longleftarrow}{m\in\N}} R[[\bfs]]_{\Delta^m}.
$$
As in (\ref{eq:A[[s]]wotM-new}), for $A[[\bfs]]_\Delta \otimes_A M$ (resp. $M\otimes_A A[[\bfs]]_\Delta$) endowed with the natural topology, we have that the natural map $A[[\bfs]]_\Delta \otimes_A M \to M[[\bfs]]_\Delta$ (resp. $M\otimes_A A[[\bfs]]_\Delta \to M[[\bfs]]_\Delta$)
is continuous and gives rise to a $(A[[\bfs]]_\Delta;A)$-linear (resp. to a $(A;A[[\bfs]]_\Delta)$-linear) isomorphism
$$ A[[\bfs]]_\Delta \widehat{\otimes}_A M \xrightarrow{\sim} M[[\bfs]]_\Delta\quad\quad \text{(resp. $
M\widehat{\otimes}_A A[[\bfs]]_\Delta \xrightarrow{\sim} M[[\bfs]]_\Delta$)}.
$$
Each $(A;A)$-linear map $h: M \to M'$ between two bimodules induces a linear map (over 
($(A[[\bfs]]_\Delta;A[[\bfs]]_\Delta)$)
\begin{equation}   \label{eq:h-bar-mod}
\overline{h}: \sum_{\scriptscriptstyle \alpha\in\Delta} m_\alpha \bfs^\alpha \in M[[\bfs]_\Delta \longmapsto  
\sum_{\scriptscriptstyle \alpha\in\Delta} h(m_\alpha) \bfs^\alpha \in M[[\bfs]_\Delta.
\end{equation}
We have a commutative diagram
$$
\begin{tikzcd}
A[[\bfs]]_\Delta\widehat{\otimes}_A M \ar[r,"\simeq"] \ar[d,"\Id\widehat{\otimes} h "'] &
M[[\bfs]]_\Delta \ar[d,"\overline{h}"'] & M \widehat{\otimes}_A  A[[\bfs]]_\Delta \ar[l,"\simeq"'] 
\ar[d,"h\widehat{\otimes} \Id "] \\
A[[\bfs]]_\Delta\widehat{\otimes}_A M' \ar[r,"\simeq"] & M'[[\bfs]]_\Delta & \ar[l,"\simeq"'] M' \widehat{\otimes}_A  A[[\bfs]]_\Delta.
\end{tikzcd}
$$
Clearly, if $R$ is a $k$-algebra over $A$, then $R[[\bfs]]_\Delta$ is a $k[[\bfs]]_\Delta$-algebra over $A[[\bfs]]_\Delta$.

\begin{notacion} \label{notacion:Ump} Let $R$ be a ring, $p\geq 1$ and  $\Delta\subset \N^p$ a non-empty co-ideal. We denote by $\U^p(R;\Delta)$ the multiplicative sub-group of the units of $R[[\bfs]]_\Delta$ whose 0-degree coefficient is $1$. The multiplicative inverse of a unit $r\in R[[\bfs]]_\Delta$ will be denoted by $r^*$.
Clearly, $\U^p(R;\Delta)^{\text{\rm opp}} = \U^p(R^{\text{\rm opp}};\Delta)$.
For $\Delta\subset \Delta'$ co-ideals we have $\tau_{\Delta'\Delta}\left(\U^p(R;\Delta')\right) \subset \U^p(R;\Delta)$ and the truncation map
$\tau_{\Delta'\Delta}:\U^p(R;\Delta') \to \U^p(R;\Delta)$ is a group homomorphism. Clearly, we have:
\begin{equation} \label{eq:U-inv-limit-finite}
   \U^p(R;\Delta) = \lim_{\stackrel{\longleftarrow}{m\in\N}} \U^p(R;\Delta^m) = \lim_{\stackrel{\longleftarrow}{\substack{\scriptscriptstyle \Delta' \subset \Delta\\ \scriptscriptstyle  \sharp\Delta'<\infty}}} \U^p(R;\Delta').
\end{equation}
If $p=1$ and $\Delta = \ft_m = \{i\in\N\ |\ i\leq m\}$ we will simply denote $\U(R;m):= \U^1(R;\ft_m)$.
\medskip

\noindent
If $R = \cup_{d\geq 0} R_d$ is a filtered ring, we denote: 
$$\Ufil^p(R;\Delta):= \left\{ \sum_{\scriptscriptstyle \alpha\in\Delta} r_\alpha \bfs^\alpha\in\U^p(R;\Delta)\ \left|\  r_\alpha \in R_{|\alpha|}\ \forall \alpha \in\Delta\right.\right\}.
$$
It is clear that $\Ufil^p(R;\Delta)$ is a subgroup of $\U^p(R;\Delta)$.
\medskip

\noindent If $R= \bigoplus_{d\in\N} R_d$ is a graded ring, we denote: 
$$\Ugr^p(R;\Delta):= \left\{ \sum_{\scriptscriptstyle \alpha\in\Delta} r_\alpha \bfs^\alpha\in\U^p(R;\Delta)\ \left|\  r_\alpha \in R_{|\alpha|}\ \forall \alpha \in\Delta\right.\right\}.
$$
It is clear that $\Ugr^p(R;\Delta)$ is a subgroup of $\U^p(R;\Delta)$.
\medskip

\noindent If $R$ be a filtered ring, we will denote by
$ \bupsigma: \Ufil^p(R;\Delta) \longrightarrow \Ugr^p(\gr R;\Delta)
$
the {\em total symbol map} defined as:
$$ \bupsigma \left( \sum_{\scriptscriptstyle \alpha\in\Delta} r_\alpha \bfs^\alpha \right) := 
\sum_{\scriptscriptstyle \alpha\in\Delta} \sigma_{|\alpha|}(r_\alpha) \bfs^\alpha.
$$
It is clear that $\bupsigma$ is a group homomorphism compatible with truncations.
\medskip

For any ring homomorphism $f:R\to R'$, the induced ring homomorphism $\overline{f}: R[[\bfs]]_\Delta \to R'[[\bfs]]_\Delta$ sends $\U^p(R;\Delta)$ into $\U^p(R';\Delta)$ and so it induces natural group homomorphisms
$\U^p(R;\Delta) \to \U^p(R';\Delta)$. Similar results hold for the filtered or graded cases.
\end{notacion}

\begin{defi} \label{defi:external-x}
Let $R$ be a ring, $p,q\geq 0$, $\bfs=\{s_1,\dots,s_p\},\bft=\{t_1,\dots,t_q\}$ disjoint sets of variables and  $\nabla\subset \N^p, \Delta\subset \N^q$ non-empty co-ideals.
For each $r\in R[[\bfs]]_\nabla, r'\in R[[\bft]]_\Delta$, the {\em external product} $r\boxtimes r'\in R[[\bfs \sqcup \bft]]_{\nabla \times \Delta}$ (notice that $\nabla \times \Delta \subset \N^{p+q}$ is a non-empty co-ideal) is defined as
$$ r\boxtimes r' := \sum_{\scriptscriptstyle (\alpha,\beta)\in \nabla \times \Delta} r_\alpha r'_\beta \bfs^\alpha \bft^\beta.$$
\end{defi}

The above definition is consistent with the existence of natural isomorphism of $(R;R)$-bimodules $R[[\bfs]]_\nabla \widehat{\otimes}_R R[[\bft]]_\Delta \simeq R[[\bfs \sqcup \bft]]_{\nabla \times \Delta}\simeq   R[[\bft \sqcup \bfs]]_{\Delta \times \nabla} \simeq    R[[\bft]]_\Delta \widehat{\otimes}_R R[[\bfs]]_\nabla$. Let us also notice that $1\boxtimes 1 = 1$ and $r\boxtimes r' = (r \boxtimes 1) (1 \boxtimes r')$. Moreover, 
if
$r\in \U^p(R;\nabla)$, $r'\in \U^q(R;\Delta)$, then 
$r\boxtimes r' \in \U^{p+q}(R;\nabla\times \Delta)$
and  $(r\boxtimes r')^* = {r'}^* \boxtimes r^*$.
\medskip

Let $E,F$ be two $A$-modules and $\Delta\subset \N^p$ a non-empty co-ideal. The proof of the following proposition is straightforward.

\begin{prop} \label{prop:induced-cont-maps-M[[s]]_m} Under the above hypotheses, any $k[[\bfs]]_\Delta$-linear map $f:E[[\bfs]]_\Delta \to F[[\bfs]]_\Delta$ is continuous for the natural topologies, and for any co-ideal $\Delta'\subset \N^p$ with $\Delta'\subset \Delta$ we have
$ f\left(\Delta'_E/\Delta_E\right) \subset \Delta'_F/\Delta_F
$
and so there is a unique $k[[\bfs]]_{\Delta'}$-linear map $\overline{f}:E[[\bfs]]_{\Delta'} \to F[[\bfs]]_{\Delta'}$ such that the following diagram is commutative:
$$
\begin{tikzcd}
E[[\bfs]]_\Delta \ar[r,"f"] \ar[d,"\text{trunc.}"'] &   F[[\bfs]]_\Delta  \ar[d,"\text{trunc.}"]\\
E[[\bfs]]_{\Delta'} \ar[r,"\overline{f}"] & F[[\bfs]]_{\Delta'}.
\end{tikzcd}
$$
\end{prop}
\bigskip

\numero \label{nume:tilde}
For each $r= \sum_\beta r_\beta \bfs^\beta \in \Hom_k(E,F)[[\bfs]]_\Delta$ we define $\widetilde{r}: E[[\bfs]]_\Delta \to F[[\bfs]]_\Delta$ by
$$ \widetilde{r} \left( \sum_{\scriptscriptstyle \alpha\in \Delta} e_\alpha \bfs^\alpha \right) :=
\sum_{\scriptscriptstyle \alpha\in \Delta } \left( \sum_{\scriptscriptstyle \beta + \gamma=\alpha} r_\beta(e_\gamma) \right) \bfs^\alpha,
$$
which is obviously a $k[[\bfs]]_\Delta$-linear map. 
\medskip

Let us notice that 
$ \widetilde{r} = \sum_\beta \bfs^\beta  \widetilde{r_\beta}$. It is clear that the map
\begin{equation} \label{eq:tilde-map}
 r\in \Hom_k(E,F)[[\bfs]]_\Delta \longmapsto \widetilde{r} \in \Hom_{k[[\bfs]]_\Delta}(E[[\bfs]]_\Delta,F[[\bfs]]_\Delta)
\end{equation}
is $(A[[\bfs]]_\Delta;A[[\bfs]]_\Delta)$-linear.
\medskip

If $f:E[[\bfs]]_\Delta \to F[[\bfs]]_\Delta$ is a $k[[\bfs]]_\Delta$-linear map, let us denote by $f_\alpha:E \to F$, $\alpha\in\Delta $, the $k$-linear maps defined by
$$ f(e) = \sum_{\scriptscriptstyle \alpha\in \Delta} f_\alpha(e) \bfs^\alpha,\quad \forall e\in E.
$$
If $g:E\to F[[\bfs]]_\Delta$ is a $k$-linear map, we denote by $g^e:E[[\bfs]]_\Delta \to F[[\bfs]]_\Delta$ the unique $k[[\bfs]]_\Delta$-linear map extending $g$ to $E[[\bfs]]_\Delta = k[[\bfs]]_\Delta \widehat{\otimes}_k E$. It is given by
\begin{equation}  \label{eq:g^e}
 g^e \left( \sum_\alpha e_\alpha \bfs^\alpha \right) := \sum_\alpha g(e_\alpha) \bfs^\alpha.
\end{equation}
We have a $k[[\bfs]]_\Delta$-bilinear and $A[[\bfs]]_\Delta$-balanced map
$$ \langle -,-\rangle : (r,e) \in \Hom_k(E,F)[[\bfs]]_\Delta \times E[[\bfs]]_\Delta \longmapsto \langle r,e\rangle := \widetilde{r}(e) \in F[[\bfs]]_\Delta.
$$

\begin{lemma} \label{lema:tilde-map}
With the above hypotheses, the following properties hold:
\begin{enumerate}
\item[1)] The map (\ref{eq:tilde-map})
 is an isomorphism of $(A[[\bfs]]_\Delta;A[[\bfs]]_\Delta)$-bimodules. When $E=F$ it is an isomorphism of 
$k[[\bfs]]_\Delta$-algebras over $A[[\bfs]]_\Delta$.
\item[2)] The restriction map 
$$f \in  \Hom_{k[[\bfs]]_\Delta}(E[[\bfs]]_\Delta,F[[\bfs]]_\Delta) \mapsto f|_E \in \Hom_k(E,F[[\bfs]]_\Delta)$$ is an isomorphism of $(A[[\bfs]]_\Delta;A)$-bimodules.
\item[3)] For $r \in \Hom_k(A,F)[[\bfs]]_\Delta$, we have 
$$ r \in \Der_k(A,F)[[\bfs]]_\Delta \Longleftrightarrow \widetilde{r} \in \Der_{k[[\bfs]]_\Delta}(A[[\bfs]]_\Delta,F[[\bfs]]_\Delta),
$$
and so the map (\ref{eq:tilde-map}) for $E=A$ induces an isomorphism of $A[[\bfs]]_\Delta$-modules
$$ \Der_k(A,F)[[\bfs]]_\Delta \xrightarrow{\sim} \Der_{k[[\bfs]]_\Delta}(A[[\bfs]]_\Delta,F[[\bfs]]_\Delta).
$$
\end{enumerate}
\end{lemma}

\begin{proof} Parts 1) and 2) are proven in \cite[Lemma 3]{nar_subst_2018}. For part 3), let us write $r= \sum_\beta r_\beta \bfs^\beta$.
\medskip

\noindent ($\Rightarrow$) For all $a = \sum_\alpha, b=\sum_\alpha \in A[[\bfs]]_\Delta$ we have:
\begin{eqnarray*}
&\displaystyle
\widetilde{r}(a b) = \cdots = \sum_{\scriptscriptstyle \alpha \in \Delta} \left( \sum_{\scriptscriptstyle \beta+ \gamma + \delta= \alpha} r_\beta (a_\gamma b_\delta) \right) \bfs^\alpha = 
&
\\
&\displaystyle
\sum_{\scriptscriptstyle \alpha \in \Delta} \left(  \sum_{\scriptscriptstyle \beta+\gamma + \delta= \alpha}  (b_\delta r_\beta (a_\gamma ) + a_\gamma r_\beta (b_\delta)) \right) \bfs^\alpha=
\cdots = b\, \widetilde{r}(a) + a\,  \widetilde{r}(b).
\end{eqnarray*}
\noindent ($\Leftarrow$) For all $a,b\in A$ we have:
$$ \sum_{\scriptscriptstyle \beta \in \Delta} r_\beta (a b) \bfs^\beta = \widetilde{r}(ab) = b\, \widetilde{r}(a) + a\,  \widetilde{r}(b) = \cdots= \sum_{\scriptscriptstyle \beta \in \Delta} (b\, r_\beta (a) + a\, r_\beta(b)) \bfs^\beta 
$$
and so $r_\beta \in \Der_k(A,F)$ for all $\beta\in\Delta$.
\end{proof}

Let us call $R = \End_k(E)$. As a consequence of the above lemma, the composition of the maps
\begin{equation}  \label{eq:comple_formal_1}
 R[[\bfs]]_\Delta \xrightarrow{r \mapsto \widetilde{r}} \End_{k[[\bfs]]_\Delta}(E[[\bfs]]_\Delta) \xrightarrow{f \mapsto f|_E} \Hom_k(E,E[[\bfs]]_\Delta)
\end{equation}
is an isomorphism of $(A[[\bfs]]_\Delta;A)$-bimodules, and so $\Hom_k(E,E[[\bfs]]_\Delta)$ inherits a natural structure of $k[[\bfs]]_\Delta$-algebra over $A[[\bfs]]_\Delta$. Namely, if $g,h:E \to E[[\bfs]]_\Delta$ are $k$-linear maps
with
$$ g(e)=\sum_{\scriptscriptstyle \alpha\in \Delta} g_\alpha(e)\bfs^\alpha,\ h(e)=\sum_{\scriptscriptstyle \alpha\in \Delta} h_\alpha(e)\bfs^\alpha,\quad \forall e\in E,\quad g_\alpha, h_\alpha \in \Hom_k(E,E),
$$
then the product $h g \in \Hom_k(E,E[[\bfs]]_\Delta)$ is given by
\begin{equation} \label{eq:product-HomEE[[s]]}
 (hg)(e) = \sum_{\scriptscriptstyle \alpha\in \Delta} \left( \sum_{\scriptscriptstyle \beta + \gamma = \alpha} (h_\beta \pcirc g_\gamma)(e) \right) \bfs^\alpha.
\end{equation}

\begin{defi} \label{defi:tensor-of-maps} Let $p,q\geq 0$, $\bfs=\{s_1,\dots,s_p\},\bft=\{t_1,\dots,t_q\}$ disjoint sets of variables and $\Delta\subset\N^p, \nabla\subset\N^q$ non-empty co-ideals. 
For each $f\in\End_{k[[\bfs]]_\Delta}(E[[\bfs]]_\Delta)$ and each $g\in\End_{k[[\bft]]_\nabla}(E[[\bft]]_\nabla)$, with
$$ f(e)=\sum_{\scriptscriptstyle \alpha \in \Delta} f_\alpha(e) \bfs^\alpha, \quad g(e)=\sum_{\scriptscriptstyle \beta \in \nabla} g_\beta(e) \bft^\beta\quad \forall e\in E,
$$
we define $f\boxtimes g\in \End_{k[[\bfs\sqcup \bft]]_{\Delta\times\nabla}}(E[[\bfs\sqcup \bft]]_{\Delta\times\nabla})$ as $f\boxtimes g := h^e$, with:
$$ h(x) := \sum_{\scriptscriptstyle (\alpha,\beta) \in \Delta\times\nabla} (f_\alpha \pcirc g_\beta)(x) \bfs^\alpha \bft^\beta\quad \forall x\in E.
$$
\end{defi}

The proof of the following lemma is clear and it is left to the reader.

\begin{lemma} \label{lemma:tilde-otimes}
With the above hypotheses, for each $r\in R[[\bfs]]_\Delta, r'\in R[[\bft]]_\nabla$, we have
$ \widetilde{r\boxtimes r'} = \widetilde{r}\boxtimes \widetilde{r'}$ (see Definition \ref{defi:external-x}).
\end{lemma}

\begin{notacion} \label{notacion:pcirc}
We denote :
$$
\Hom_k^\pcirc(E,E[[\bfs]]_\Delta) := \left\{ f \in \Hom_k(E,E[[\bfs]]_\Delta), f(e) \equiv e\!\!\!\!\mod \left(\fn_0\right)_E \forall e\in E \right\}, 
$$
$$
\Aut_{k[[\bfs]]_\Delta}^\pcirc(E[[\bfs]]_\Delta) 
:= \left\{    
f \in \Aut_{k[[\bfs]]_\Delta}(E[[\bfs]]_\Delta), f(e) \equiv e_0\, \text{\rm mod} \left(\fn_0\right)_E \forall e\in E[[\bfs]]_\Delta  \right\}.
$$
Let us notice that a $f \in \Hom_k(E,E[[\bfs]]_\Delta)$, given by $f(e) = \sum_{\alpha\in \Delta} f_\alpha(e) \bfs^\alpha$, belongs to $\Hom_k^\pcirc(E,E[[\bfs]]_\Delta)$ if and only if $f_0=\Id_E$.
\end{notacion}

The isomorphism in (\ref{eq:comple_formal_1}) gives rise to a group isomorphism
\begin{equation} \label{eq:U-iso-pcirc}
 r\in \U^p(\End_k(E);\Delta) \stackrel{\sim}{\longmapsto} \widetilde{r} \in \Aut_{k[[\bfs]]_\Delta}^\pcirc(E[[\bfs]]_\Delta) 
\end{equation}
and to a bijection 
\begin{equation} \label{eq:Aut-iso-pcirc} 
f\in \Aut_{k[[\bfs]]_\Delta}^\pcirc(E[[\bfs]]_\Delta) \stackrel{\sim}{\longmapsto} f|_E \in \Hom_k^\pcirc(E,E[[\bfs]]_\Delta).
\end{equation}
So, $\Hom_k^\pcirc(E,E[[\bfs]]_\Delta)$ is naturally a group with the product described in (\ref{eq:product-HomEE[[s]]}).

\subsection{Substitution maps}

In this section we give a summary of sections 2 and 3 of \cite{nar_subst_2018}. Let $k$ be a commutative ring, $A$ a commutative $k$-algebra, $\bfs=\{s_1,\dots,s_p\},\bft=\{t_1,\dots,t_q\}$ two sets of variables and
$\Delta\subset \N^p, \nabla\subset \N^q$ non-empty co-ideals.
\medskip

\begin{defi}  \label{def:substitution_maps} 
An $A$-algebra map $\varphi:A[[\bfs]]_\Delta \xrightarrow{} A[[\bft]]_\nabla$
will be called a {\em substitution map} whenever $\ord (\varphi(s_i)) \geq 1$ for all $i=1,\dots, p$.
A such map is continuous and uniquely determined by the family $c=\{\varphi(t_i), i=1,\dots, p\}$.
\medskip

If $\varphi:A[[\bfs]]_\Delta \xrightarrow{} A[[\bft]]_\nabla$ is a substitution map, its {\em order} is defined as 
$$\ord (\varphi) : = \min \{\ord (\varphi(s_i))\ |\ i=1,\dots, p\} \geq 1.$$
The set of substitution maps $A[[\bfs]]_\Delta \xrightarrow{} A[[\bft]]_\nabla$ will be denoted by $\Sub_A(p,q;\Delta,\nabla)$.
The {\em trivial} substitution map  $A[[\bfs]]_\Delta \xrightarrow{} A[[\bft]]_\nabla $ is the one sending any $s_i$ to $0$ ($\ord(0)=\infty$). It will be denoted by $\mathbf{0}$. 
\end{defi}

The composition of substitution maps is obviously a substitution map. Any substitution map $\varphi:A[[\bfs]]_\Delta \xrightarrow{} A[[\bft]]_\nabla$ determines and is determined by a family 
$$\left\{{\bf C}_e(\varphi,\alpha), e\in \nabla, \alpha\in\Delta, |\alpha|\leq |e|  \right\}\subset A,\quad \text{with\ }\ {\bf C}_0(\varphi,0)=1,
$$
such that:
\begin{equation} \label{eq:exp-substi}
\varphi\left( \sum_{\scriptscriptstyle\alpha\in\Delta} a_\alpha \bfs^\alpha \right) =
\sum_{\scriptscriptstyle e\in \nabla}
\left(  \sum_{\substack{\scriptscriptstyle \alpha\in\Delta\\ \scriptscriptstyle  |\alpha|\leq |e| }} {\bf C}_e(\varphi,\alpha) a_\alpha \right) \bft^e.
\end{equation}
In section 3, 2. of \cite{nar_subst_2018} the reader can find the explicit expression of the ${\bf C}_e(\varphi,\alpha)$ in terms of the $\varphi(s_i)$. 
The following lemma is clear.

\begin{lemma} \label{lemma:truncations-are-substitutions}
If $\Delta \subset \Delta'\subset \N^p$ are non-empty co-ideals, the truncation 
$\tau_{\Delta' \Delta} :  A[[\bfs]]_{\Delta'} \to A[[\bfs]]_{\Delta}$ is clearly a substitution map and 
${\bf C}_\beta\left(\tau_{\Delta' \Delta},\alpha\right) = \delta_{\alpha \beta}$ for all $\alpha\in\Delta$ and for all $\beta\in\Delta'$ with $|\alpha|\leq |\beta|$.
\end{lemma}

\begin{defi} We say that a substitution map $\varphi: A[[\bfs]]_\Delta \xrightarrow{} A[[\bft]]_\nabla $ has {\em constant coefficients} 
if $\varphi(s_i) \in k[[\bft]]_\nabla$ for all $i=1,\dots,p$.
This is equivalent to saying that
 ${\bf C}_e(\varphi,\alpha)\in k$ for all $e\in\nabla$ and for all $\alpha\in \Delta$ with $|\alpha|\leq |e|$. Substitution maps with constant coefficients are induced by substitution maps $k[[\bfs]]_\Delta \xrightarrow{} k[[\bft]]_\nabla$.\\

We say that a substitution map $\varphi: A[[\bfs]]_\Delta \xrightarrow{} A[[\bft]]_\nabla $ is {\em combinatorial} if $\varphi(s_i) \in \bft$ for all $i=1,\dots,p$. A combinatorial substitution map has constant coefficients and is determined by (and determines) a map $\bfs \to \bft$. If $\iota : \bfs \to \bft$ is such a map, we will also denote by $\iota: A[[\bfs]]_\Delta \xrightarrow{} A[[\bft]]_\nabla $ the corresponding substitution map,  for any non-empty co-ideal $\nabla\subset 
\iota_*(\Delta):= \{\beta \in \N^q\ |\ \beta \pcirc \iota \in \Delta\}$ (here multi-indexes in $\N^q$ or $\N^p$ are considered as maps $\bft \to \N$ or $\bfs \to \N$ respectively).
\end{defi}

\begin{defi} \label{defi:tensor-prod-of-substi}
Let
$\bfu=\{u_1,\dots,u_m\},\bfv=\{v_1,\dots,v_n\}$ be another sets of variables.
The {\em tensor product} \index{tensor product (of substitution maps)} of two substitution maps $\varphi:A[[\bfs]]_\nabla \to A[[\bft]]_\Delta$, $\psi:A[[\bfu]]_{\nabla'} \to A[[\bfv]]_{\Delta'}$ is the unique substitution map
$$ \varphi \otimes \psi: A[[\bfs\sqcup \bfu]]_{\nabla \times \nabla'} \longrightarrow A[[\bft\sqcup \bfv]]_{\Delta \times \Delta'}
$$
making commutative the following diagram:
$$
\begin{tikzcd}
A[[\bfs]]_\nabla \ar[d,"\varphi"] \ar[r] & A[[\bfs\sqcup \bfu]]_{\nabla \times \nabla'} \ar[d,"\varphi \otimes \psi"] & A[[\bfu]]_{\nabla'} \ar[l] \ar[d,"\psi"]\\ 
A[[\bft]]_\Delta \ar[r] & A[[\bft\sqcup \bfv]]_{\Delta \times \Delta'}   & \ar[l] A[[\bfv]]_{\Delta'},
\end{tikzcd}
$$
where the horizontal arrows are the combinatorial substitution maps induced by the inclusions $\bfs, \bfu \hookrightarrow \bfs\sqcup \bfu$, $\bft, \bfv \hookrightarrow \bft\sqcup \bfv$\footnote{Let us notice that there are canonical continuous isomorphisms of $A$-algebras $A[[\bfs\sqcup \bfu]]_{\nabla \times \nabla'} \simeq A[[\bfs]]_\nabla \widehat{\otimes}_A A[[\bfu]]_{\nabla'}$,  
$A[[\bft\sqcup \bfv]]_{\Delta \times \Delta'} \simeq A[[\bft]]_\Delta \widehat{\otimes}_A A[[\bfv]]_{\Delta'}$.}.
\end{defi}

For all $(\alpha,\beta)\in \nabla \times \nabla' \subset  \N^p \times \N^m \equiv \N^{p+m}$ we have
$$ (\varphi \otimes \psi)(\bfs^\alpha \bfu^\beta) =  \varphi (\bfs^\alpha) \psi(\bfu^\beta) =
\cdots = 
\sum_{\substack{\scriptscriptstyle  e\in \Delta, f\in \Delta'\\ \scriptscriptstyle  |e|\geq |\alpha| \\ \scriptscriptstyle |f|\geq |\beta|  }} 
{\bf C}_e(\varphi,\alpha) {\bf C}_f(\psi,\beta) \bft^e \bfv^f
$$
and so, for all $(e,f)\in \Delta \times \Delta'$ and all $(\alpha,\beta) \in \nabla \times \nabla'$
 with $ |e|+|f|= |(e,f)| \geq |(\alpha,\beta)| = |\alpha|+|\beta|$ we have
$$
{\bf C}_{(e,f)}(\varphi\otimes \psi,(\alpha,\beta)) = \left\{ 
\begin{array}{ll}
{\bf C}_e(\varphi,\alpha) {\bf C}_f(\psi,\beta) & \text{if\ }\ |\alpha|\leq |e|\ \text{and\ }\ |\beta|\leq |f|, \\
0 & \text{otherwise}.
\end{array}  \right.
$$

\begin{prop} \label{prop:init}
Let $\varphi \in \Sub_A(p,q;\Delta,\nabla)$ be a substitution map and 
$\displaystyle \varphi(s_i) = \sum_{\scriptscriptstyle |\beta|>0 } c^i_\beta \bft^\beta \in   A[[\bft]]_\nabla$, $ i=1,\dots,p$. 
Let us denote
$\displaystyle  \init \varphi(s_i) :=  \sum_{\scriptscriptstyle |\beta|=1} c^i_\beta \bft^\beta \in   A[[\bft]]_\nabla$, $i=1,\dots,p
$
and $\psi: A[[\bfs]] \to  A[[\bft]]_\nabla$ the substitution map determined by $\psi (s_i) =  \init \varphi(s_i)$ for $i=1,\dots,p$.
Then, $\psi(\Delta_A) = \{0\}$ and there is a unique induced substitution map  $\init \varphi: A[[\bfs]]_\Delta \to  A[[\bft]]_\nabla$ satisfying $(\init \varphi)(s_i) = \init \varphi(s_i)$, $i=1,\dots,p$.
\end{prop}

\begin{proof} First, let us prove that $\supp \psi(\bfs^\alpha) \subset \supp \varphi(\bfs^\alpha)$ for all $\alpha \in \N^p$. 
Since the $\init \varphi(s_i)$ are homogeneous of degree $1$, we deduce that $\psi(\bfs^\alpha)$ is homogeneous of degree $|\alpha|$ for all $\alpha\in \N^p$. So, if $e\in \supp \psi(\bfs^\alpha)$, then $|e|=|\alpha|$ and ${\bf C}_e(\psi,\alpha) \neq 0$, but from \cite[Lemma 6, (2)]{nar_subst_2018} we have ${\bf C}_e(\varphi,\alpha) = {\bf C}_e(\psi,\alpha) \neq 0$ and we deduce $e \in \supp \varphi(\bfs^\alpha)$.
\medskip

The substitution map $\overline{\varphi}: A[[\bfs]] \to  A[[\bft]]_\nabla$  obtained by composing $\varphi$ with the projection $A[[\bfs]] \to A[[\bfs]]_\Delta$ satisfies $ \overline{\varphi}(\Delta_A) =\{0\}$, i.e. for all $\alpha \notin \Delta$ we have $\overline{\varphi}(\bfs^\alpha) = 0$, and so $\psi(\bfs^\alpha) = 0$. We deduce that $\psi(\Delta_A)=\{0\}$ and so it induces a unique substitution map $\init \varphi: A[[\bfs]]_\Delta \to  A[[\bft]]_\nabla$ as required.
\end{proof}

Let us notice that, with the notations of Proposition \ref{prop:init}, we have  $\ord \varphi > 1$ if and only if $\init \varphi =\mathbf{0}$.
\bigskip

\numero  \label{sec:action-substi}
Let
$M$ be an $(A;A)$-bimodule.
Any substitution map $\varphi: A[[\bfs]]_\Delta \to A[[\bft]]_\nabla$ induces 
$(A;A)$-linear maps:  
$$\varphi_M := \varphi \widehat{\otimes} \Id_M : M[[\bfs]]_\Delta \equiv A[[\bfs]]_\Delta \widehat{\otimes}_A M \longrightarrow M[[\bft]]_\nabla \equiv A[[\bft]]_\nabla \widehat{\otimes}_A M
$$ 
and 
$$\sideset{_M}{}\opvarphi := \Id_M \widehat{\otimes} \varphi  : M[[\bfs]]_\Delta \equiv M \widehat{\otimes}_A A[[\bfs]]_\Delta \longrightarrow M[[\bft]]_\nabla \equiv M \widehat{\otimes}_A A[[\bft]]_\nabla.
$$
We have:
\begin{eqnarray*}
&\displaystyle \varphi_M \left(\sum_{ \scriptscriptstyle \alpha\in\Delta} m_\alpha \bfs^\alpha \right) = \sum_{\scriptscriptstyle \alpha\in\Delta} \varphi( \bfs^\alpha ) m_\alpha = \sum_{ \scriptscriptstyle e\in \nabla}   \left( \sum_{\substack{\scriptscriptstyle \alpha\in\Delta\\ \scriptscriptstyle  |\alpha|\leq |e| } }
{\bf C}_e(\varphi,\alpha) m_\alpha \right) \bft^e,
&\\
&\displaystyle 
\sideset{_M}{}\opvarphi \left(\sum_{ \scriptscriptstyle \alpha\in\Delta} m_\alpha \bfs^\alpha \right) = \sum_{\scriptscriptstyle \alpha\in\Delta}  m_\alpha \varphi( \bfs^\alpha ) = \sum_{ \scriptscriptstyle e\in\nabla }   \left( \sum_{ \substack{\scriptscriptstyle \alpha\in\Delta\\ \scriptscriptstyle  |\alpha|\leq |e| } } 
m_\alpha {\bf C}_e(\varphi,\alpha)  \right) \bft^e&
\end{eqnarray*}
for all $m\in M[[\bfs]]_\Delta$. 
If $M$ is a trivial bimodule, then $\varphi_M = \sideset{_M}{}\opvarphi$. 
If $\varphi': A[[\bft]]_\nabla \to A[[\bfu]]_\Omega$ is another substitution map and $\opvarphi'' = \opvarphi\pcirc \opvarphi'$, we have
$\opvarphi''_M = \varphi_M \pcirc \varphi'_M$, $\sideset{_M}{}\opvarphi'' = \sideset{_M}{}\opvarphi \pcirc \sideset{_M}{}\opvarphi'$.
\bigskip

For all $m\in M[[\bfs]]_\Delta$ and all $a \in A[[\bfs]]_\nabla$, we have
$$ \varphi_M (a m) = \varphi(a) \varphi_M(m),\ \sideset{_M}{}\opvarphi (m a) =  \sideset{_M}{}\opvarphi (m) \varphi(a),
$$
i.e. $\varphi_M$ is $(\varphi;A)$-linear and $\sideset{_M}{}\opvarphi$ is $(A;\varphi)$-linear. Moreover, $\varphi_M$ and $\sideset{_M}{}\opvarphi$ are compatible with the augmentations, i.e.
\begin{equation} \label{eq:compat-augment}
\varphi_M(m) \equiv m_0\!\!\!\!\mod \left(\fn_0\right)_M/\nabla_M,\ 
\sideset{_M}{}\opvarphi(m) \equiv m_0\!\!\!\!\mod \left(\fn_0\right)_M/\nabla_M,\ m\in M[[\bfs]]_\Delta.
\end{equation}
If $\varphi$ is the trivial substitution map (i.e. $\varphi(s_i)=0$ for all $s_i\in\bfs$), then $\varphi_M : M[[\bfs]]_\Delta \to M[[\bft]]_\nabla$ and $\sideset{_M}{}\opvarphi: M[[\bfs]]_\Delta \to M[[\bft]]_\nabla$ are also trivial, i.e.
$ \varphi_M (m) = \sideset{_M}{}\opvarphi
(m) = m_0$, for all $m\in M[[\bfs]]_\nabla$.
\bigskip

\numero \label{nume:def-sbullet}
The above constructions apply in particular to the case of any $k$-algebra $R$ over $A$, for which we have two induced continuous maps:
$\varphi_R= \varphi \widehat{\otimes} \Id_R : R[[\bfs]]_\Delta \to R[[\bft]]_\nabla$, which is $(A;R)$-linear, and  $\sideset{_R}{}\opvarphi= \Id_R \widehat{\otimes} \varphi  : R[[\bfs]]_\Delta \to R[[\bft]]_\nabla$, which is $(R;A)$-linear. 
For $r\in R[[\bfs]]_\Delta$ we will denote
$ \varphi \sbullet r := \varphi_R(r)$, $r \sbullet \varphi := \sideset{_R}{}\opvarphi(r)$. 
Explicitly, if $r=\sum_\alpha
r_\alpha \bfs^\alpha$ with $\alpha\in\Delta$, then:
\begin{equation} \label{eq:explicit-bullet}
\varphi \sbullet r
 = 
\sum_{\scriptscriptstyle e \in\nabla}
\left( \sum_{\substack{\scriptscriptstyle \alpha\in\Delta\\ \scriptscriptstyle  |\alpha|\leq |e|}} {\bf C}_e(\varphi,\alpha) r_\alpha \right) \bft^e,\quad
 r \sbullet \varphi 
 = 
\sum_{\scriptscriptstyle e \in\nabla}
\left( \sum_{\substack{\scriptscriptstyle \alpha\in\Delta\\ \scriptscriptstyle  |\alpha|\leq |e|}} r_\alpha {\bf C}_e(\varphi,\alpha)  \right) \bft^e.
\end{equation}
From (\ref{eq:compat-augment}), we deduce that:
$$\varphi \sbullet \U^p(R;\Delta) \subset \U^q(R;\nabla),\quad \U^p(R;\Delta)\sbullet \varphi \subset \U^q(R;\nabla),
$$ 
and if $R$ is a filtered $k$-algebra over $A$, then  $\varphi \sbullet \Ufil^p(R;\Delta) \subset \Ufil^q(R;\nabla)$ and $\Ufil^p(R;\Delta)\sbullet \varphi \subset \Ufil^q(R;\nabla)$.
We also have $\varphi \sbullet 1 = 1 \sbullet \varphi = 1$. 
\medskip

\noindent
If $\varphi$ is a substitution map with \underline{constant coefficients}, then $\varphi_R = \sideset{_R}{}\opvarphi$ is a ring homomorphism over $\varphi$. 
In particular, $\varphi \sbullet r = r \sbullet \varphi$ and $\varphi \sbullet (rr') = (\varphi \sbullet r) (\varphi \sbullet r')$.
\medskip

\noindent 
If $\varphi = \mathbf{0}: A[[\bfs]]_\Delta \to A[[\bft]]_\nabla$ is the trivial substitution map, then 
$\mathbf{0} \sbullet r = r\sbullet \mathbf{0} = r_0$ for all $r\in R[[\bfs]]_\Delta$. In particular,
$\mathbf{0} \sbullet r = r \sbullet \mathbf{0} = 1$ for all $r\in \U^p(R;\Delta)$.
\medskip

\noindent 
If $\bfu=\{u_1,\dots,u_r\}$ is another set of variables, $\Omega \subset \N^r$ is a non-empty co-ideal and
$\psi:R[[\bft]]_\nabla \to R[[\bfu]]_\Omega$ is another substitution map, one has:
$$ \psi \sbullet (\varphi \sbullet r) = (\psi \pcirc \varphi) \sbullet r,\quad
(r \sbullet \varphi) \sbullet \psi  =  r \sbullet (\psi \pcirc \varphi).
$$
Since $\left(R[[\bfs]]_\Delta\right)^{\text{opp}} = R^{\text{opp}}[[\bfs]]_\Delta$, for any substitution map $\varphi: A[[\bfs]]_\Delta \to A[[\bft]]_\nabla$  we have $\left( \varphi_R \right)^{\text{opp}} = \sideset{_{R^{\text{opp}}}}{}\opvarphi$ and 
$\left( \sideset{_R}{}\opvarphi \right)^{\text{opp}} = \varphi_{R^{\text{opp}}}$. 
\medskip

The proof of the following lemma is straightforward and it is left to the reader.

\begin{lemma} \label{lemma:phi-linearity-of-phi_R}
If $\varphi: A[[\bfs]]_\Delta \to A[[\bft]]_\nabla$ is a substitution map, then:
\begin{itemize}
\item[(i)] $\varphi_R$ is left $\varphi$-linear, i.e. $\varphi_R(ar) = \varphi(a) \varphi_R(r)$ for all $a\in A[[\bfs]]_\Delta$ and for all $r\in R[[\bfs]]_\Delta$.
\item[(ii)] $\sideset{_R}{}\opvarphi$ is right $\varphi$-linear, i.e. $\sideset{_R}{}\opvarphi(ra)= \sideset{_R}{}\opvarphi(r) \varphi(a)$  for all $a\in A[[\bfs]]_\Delta$ and for all $r\in R[[\bfs]]_\Delta$.
\end{itemize}
\end{lemma}

For each substitution map $\varphi: A[[\bfs]]_\Delta \to A[[\bft]]_\nabla$ we define the $(A;A)$-linear map:
\begin{equation} \label{eq:varphi_star}
 \varphi_*: f\in \Hom_k(A,A[[\bfs]]_\Delta) \longmapsto \varphi_*(f)= \varphi \pcirc f \in \Hom_k(A,A[[\bft]]_\nabla)
\end{equation}
which induces another one $\overline{\varphi_*}: \End_{k[[\bfs]]_\Delta}(A[[\bfs]]_\Delta) \longrightarrow
\End_{k[[\bft]]_\nabla}(A[[\bft]]_\nabla)$ given by:
\begin{equation} \label{eq:overline_varphi_star}
\overline{\varphi_*}(f):=
\left(\varphi_*\left(f|_A\right)\right)^e =
\left(\varphi \pcirc f|_A\right)^e\quad \forall f\in \End_{k[[\bfs]]_\Delta}(A[[\bfs]]_\Delta).
\end{equation}
More generally, for any left $A$-modules $E,F$ we have $(A;A)$-linear maps:
\begin{eqnarray*}
&\displaystyle
 (\varphi_{F})_*: f\in \Hom_k(E,F[[\bfs]]_\Delta) \longmapsto (\varphi_{F})_*(f)= \varphi_F \pcirc f \in \Hom_k(E,F[[\bft]]_\nabla),
&\\
&\displaystyle
\overline{(\varphi_F)_*}:  \Hom_{k[[\bfs]]_\Delta}(E[[\bfs]]_\Delta,F[[\bfs]]_\Delta) \longrightarrow  \Hom_{k[[\bft]]_\nabla}(E[[\bft]]_\nabla,F[[\bft]]_\nabla),
&\\
&\displaystyle
 \overline{(\varphi_F)_*}(f):=\left(\varphi_F \pcirc f|_E\right)^e.
\end{eqnarray*}
Let us consider the $(A;A)$-bimodule $M=\Hom_k(E,F)$. For each $m\in M[[\bfs]]_\Delta$ and for each $e\in E$ we have
$ \widetilde{\varphi_M(m)}(e) = \varphi_F\left(\widetilde{m}(e) \right)$, i.e.
\begin{equation}   \label{eq:tilde-varphi-M-varphi-F}
\widetilde{\varphi_M(m)}|_E = \varphi_F \pcirc \left(\widetilde{m}|_E\right),
\end{equation}
or more graphically, the following diagram is commutative (see (\ref{eq:comple_formal_1})):
\begin{equation} \label{eq:commut-CD-varphi-bullet}
\begin{tikzcd}
M[[\bfs]]_\Delta 
\ar[d,"\varphi_M"'] \ar[r,"\sim", "m\mapsto \widetilde{m}"'] & 
\Hom_{k[[\bfs]]_\Delta}(E[[\bfs]]_\Delta,F[[\bfs]]_\Delta) \ar[r,"\sim", "\text{restr.}"']
\ar[d,"\overline{(\varphi_F)_*}"] & \Hom_k(E,F[[\bfs]]_\Delta) \ar[d,"(\varphi_F)_*"'] \\
M[[\bft]]_\nabla \ar[r,"\sim", "m\mapsto \widetilde{m}"'] &
\Hom_{k[[\bft]]_\nabla}(E[[\bft]]_\nabla,F[[\bft]]_\nabla) \ar[r,"\sim", "\text{restr.}"'] &
\Hom_k(E,F[[\bft]]_\nabla).
\end{tikzcd}
\end{equation}
In order to simplify notations, we will also write:
$$ \varphi \sbullet f := \overline{(\varphi_F)_*}(f)\quad \forall f \in \Hom_{k[[\bfs]]_\Delta}(E[[\bfs]]_\Delta,F[[\bfs]]_\Delta),
$$
and so we have
$ \widetilde{\varphi \sbullet m} = \varphi \sbullet \widetilde{m}$ for all $m\in M[[\bfs]]_\Delta$. 
Let us notice that $(\varphi \sbullet f)(e) = (\varphi_F \pcirc f)(e)$ for all $e\in E$, i.e.  
\begin{equation} \label{eq:atencion-bullet}
\framebox{$(\varphi \sbullet f)|_E = (\varphi_F \pcirc f)|_E = \varphi_F \pcirc \left(f|_E\right)$, but in general $\varphi \sbullet f \neq  \varphi_F \pcirc f$.}
\end{equation}
If $\varphi = \mathbf{0}$ is the trivial substitution map, then for each $f=\sum_\alpha f_\alpha \bfs^\alpha \in \Hom_k(E,E[[\bfs]]_\Delta)$ (resp. $f=\sum_\alpha f_\alpha \bfs^\alpha \in\End_k(E)[[\bfs]]_\Delta \equiv \End_{k[[\bfs]]_\Delta}(E[[\bfs]]_\Delta)$), we have $\mathbf{0} \sbullet f = f\sbullet \mathbf{0} = f_0 \in \End_k(E) \subset \Hom_k(E,E[[\bfs]]_\Delta)$ (resp. $\mathbf{0} \sbullet f = f\sbullet \mathbf{0} = f_0^e = \overline{f_0} \in \End_{k[[\bfs]]_\Delta}(E[[\bfs]]_\Delta)$).
\bigskip

If $\varphi: A[[\bfs]]_\Delta \to A[[\bft]]_\nabla$ is a substitution map, we have:
\begin{eqnarray*}
&\displaystyle
\varphi \sbullet (af) = \varphi(a) \left( \varphi \sbullet f\right),\   (fa) \sbullet \varphi  =  \left(f \sbullet \varphi \right) \varphi(a)
\end{eqnarray*}
for all $a\in A[[\bfs]]_\Delta$ and for all  $f\in \Hom_k(E,E[[\bfs]]_\Delta)$ (or $f\in \End_{k[[\bfs]]_\Delta}(E[[\bfs]]_\Delta)$).
\medskip

Moreover:
\begin{eqnarray*}
& \displaystyle (\varphi_E)_*(\Hom_k^\pcirc(E,M[[\bfs]]_\Delta)) \subset \Hom_k^\pcirc(E,E[[\bft]]_\nabla),
&
\\
& \displaystyle 
 \varphi \sbullet  \left( \Aut_{k[[\bfs]]_\Delta}^\pcirc(E[[\bfs]]_\Delta) \right) \subset
\Aut_{k[[\bft]]_\nabla}^\pcirc(E[[\bft]]_\nabla)&
\end{eqnarray*}
and so we have a commutative diagram:
\begin{equation}   \label{eq:diag-funda}
\begin{tikzcd}
\U^p(R;\Delta) \ar[r,"\sim", "r\mapsto \widetilde{r}"'] \ar[d,"\varphi \sbullet (-)"']  &   
 \Aut_{k[[\bfs]]_\Delta}^\pcirc(E[[\bfs]]_\Delta) \ar[d,"\varphi \sbullet (-)"] \ar[r,"\sim", "\text{restr.}"']  &
\Hom_k^\pcirc(E,E[[\bfs]]_\Delta)  \ar[d,"(\varphi_{E})_*"] \\
\U^q(R;\nabla) \ar[r,"\sim", "r\mapsto \widetilde{r}"'] & \Aut_{k[[\bft]]_\nabla}^\pcirc(E[[\bft]]_\nabla)
\ar[r,"\sim", "\text{restr.}"'] & \Hom_k(E,F[[\bft]]_\nabla).
\end{tikzcd}
\end{equation}
\bigskip

\numero \label{nume:properties-external-x} 
Let us denote $\iota: A[[\bfs]]_\Delta \to A[[\bfs \sqcup \bft]]_{\Delta \times \nabla}$, $\kappa: A[[\bft]]_\nabla \to A[[\bfs \sqcup \bft]]_{\Delta \times \nabla}$ the combinatorial substitution maps given by the inclusions $\bfs \hookrightarrow \bfs \sqcup \bft$, $\bft \hookrightarrow \bfs \sqcup \bft$.

Let us notice that for $r\in R[[\bfs]]_\Delta$ and $r'\in R[[\bft]]_\nabla$, we have (see Definition \ref{defi:external-x})
$ r\boxtimes r' = (\iota\sbullet r) (\kappa\sbullet r') \in R[[\bfs \sqcup \bft]]_{\Delta \times \nabla}$. 
If $\Delta'\subset\Delta\subset\N^p$, $\nabla'\subset\nabla\subset \N^q$ are non-empty co-ideals, we have 
$$\tau_{\Delta \times \nabla,\Delta' \times \nabla'}(r\boxtimes r') = \tau_{\Delta,\Delta'}(r) \boxtimes \tau_{\nabla,\nabla'}(r').
$$
If we denote by $\Sigma: R[[\bfs \sqcup \bfs]]_{\nabla\times\nabla}  \rightarrow R[[\bfs]]_\nabla$ the combinatorial substitution map given by the co-diagonal map $\bfs \sqcup \bfs \to \bfs$, it is clear that for each $r,r'\in R[[\bfs]]_\nabla$ we have
\begin{equation} \label{eq:prod-in-terms-x}
r r' = \Sigma\sbullet (r \boxtimes r').
\end{equation}
If $\varphi:A[[\bfs]]_\Delta \to A[[\bfu]]_\Omega$ and $\psi:A[[\bft]]_\nabla \to A[[\bfv]]_{\Omega'}$ are substitution maps, we have new substitution maps 
$\varphi \otimes \Id:  A[[\bfs \sqcup \bft]]_{\Delta\times\nabla} \to A[[\bfu \sqcup \bft]]_{\Omega\times\nabla}$ and $\Id \otimes \psi: A[[\bfs \sqcup \bft]]_{\Delta\times\nabla} \to A[[\bfs \sqcup \bfv]]_{\Delta\times\Omega'}$ (see Definition \ref{defi:tensor-prod-of-substi}) taking part in the following commutative diagrams of $(A;A)$-bimodules:
$$
\begin{tikzcd}
R[[\bfs]]_\Delta \otimes_R R[[\bft]]_\nabla \ar[r,"\varphi_R \otimes\Id"] \ar[d,"\text{can.}"'] &
 R[[\bfu]]_\Omega \otimes_R R[[\bft]]_\nabla \ar[d,"\text{can.}"]\\
R[[\bfs \sqcup \bft]]_{\Delta\times\nabla}  \ar[r,"\left(\varphi\otimes\Id\right)_R"] &
R[[\bfu \sqcup \bft]]_{\Omega\times\nabla}
\end{tikzcd}
$$
and
$$
\begin{tikzcd}
R[[\bfs]]_\Delta \otimes_R R[[\bft]]_\nabla \ar[r,"\Id\otimes\psi"] \ar[d,"\text{can.}"'] &
 R[[\bfs]]_\Delta \otimes_R R[[\bfv]]_{\Omega'} \ar[d,"\text{can.}"]\\
R[[\bfs \sqcup \bft]]_{\Delta\times\nabla}  \ar[r,"\left(\Id\otimes\varphi\right)_R"] &
R[[\bfs \sqcup \bfv]]_{\Delta\times\Omega'}.
\end{tikzcd}
$$
We deduce that $(\varphi\sbullet r) \boxtimes r'=
(\varphi\otimes \Id)\sbullet (r\boxtimes r')$ and $r\boxtimes (r'\sbullet \psi)= (r\boxtimes r')\sbullet (\Id \otimes \psi)$.

\begin{prop} \label{prop:sigma-varphi-bullet}
Let $R$ be a filtered $k$-algebra over $A$ and $\varphi \in \Sub_A(p,q;\Delta,\nabla)$ a substitution map. The following diagram is commutative:
$$
\begin{tikzcd}
\Ufil^p(R;\Delta)
\ar[d,"\varphi\sbullet (-)"'] \ar[r,"\bupsigma"] & 
\Ugr^p(\gr R;\Delta)  \ar[d, "(\init \varphi)\sbullet (-)"]\\
\Ufil^q(R;\nabla) \ar[r,"\bupsigma"] & \Ugr^q(\gr R;\nabla),
\end{tikzcd}
$$
where $\init \varphi$ has been defined in Proposition \ref{prop:init}. 
\end{prop}

\begin{proof} For any element $r=\sum_{\scriptscriptstyle \alpha} r_\alpha \bfs^\alpha\in\Ufil^p(R;\Delta)$ we have:
\begin{eqnarray*}
& \displaystyle
\bupsigma \left( \varphi\sbullet r\right) = \bupsigma \left( \sum_{\scriptscriptstyle e\in \nabla} 
\left( \sum_{\substack{\scriptscriptstyle \alpha\in \Delta\\ \scriptscriptstyle |e|\geq |\alpha|}} {\bf C}_e(\varphi,\alpha) r_\alpha \right) \bft^e\right) = 
\sum_{\scriptscriptstyle e\in \nabla} 
\sigma_{|e|} \left( \sum_{\substack{\scriptscriptstyle \alpha\in \Delta\\ \scriptscriptstyle |e|\geq |\alpha|}} {\bf C}_e(\varphi,\alpha) r_\alpha \right) \bft^e =
&
\\
& \displaystyle
\sum_{\scriptscriptstyle e\in \nabla} 
\sigma_{|e|} \left( \sum_{\substack{\scriptscriptstyle \alpha\in \Delta\\ \scriptscriptstyle |e|= |\alpha|}} {\bf C}_e(\varphi,\alpha) r_\alpha \right) \bft^e =
\sum_{\scriptscriptstyle e\in \nabla} 
\sigma_{|e|} \left( \sum_{\substack{\scriptscriptstyle \alpha\in \Delta\\ \scriptscriptstyle |e|= |\alpha|}} {\bf C}_e(\init \varphi,\alpha) r_\alpha \right) \bft^e =
&
\\
& \displaystyle
\sum_{\scriptscriptstyle e\in \nabla} 
 \left( \sum_{\substack{\scriptscriptstyle \alpha\in \Delta\\ \scriptscriptstyle |e|= |\alpha|}} {\bf C}_e(\init \varphi,\alpha) \sigma_{|\alpha|} \left(r_\alpha \right) \right) \bft^e =
  (\init \varphi) \sbullet \bupsigma(r).
\end{eqnarray*}
\end{proof}

\subsection{Exponential type series and divided power algebras}

General references for the notions and results in this section are
\cite{roby_63,roby_65}, \cite{bert_ogus} and \cite{laksov-notes}. 
In this section, $A$ will be a fixed commutative ring.
\medskip

For a given integer $m\geq 1$  or
$m=\infty$, we consider the following substitution maps:
\begin{eqnarray*}
& \displaystyle  \varphi:A[[t]]_m \longrightarrow A[[t,t']]_m,\quad \varphi(t) = t+t', &
\\
& \displaystyle 
\iota: A[[t]]_m \longrightarrow A[[t,t']]_m,\quad  \iota(t) = t, &
\\
& \displaystyle 
\iota': A[[t]]_m \longrightarrow A[[t,t']]_m,\quad  \iota'(t) = t'. &
\end{eqnarray*}
For each commutative $A$-algebra $B$, the above substitution maps induce homomorphisms of $A$-algebras (actually, they are the ``same'' substitution maps over $B$):
\begin{eqnarray*}
& \displaystyle  \varphi\sbullet (-): r(t) \in B[[t]]_m \longmapsto r(t+t') \in B[[t,t']]_m, &
\\
& \displaystyle 
\iota\sbullet (-): r(t)\in B[[t]]_m \longmapsto r(t)\in B[[t,t']]_m, &
\\
& \displaystyle 
\iota'\sbullet (-):  r(t)\in B[[t]]_m \longmapsto r(t')\in B[[t,t']]_m.  &
\end{eqnarray*}

\begin{defi} \label{defi:exponential-series} 
An element
$r=r(t)=\sum_{i=0}^m r_i t^i$ in $B[[t]]_m$ is said to
be of {\em exponential type} \index{exponential type series} if $r_0=1$ and $r(t+t')=r(t)r(t')$, i.e. $\varphi\sbullet r = \left(\iota\sbullet r\right) \left(\iota'\sbullet r\right)$, or
equivalently, if
$$ \binom{i+j}{i} r_{i+j} = r_i r_j,\quad \text{whenever}\ i+j <
m+1.$$ The set of elements in $B[[t]]_m$ of exponential type will be
denoted by $\EXP_m(B)$. The set $\EXP_{\infty}(B)$ will be simply
denoted by $\EXP(B)$.
\end{defi}

The set $\EXP_m(B)$ is a subgroup $\U(B;m)$ and
the external operation
\begin{equation} \label{eq:mod_action_on_EXP}
 \left(a,\sum_{i=0}^m r_i t^i\right) \in B \times \EXP_m(B) \mapsto
\sum_{i=0}^m r_i (at)^i=\sum_{i=0}^m r_i a^i t^i\in \EXP_m(B)
\end{equation}
defines a natural $B$-module structure on $\EXP_m(B)$. It is clear
that $\EXP_1(B)$ is canonically isomorphic to $B$ (as $B$-module). 
\medskip

Let $C$ be another commutative $A$-algebra. For each $m\geq 1$, any $A$-algebra
map $h:B\to C$ induces obvious A-linear maps $\EXP_m(h):\EXP_m(B)
\to \EXP_m(C)$. In this way we obtain functors $\EXP_m$ from the
category of commutative $A$-algebras to the category of $A$-modules.  
For $1\leq
m \leq q\leq \infty$, the projections $B[[t]]_q \to B[[t]]_m$ induce natural
truncation maps  $\EXP_q \to \EXP_m$ and we have (see (\ref{eq:U-inv-limit-finite})):
\begin{equation*} 
   \EXP(B) = \lim_{\stackrel{\longleftarrow}{m\in\N}} \EXP_m(B).
\end{equation*}
When $\QQ \subset B$, any $r=\sum_{i=0}^m r_i t^i \in \EXP_m(B)$ is determined by $r_1$, since $r_i= \frac{r^i}{i!}$ for all $i=0\dots, m$, and so all truncation maps $\EXP_q(B) \to \EXP_m(B)$, $1\leq
m \leq q\leq \infty$, are isomorphisms and $B \simeq \EXP_1(B) \simeq \EXP_m(B) \simeq \EXP(B).$
\medskip

The
following result is proven in \cite[Chap. III]{roby_63} in the case $m=\infty$.
The proof for any integer $m\geq 1$ is completely similar.

\begin{prop} \label{prop:PU-PD} For each $A$-module $M$ and each $m\geq 1$ there is an universal
pair $(\Gamma_{A,m} M,\gamma_{A,m})$, where $\Gamma_{A,m} M$ is a commutative $A$-algebra and $\gamma_{A,m}:M \to \EXP_m (\Gamma_{A,m} M)$ is an $A$-linear map, satisfying the
following universal property: for any commutative $A$-algebra $B$ and any $A$-linear map $H:M\to\EXP_m(B)$ there is a unique homomorphism of $A$-algebras
$h:\Gamma_{A,m} M\to B$ such that $H=\EXP_m(h)\pcirc \gamma_{A,m}$, or equivalently, the map
$$ h \in \Hom_{A-\text{\rm alg}}(\Gamma_{A,m} M,B) \mapsto \EXP_m(h)\pcirc
\gamma_{A,m} \in \Hom_A(M,\EXP_m(B))$$ is bijective.
\end{prop}

The pair $(\Gamma_{A,m} M,\gamma_{A,m})$ is unique up to a unique isomorphism. For $m=1$, we have a canonical isomorphism $\Sim_A M \xrightarrow{\sim} \Gamma_{A,1} M$.

\begin{defi}  \label{defi:power-divided-algebra} 
The $A$-algebra $\Gamma_{A,m} M$ is called the {\em algebra
of $m$-divided powers} \index{divided powers, algebra of} of $M$ and it is canonically $\N$-graded with
$\Gamma_{A,m}^0 M=A$, $\Gamma_{A,m}^1 M=M$. In the case $m=\infty$,
$(\Gamma_{A,\infty} M,\gamma_{A,\infty})$ is simply denoted by $(\Gamma_A
M,\gamma_A)$ and it is called the {\em algebra of divided powers} of
$M$.
\end{defi}

In this way $\Gamma_{A,m}$ becomes a functor from the category of $A$-modules to the category of ($\N$-graded) commutative $A$-algebras, which is left adjoint
to $\EXP_m$. For $1\leq m \leq q\leq \infty$ the truncations $\EXP_q \to \EXP_m$ induce  natural transformations $\Gamma_{A,m} \to \Gamma_{A,q}$ and $\displaystyle
\Gamma_A = \lim_{\stackrel{\longrightarrow}{m\in\N}} \Gamma_{A,m}$.
\medskip

When $\QQ \subset A$, we have $\Sim_A  \xrightarrow{\sim} \Gamma_{A,1} \xrightarrow{\sim} \Gamma_{A,m} \xrightarrow{\sim} \Gamma_A$ for all $m\geq 1$. For instance, for $A=\ZZ$ and $M=\ZZ x$ a free abelian group of rank $1$, the algebra $\Gamma_{\ZZ,m} M$ is the $\ZZ$-subalgebra $\ZZ\left[x^i/i!, 1\leq i \leq m\right] \subset \QQ[x]$  and
$$ \gamma_{A,m}: nx \in \ZZ x \longmapsto \sum_{i=0}^m n^i \frac{x^i}{i!} t^i \in \EXP_m \left(\ZZ\left[x^i/i!, 1\leq i \leq m\right]\right).
$$

\section{Hasse--Schmidt derivations}

\subsection{Definitions and first results}
\label{section:HS}

In this section we recall some notions and results of the theory of Hasse--Schmidt derivations \cite{has37} as developed in \cite{nar_subst_2018}. See also \cite{hoff_kow_2016}.
\medskip

From now on $k$ will be a commutative ring, $A$ a commutative $k$-algebra, $\bfs =\{s_1,\dots,s_p\}$ a set of variables and $\Delta \subset \N^p$ a non-empty co-ideal.

\begin{defi} \label{defi:HS}
A {\em $(p,\Delta)$-variate Hasse--Schmidt derivation}, or 
a {\em $(p,\Delta)$-va\-riate HS-de\-ri\-va\-tion} for short, of $A$ over $k$  
 is a family $D=(D_\alpha)_{\alpha\in \Delta}$ 
 of $k$-linear maps $D_\alpha:A
\longrightarrow A$, with $D_0=\Id_A$ and satisfying the following Leibniz type identities: $$ 
D_\alpha(xy)=\sum_{\beta+\gamma=\alpha}D_\beta(x)D_\gamma(y) $$ for all $x,y \in A$ and for all
$\alpha\in \Delta$. We denote by
$\HS^p_k(A;\Delta)$ the set of all $(p,\Delta)$-variate HS-derivations of $A$ over
$k$ and $\HS^p_k(A)$ for $\Delta = \N^p$. When $\Delta=\ft_m$ we will simply denote $\HS^p_k(A;m) := \HS^p_k(A;\ft_m)$. For $p=1$, a $1$-variate HS-derivation will be simply called a {\em Hasse--Schmidt derivation}   (a HS-derivation for short), or a {\em higher derivation}\footnote{This terminology is used for instance in \cite{mat_86}.}, and we will simply write $\HS_k(A;m):= \HS^1_k(A;\Delta)$ for $\Delta=\ft_m=\{q\in\N\ |\ q\leq m\}$\footnote{These HS-derivations are called of length $m$ in \cite{nar_2012}.} and $\HS_k(A) := \HS^1_k(A)$.
\end{defi}

Any $(p,\Delta)$-variate HS-derivation $D$ of $A$ over $k$  can be understood as a power series 
$$ \sum_{\scriptscriptstyle \alpha\in\Delta} D_\alpha \bfs^\alpha \in R[[\bfs]]_\Delta,\quad R=\End_k(A),$$
and so we consider $\HS^p_k(A;\Delta) \subset R[[\bfs]]_\Delta$. Actually 
$\HS^p_k(A;\Delta)$ is a (multiplicative) sub-group of $\U^p(R;\Delta)$. 
The group operation in $\HS^p_k(A;\Delta)$ is explicitly given by:
$$ (D,E) \in \HS^p_k(A;\Delta) \times \HS^p_k(A;\Delta) \longmapsto D \pcirc E \in \HS^p_k(A;\Delta)$$
with
$$ (D \pcirc E)_\alpha = \sum_{\scriptscriptstyle \beta+\gamma=\alpha} D_\beta \pcirc E_\gamma,
$$
and the identity element of $\HS^p_k(A;\Delta)$ is $\mathbb{I}$ with $\mathbb{I}_0 = \Id$ and 
$\mathbb{I}_\alpha = 0$ for all $\alpha \neq 0$. The inverse of a $D\in \HS^p_k(A;\Delta)$ will be denoted by $D^*$. 
\medskip

For $\Delta' \subset \Delta \subset \N^p$ non-empty co-ideals, we have truncations
$$\tau_{\Delta \Delta'}: \HS^p_k(A;\Delta) \longrightarrow \HS^p_k(A;\Delta'),
$$ which obviously are group homomorphisms. 
For $m\geq n$ we will denote $\tau_{mn}: \HS^p_k(A;m) \to \HS^p_k(A;n)$ the truncation map.
Since any $D\in\HS_k^p(A;\Delta)$ is determined by its finite truncations, we have a natural group isomorphism
\begin{equation} \label{eq:HS-inv-limit-finite}
 \HS_k^p(A) =  \lim_{\stackrel{\longleftarrow}{\substack{\scriptscriptstyle \Delta' \subset \Delta\\ \scriptscriptstyle  \sharp\Delta'<\infty}}} \HS_k^p(A;\Delta').
\end{equation}
The proof of the following proposition is clear and is left to the reader. 

\begin{prop} Let $\bft=\{t_1,\dots,t_q\}$ be another set of variables, $\nabla \subset \N^q$ a non-empty co-ideal,
and $D\in \HS^p_k(A;\Delta)$, $E \in \HS^q_k(A;\nabla)$ HS-derivations. Then its external product $D\boxtimes E$ (see Definition \ref{defi:external-x}) is a $(p+q,\nabla\times\Delta)$-variate HS-derivation. 
\end{prop}

\begin{defi}  \label{defi:sbullet-0}
For each $a\in A^p$ and for each $D\in \HS_k^p(A;\Delta)$, we define $a \sbullet D$ as
$$ (a\sbullet D)_\alpha := a^\alpha D_\alpha,\quad \forall \alpha\in \Delta.
$$
It is clear that $a \sbullet D\in \HS_k^p(A;\Delta)$, $a'\sbullet (a\sbullet D) = (a'a)\sbullet D$, $1 \sbullet D=D$ and $0 \sbullet D=\mathbb{I}$.
\end{defi}

If $\Delta' \subset \Delta \subset \N^p$ are non-empty co-ideals,
we have $\tau_{\Delta \Delta'}(a\sbullet D)=a\sbullet \tau_{\Delta \Delta'}(D)$. 
In particular, the image of $\tau_{m1}: \HS_k(A;m) \to \HS_k(A;1) \equiv \Der_k(A)$
is an $A$-submodule.

\begin{notacion} \label{notacion:pcirc-HS}
Let us denote:
\begin{eqnarray*}
&\displaystyle  \Hom_{k-\text{\rm alg}}^\pcirc(A,A[[\bfs]]_\Delta) :=
\left\{ f \in \Hom_{k-\text{\rm alg}}(A,A[[\bfs]]_\Delta),\ f(a) \equiv a\!\!\!\!\mod \left(\fn_0\right)_A\  \forall a\in A \right\}, &\\
&\displaystyle  \Aut_{k[[\bfs]]_\Delta-\text{\rm alg}}^\pcirc(A[[\bfs]]_\Delta) :=
 \left\{ f \in \Aut_{k[[\bfs]]_\Delta-\text{\rm alg}}(A[[\bfs]]_\Delta)\ |\ f(a) \equiv a_0\!\!\!\!\mod \left(\fn_0\right)_A\ \forall a\in A[[\bfs]]_\Delta \right\}.
\end{eqnarray*}
\end{notacion}
It is clear that $\Hom_{k-\text{\rm alg}}^\pcirc(A,A[[\bfs]]_\Delta) \subset \Hom_k^\pcirc(A,A[[\bfs]]_\Delta)$ and 
$$\Aut_{k[[\bfs]]_\Delta-\text{\rm alg}}^\pcirc(A[[\bfs]]_\Delta) \subset \Aut_{k[[\bfs]]_\Delta}^\pcirc(A[[\bfs]]_\Delta)$$ (see Notation \ref{notacion:pcirc}) are subgroups and
we have  group isomorphisms (see (\ref{eq:Aut-iso-pcirc}) and  (\ref{eq:U-iso-pcirc})):

\begin{equation} \label{eq:HS-funda}
\begin{tikzcd}
\HS^p_k(A;\Delta) \ar[r,"D \mapsto \widetilde{D}","\simeq"'] & \Aut_{k[[\bfs]]_\Delta-\text{\rm alg}}^\pcirc(A[[\bfs]]_\Delta)  \ar[r,"\text{restr.}","\simeq"'] & \Hom_{k-\text{\rm alg}}^\pcirc(A,A[[\bfs]]_\Delta).
\end{tikzcd}
\end{equation}
The composition of the above isomorphisms is given by:
\begin{equation} \label{eq:HS-iso-Hom(A,A[[s]])}
 D\in \HS^p_k(A;\Delta) \stackrel{\sim}{\longmapsto} \Phi_D := \left[a\in A \mapsto \sum_{\scriptscriptstyle \alpha\in\Delta} D_\alpha(a) \bfs^\alpha\right] \in \Hom_{k-\text{\rm alg}}^\pcirc(A,A[[\bfs]]_\Delta).
\end{equation}
Notice that the identity $D_0 = \Id$ corresponds to the fact that $\Phi_D(a) \equiv a$ modulo $\left(\fn_0\right)_A$ for all $a\in A$, Leibniz identities in Definition \ref{defi:HS} correspond to the fact that $\Phi_D$ is a ring homomorphism, and $k$-linearity of the $D_\alpha$ correspond to $k$-linearity of $\Phi_D$.
\medskip

For each HS-derivation $D\in \HS^p_k(A;\Delta)$ we have $\widetilde{D} = \left(\Phi_D\right)^e$, i.e.:
$$ \widetilde{D}\left( \sum_{\scriptscriptstyle \alpha\in\Delta} a_\alpha \bfs^\alpha \right) = 
\sum_{\scriptscriptstyle \alpha\in\Delta} \Phi_D(a_\alpha) \bfs^\alpha\quad 
$$
for all $\sum_\alpha a_\alpha \bfs^\alpha \in A[[\bfs]]_\Delta$, and for any $E\in \HS^p_k(A;\Delta)$ we have $\Phi_{D \smallcirc E} = \widetilde{D} \pcirc \Phi_E$. 
If $\Delta'\subset \Delta$ is another non-empty co-ideal and
 we denote by  $\pi_{\Delta \Delta'}: A[[\bfs]]_\Delta \to A[[\bfs]]_{\Delta'}$ the projection (or truncation), one has $\Phi_{\tau_{\Delta \Delta'} (D)} = \pi_{\Delta \Delta'}\pcirc \Phi_D$. 

\begin{defi} \label{def:ell-new} For each HS-derivation
$E\in\HS^p_k(A;\Delta)$, we denote\footnote{This definition changes slightly with respect to Definition (1.2.7) in \cite{nar_2012}.}
$$ \ell(E) := \min \{r\geq 1 \ |\ \exists \alpha\in \Delta, |\alpha|= r, E_\alpha \neq 0 \} \geq 1$$
if $E\neq \mathbb{I}$ and $ \ell(E) = \infty$ if $E= \mathbb{I}$. In other words, $\ell (E) = \ord (E-\mathbb{I})$. 
\end{defi}

We obviously have $\ell(E\pcirc E') \geq \min \{\ell(E),\ell(E')\}$ and $\ell(E^*) = \ell(E)$. Moreover, if $\ell(E') > \ell(E)$, then $\ell(E\pcirc E')=\ell(E)$. The next two results are proven in Propositions 7 and 8 of \cite{nar_subst_2018}.

\begin{prop}  \label{prop:ell-new}
For each $D\in \HS^p_k(A;\Delta)$ we have that $D_\alpha$ is a $k$-linear differential operator of order $\leq \lfloor \frac{|\alpha|}{\ell(D)}\rfloor$
for all $\alpha\in \Delta$. 
\end{prop}

\noindent
As a consequence of the above proposition we have $\HS^p_k(A;\Delta) \subset \Ufil^p(\DD_{A/k};\Delta)$.

\begin{lemma} \label{lemma:ell-corchete}
For any $D,E\in \HS^{\bfs}_k(A;\Delta)$ we have $ \ell([D,E]) \geq \ell(D) + \ell(E)$.
\end{lemma}

\begin{proof} It is a consequence of the identity 
$[D,E]- \mathbb{I} = \left[ (D- \mathbb{I}), (E- \mathbb{I})\right] D^* E^*.
$
\end{proof}

Proposition \ref{prop:ell-new} can be improved in the following way.

\begin{defi} \label{def:ell-new-alpha} For each HS-derivation
$E\in\HS^p_k(A;\Delta)$ and each $\alpha\in\Delta$, we denote 
$ \ell_\alpha(E) := \ell\left(\tau_{\Delta,\fn_\alpha}(E)\right)$, 
i.e.
$$ \ell_\alpha(E):= \min \{r\geq 1 \ |\ \exists \beta\leq\alpha, |\beta|= r, E_\beta \neq 0 \} \geq 1$$
if $\tau_{\Delta,\fn_\alpha}(E)\neq \mathbb{I}$ and $ \ell_\alpha(E) = \infty$ if $\tau_{\Delta,\fn_\alpha}(E)= \mathbb{I}$.
\end{defi}

It is clear that $\ell(E) \leq \ell_\alpha(E)$ for all $\alpha\in\Delta$. Replacing $D$ with $\tau_{\Delta,\fn_\alpha}(D)$ makes obvious the following proposition.

\begin{prop}  \label{prop:ell-new-alpha}
For each $D\in \HS^p_k(A;\Delta)$ we have that $D_\alpha$ is a $k$-linear differential operator or order $\leq \lfloor \frac{|\alpha|}{\ell_\alpha(D)}\rfloor$
for all $\alpha\in \Delta$. 
\end{prop}

\subsection{The action of substitution maps on HS-derivations}

In this section, $k$ will be a commutative ring, $A$ a commutative $k$-algebra, $R=\End_k(A)$, $\bfs =\{s_1,\dots,s_p\}$, $\bft =\{t_1,\dots,t_p\}$  sets of variables and $\Delta \subset \N^p$, $\nabla \subset \N^q$ non-empty co-ideals.
\medskip

Let us recall Proposition 10 in \cite{nar_subst_2018}.

\begin{prop} \label{prop:equiv-action-subs-HS}
For any substitution map $\varphi:A[[\bfs]]_\Delta \to A[[\bft]]_\nabla$, we have:
\begin{enumerate}
\item[1)] $\varphi_*\left(\Hom_{k-\text{\rm alg}}^\pcirc(A,A[[\bfs]]_\Delta)\right) \subset \Hom_{k-\text{\rm alg}}^\pcirc(A,A[[\bft]]_\nabla)$,
\item[2)] $\varphi \sbullet \HS^p_k(A;\Delta) \subset \HS^q_k(A;\nabla)$,
\item[3)] $ \varphi\sbullet \Aut_{k[[\bfs]]_\Delta-\text{\rm alg}}^\pcirc(A[[\bfs]]_\Delta)  \subset
\Aut_{k[[\bft]]_\nabla-\text{\rm alg}}^\pcirc(A[[\bft]]_\nabla)$.
\end{enumerate}
\end{prop}

We have then a commutative diagram:
\begin{equation} \label{eq:diag-funda-HS}
\begin{tikzcd}
\Hom_{k-\text{\rm alg}}^\pcirc(A,A[[\bfs]]_\Delta) \ar[d,"\varphi_*"'] & \HS^p_k(A;\Delta) \ar[l,"\sim"', "\Phi_D \mapsfrom D"] \ar[r,"\sim"] \ar[d,"\varphi\sbullet (-)"] & \Aut_{k[[\bfs]]_\Delta-\text{\rm alg}}^\pcirc(A[[\bfs]]_\Delta) \ar[d,"\varphi\sbullet (-)"] \\
\Hom_{k-\text{\rm alg}}^\pcirc(A,A[[\bft]]_\nabla)  & \HS^q_k(A;\nabla) \ar[l,"\sim"', "\Phi_D \mapsfrom D"]
\ar[r,"\sim"] &
\Aut_{k[[\bft]]_\nabla-\text{\rm alg}}^\pcirc(A[[\bft]]_\nabla).
\end{tikzcd}
\end{equation}
In particular, for any HS-derivation $D\in \HS^p_k(A;\Delta)$ we have $\varphi \sbullet D\in \HS^q_k(A;\nabla)$ (see \ref{nume:def-sbullet}). Moreover
$ \Phi_{\varphi \sbullet D} = \varphi \pcirc \Phi_D$.
\medskip

It is clear that for any co-ideals $\Delta' \subset \Delta$ and $\nabla' \subset \nabla$ with $\varphi \left( \Delta'_A/\Delta_A \right) \subset \nabla'_A/\nabla_A$
we have
\begin{equation} \label{eq:trunca_bullet}
\tau_{\nabla \nabla'}(\varphi \sbullet D) = \varphi' \sbullet \tau_{\Delta \Delta'}(D),
\end{equation}
where $\varphi': A[[\bfs]]_{\Delta'} \to A[[\bft]]_{\nabla'}$ is the substitution map induced by $\varphi$.
\medskip

Let us notice that any  $a\in A^p$ gives rise to a substitution map $\varphi: A[[\bfs]]_\Delta  \to A[[\bfs]]_\Delta$ given by $\varphi(s_i)= a_i s_i$ for all $i=1,\dots,p$, and one has $a \sbullet D = \varphi \sbullet D$.
\bigskip

\numero
\label{nume:def-sbullet-HS} Let $\bfu=\{u_1,\dots,u_r\}$ be another set of variables, $\Omega\subset \N^r$ a non-empty co-ideal, $\varphi\in\Sub_A(p,q;\Delta,\nabla)$, $\psi\in\Sub_A(q,r;\nabla,\Omega)$  substitution maps and $D,D'\in\HS_k^p(A;\Delta)$ HS-derivations. From \ref{nume:def-sbullet} we deduce the following properties:\medskip

\noindent -) If we denote $E:= \varphi \sbullet D \in\HS_k^q(A;\nabla)$,  we have
\begin{equation} \label{eq:expression_phi_D} 
E_0=\Id,\quad  E_e = \sum_{\substack{\scriptstyle \alpha\in\Delta\\ \scriptstyle |\alpha|\leq |e| }} {\bf C}_e(\varphi,\alpha) D_\alpha,\quad \forall e\in \nabla.
\end{equation}
-) If $\varphi = \mathbf{0}$ is the trivial substitution map or if $D=\mathbb{I}$, then 
$\varphi \sbullet D = \mathbb{I}$.
\medskip

\noindent
-) If $\varphi$ has \underline{constant coefficients}, then $(\varphi \sbullet D)^* = \varphi \sbullet D^*$
and
$\varphi \sbullet (D\pcirc D') = (\varphi \sbullet D) \pcirc (\varphi \sbullet D')$. The general case is treated in Proposition \ref{prop:varphi-D-main}.
\medskip

\noindent 
-) $ \psi \sbullet (\varphi \sbullet D) = (\psi \pcirc \varphi) \sbullet D$.
\medskip

\noindent 
-) $\ell(\varphi\sbullet D)\geq \ord(\varphi) \ell(D)$.
\medskip

The following result is proven in Propositions 11 and 12 of \cite{nar_subst_2018}.

\begin{prop} \label{prop:varphi-D-main}
Let $\varphi:A[[\bfs]]_\Delta  \to A[[\bft]]_\nabla$ be a substitution map. Then, 
the following assertions hold:
\begin{enumerate}
\item[(i)] For each $D\in\HS_k^p(A;\Delta)$ there is a unique substitution map
$\varphi^D: A[[\bfs]]_\Delta  \to A[[\bft]]_\nabla$ such that 
$\left(\widetilde{\varphi \sbullet D}\right) \pcirc \varphi^D =  \varphi \pcirc \widetilde{D}$. 
Moreover, $\left(\varphi\sbullet D\right)^* = \varphi^D \sbullet D^*$, $\varphi^{\mathbb{I}} = \varphi$ and:
$$ {\bf C}_e(\varphi,f+\nu) = \sum_{\substack{\scriptscriptstyle\beta+\gamma=e\\ 
\scriptscriptstyle |f+g|\leq |\beta|,|\nu|\leq |\gamma| }} {\bf C}_\beta(\varphi,f+g) D_g({\bf C}_\gamma(\varphi^D,\nu)) 
$$
for all $e\in \Delta$ and for all $f,\nu\in\nabla$ with $|f+\nu|\leq |e|$.
\item[(ii)] For each $D,E\in\HS_k^p(A;\Delta)$, we have $\varphi \sbullet (D \pcirc E) = (\varphi \sbullet D) \pcirc (\varphi^D \sbullet E)$ and 
 $\left(\varphi^D\right)^E = \varphi^{D \pcirc E}$. 
 In particular, $\left(\varphi^D\right)^{D^*} = \varphi$. 
\item[(iii)] If $\psi$ is another composable substitution map, then $(\varphi \pcirc \psi)^D = \varphi^{\psi\sbullet D} \pcirc \psi^D$.
\item[(iv)] If $\varphi$ has constant coefficients then $\varphi^D = \varphi$.
\end{enumerate}
\end{prop} 

\begin{defi} \label{def:D-element}
 Let $S$ be a $k$-algebra over $A$, $D\in\HS^p_k(A;\Delta)$ and $r\in\U^p(S;\Delta)$. We say that
$r$ is a {\em $D$-element} if 
$ r a = \widetilde{D}(a) r$ for all $a\in A[[\bfs]]_\Delta$. 
\end{defi}

Given $D\in  \U^p(\End_k(A);\Delta)$, it is clear that:
$$  D\in \HS^p_k(A;\Delta) \Longleftrightarrow D\ \text{is a $D$-element}.
$$
For $D=\mathbb{I}$ the identity HS-derivation, a $r\in\U^p(S;\Delta)$ is an {\em $\mathbb{I}$-element} if and only if $r$ commutes with all $a\in A[[\bfs]]_\Delta$. If $E\in \HS^p_k(A;\Delta)$ is another HS-derivation, $r\in\U^p(S;\Delta)$ is a $D$-element and $s\in\U^p(S;\Delta)$ is an $E$-element, then $rs$ is a $(D\pcirc E)$-element.
\medskip

The proof of the following lemma is easy and it is left to the reader.

\begin{lemma} \label{lemma:charact-D-elements}
With the above notations, 
for each $r=\sum_\alpha r_\alpha \bfs^\alpha\in\U^p(S;\Delta)$ the following properties are equivalent:
\begin{enumerate}
\item[-)] $r$ is a $D$-element.
\item[-)] $ b r  =  r \widetilde{D^*}(b)$ for all $b\in A[[\bfs]]_\Delta$. 
\item[-)] $r^*$ is a $D^*$-element.
\item[-)] If $r=\sum_\alpha r_\alpha \bfs^\alpha$, we have $r_\alpha a = \sum_{\beta+\gamma=\alpha} D_\beta(a) r_\gamma$ for all $a\in A$ and for all $\alpha \in\Delta$.
\item[-)] $ r a = \widetilde{D}(a) r$ for all $a\in A$.
\end{enumerate}
\end{lemma}

The following proposition generalizes Proposition \ref{prop:varphi-D-main}.

\begin{prop}   \label{prop:bullet-D-elements} 
Let $S$ be a $k$-algebra over $A$, $D\in\HS^p_k(A;\Delta)$, $\varphi:A[[\bfs]]_\Delta \to A[\bft]]_\nabla$ a substitution map and $r\in\U^p(S;\Delta)$ a $D$-element. 
Then the following properties hold:
\begin{enumerate}
\item[(a)] $ \varphi\sbullet r$ is a $(\varphi\sbullet D)$-element. 
\item[(b)] $\varphi\sbullet (rr') = (\varphi\sbullet r) (\varphi^D \sbullet r')$ for all $r'\in S[[\bfs]]_\Delta$. In particular, 
$(\varphi\sbullet r)^* = \varphi^D \sbullet r^*$. 
\end{enumerate}
Moreover, if $E$ is an $A$-module and $S=\End_k(E)$, then the following identity holds:
\begin{enumerate}
\item[(c)] $\langle \varphi \sbullet r, \varphi^D_E(e) \rangle 
= \varphi_E \left( \langle r,e\rangle\right)$ for all $e\in E[[\bfs]]_\Delta$, i.e.   
$\left(\varphi \sbullet \widetilde{r}\right) \pcirc \varphi^D_E 
= \varphi_E \pcirc \widetilde{r}$.
\end{enumerate}
\end{prop}

\begin{proof} \noindent (a) By Lemma \ref{lemma:charact-D-elements} we need to prove that $ \varphi_R(r) b = \left(\widetilde{\varphi\sbullet D}\right)\!\!(b)\, \varphi_R(r)$ for all $b\in A$, but we know that 
$r b = \widetilde{D}(b) r$ and so, from Lemma \ref{lemma:phi-linearity-of-phi_R} and (\ref{eq:tilde-varphi-M-varphi-F}), we deduce that
\begin{eqnarray*}
&\displaystyle (\varphi\sbullet r)b=
 \varphi_R(r) b = \varphi_R(rb) = \varphi_R\left( \widetilde{D}(b) r \right) =
&
\\
&\displaystyle 
\varphi \left( \widetilde{D}(b) \right)  \varphi_R(r) = \left(\widetilde{\varphi\sbullet D}\right)\!\!(b)\, \varphi_R(r) = \left(\widetilde{\varphi\sbullet D}\right)\!\!(b) (\varphi\sbullet r).
\end{eqnarray*}
(b) Since all the involved maps are $k$-linear and continuous, it is enough to prove the identity in the case where $r'= r'_\alpha \bfs^\alpha$ with $r'_\alpha \in R$ and $\alpha\in\Delta$.
But, on one hand we have
$$\varphi\sbullet (rr') =\varphi_R(r r'_\alpha \bfs^\alpha) = \varphi_R(\bfs^\alpha r r'_\alpha) = \varphi(\bfs^\alpha) \varphi_R( r r'_\alpha) = \varphi(\bfs^\alpha) \varphi_R( r ) r'_\alpha = \varphi(\bfs^\alpha) (\varphi\sbullet r) r'_\alpha, 
$$
and on the other hand, by using (a), we have
\begin{eqnarray*}
&\displaystyle
(\varphi\sbullet r) (\varphi^D \sbullet r') =
(\varphi\sbullet r) \varphi^D_R(r'_\alpha \bfs^\alpha) = (\varphi\sbullet r) \varphi^D(\bfs^\alpha) r'_\alpha =
\left(\widetilde{\varphi\sbullet D}\right) (\varphi^D(\bfs^\alpha)) (\varphi\sbullet r) r'_\alpha =
&\\
&\displaystyle
 \left(\left(\widetilde{\varphi\sbullet D}\right) \pcirc \varphi^D\right) (\bfs^\alpha) (\varphi\sbullet r) r'_\alpha =
\left(\varphi \pcirc \widetilde{D}\right) (\bfs^\alpha) (\varphi\sbullet r) r'_\alpha = \varphi(\bfs^\alpha) (\varphi\sbullet r) r'_\alpha
\end{eqnarray*}
and we are done.
For the last part,
$ 1 = \varphi_R(1) = \varphi_R(r r^*) = \varphi_R(r) \varphi^D_R(r^*)$.
\medskip

\noindent (c) As in part (b), it is enough to prove the identity for $e=e_\alpha \bfs^\alpha$, with $\alpha \in\Delta$ and $e_\alpha\in E$. By using the fact that 
$$ \varrho \in \End_k(E)[[\bfs]]_\Delta \longmapsto \widetilde{\varrho} \in \End_{k[[\bfs]]_\Delta}\left(E[[\bfs]]_\Delta \right)
$$
is an $\left(A[[\bfs]]_\Delta ;A[[\bfs]]_\Delta \right)$-linear isomorphism compatible with the $\varphi\sbullet (-)$ operation (see Lemma  \ref{lema:tilde-map} and (\ref{eq:commut-CD-varphi-bullet})),
 we deduce from part (a) that 
$\left(\widetilde{\varphi\sbullet r}\right) b =
\left(\widetilde{\varphi \sbullet D}\right)\!\! (b) \left(\widetilde{\varphi\sbullet r}\right)
$
for all $b\in A[\bft]]_\nabla$, and from
Proposition \ref{prop:varphi-D-main}, (i) and (\ref{eq:atencion-bullet}) we obtain:
\begin{eqnarray*}
&\displaystyle 
\langle \varphi \sbullet r, \varphi^D_E(e) \rangle  = \left(\widetilde{\varphi\sbullet r}\right) \left( \varphi^D_E \left( e_\alpha \bfs^\alpha \right) \right)=
\left(\widetilde{\varphi\sbullet r}\right) \left( \varphi^D(\bfs^\alpha) e_\alpha \right) =
\left(\widetilde{\varphi \sbullet D}\right) \left( \varphi^D(\bfs^\alpha) \right) \left(\widetilde{\varphi\sbullet r}\right)(e_\alpha)=
&
\\
&\displaystyle 
\varphi (\widetilde{D}(\bfs^\alpha) ) 
\varphi_E ( \widetilde{r}(e_\alpha))=
\varphi (\bfs^\alpha) \varphi_E ( \widetilde{r}(e_\alpha))= \varphi_E( \bfs^\alpha \widetilde{r} ( e_\alpha))=
\varphi_E(  \widetilde{r} (\bfs^\alpha e_\alpha)) = \varphi_E \left( \langle r,e\rangle\right).
\end{eqnarray*}
\end{proof}

\subsection{Integrable derivations and HS-smooth algebras}
\label{sec:int-der-HS-smooth}

In this section we recall some notions and results of \cite{nar_2009,nar_2012}.
Let $k$ be a commutative ring and $A$ a commutative $k$-algebra.
The following definition slightly changes with respect to Definition (2.1.1) in \cite{nar_2012}.

\begin{defi} (Cf. \cite{brown_1978,mat-intder-I})  \label{def:HS-integ}
Let $m\geq 1$ be an integer or $m=\infty$, and $\delta:A\to A$ a $k$-derivation.
 We say that $\delta$ is {\em $m$-integrable} (over $k$)  if there is a
HS-derivation $D\in \HS_k(A;m)$ such that $D_1=\delta$. Any such $D$ will
be called an  {\em $m$-integral} of $\delta$. The set of $m$-integrable
$k$-derivations of $A$ is denoted by $\Ider_k(A;m)$. We simply say that $\delta$ is
{\em integrable} if it is $\infty$-integrable and denote  $\Ider_k(A):=\Ider_k(A;\infty)$.
\smallskip

We say that $\delta$ is
{\em f-integrable} {\em (finite integrable)} if it is $m$-integrable for any \underline{integer} $m\geq 1$. The set of f-integrable
$k$-derivations of $A$ is denoted by $\Ider^f_k(A)$. 
\end{defi}

It is clear (see Definition \ref{defi:sbullet-0}) that the $\Ider_k(A;m)$ and $\Ider^f_k(A)$ are $A$-submodules of $\Der_k(A)$ and that we have exact sequences of groups:
\begin{equation} \label{eq:suex-ider}
1 \to \ker \tau_{m1} \xrightarrow{} \HS_k(A;m) \to \Ider_k(A;m) \to
0,\quad m\geq 1,
\end{equation}
and
$$ \Der_k(A)=\Ider_k(A;1)\supset\Ider_k(A;2)\supset
\Ider_k(A;3)\supset \cdots,$$
\begin{equation} \label{eq:intersection-ider} \Ider_k(A;\infty) \subset \Ider^f_k(A)= \bigcap_{\substack{\scriptscriptstyle m\in\N\\ \scriptscriptstyle  m\geq 1}} \Ider_k(A;m).
\end{equation}

\begin{exam} \label{ex:1} Let $m\geq 1$ be an integer. If $m!$ is invertible in $A$,
then any $k$-derivation $\delta$ of $A$ is $m$-integrable: we can
take $D \in \HS_k(A;m)$ defined by $D_i=\frac{\delta^i}{i!}$ for
$i=0,\dots,m$. If $\QQ\subset k$, one proves
in a similar way that any $k$-derivation of $A$ is $\infty$-integrable.
\end{exam}

Let us recall the following result (\cite[Theorem 27.1]{mat_86}):

\begin{prop} Let us assume that $A$ is a $0$-smooth $k$-algebra.
Then any $k$-derivation of $A$ is integrable.
\end{prop}

\begin{prop} \label{prop:N-infini} The following properties are equivalent:
\begin{enumerate}
\item[(a)] $\Der_k(A) = \Ider_k(A;\infty)$.
\item[(b)] $\Der_k(A) = \Ider_k(A;m)$ for all integers $m\geq
1$ ($\Leftrightarrow \Der_k(A) = \Ider^f_k(A)$).
\end{enumerate}
\end{prop}

\begin{proof} The implication (a) $\Rightarrow$ (b) is clear. 
\medskip

\noindent (b) $\Rightarrow$ (a)\ 
Let $\delta$ be a $k$-derivation of $A$. By hypothesis, there is a $2$-integral $D^{(2)}=(\Id,D_1,D_2)\in\HS_k(A;2)$ of $\delta$, and by applying \cite[Corollary 4]{nar_subst_2018} repeatedly we find 
a sequence $D^{(m)}\in \HS_k(A;m)$, $m\geq
2$, such that $\tau_{m,m-1} D^{(m)} = D^{(m-1)}$ for each $m\geq 2$.
We can take $\displaystyle D= \lim_{\stackrel{\longleftarrow}{m}}
D^{(m)}\in \HS_k(A)$, that obviously is an $\infty$-integral of $\delta$.
\end{proof}

\begin{nota} \label{nota:intersection-ider} In general, we know that
$$ \Ider_k(A;\infty) \subset \Ider^f_k(A) =\bigcap_{\scriptscriptstyle m\in\N_+} \Ider_k(A;m) \subset \Der_k(A).$$
Proposition \ref{prop:N-infini} tells us that the above inclusion is an equality whenever all the $k$-derivations of $A$ are $m$-integrable for each $m\in\N_+$. Otherwise, we do not know whether it is strict or not, or in other words, whether a derivation which is $m$-integrable for each integer $m\geq 1$ is $\infty$-integrable or not.
\end{nota}

\begin{defi} Let $m$ be a non-negative integer or $m=\infty$.
For any HS-derivation $D\in\HS_k(A;m)$ we define its
{\em total symbol} by (see Notation \ref{notacion:Ump}):
$$\Sigma_m(D) := \bupsigma(D)= \sum_{i=0}^m \sigma_i(D_i) t^i \in \Ugr(\gr \diff_{A/k};m).
$$
\end{defi}

The total symbol
map $\Sigma_m:\HS_k(A;m) \longrightarrow \Ugr(\gr \diff_{A/k};m)$ is a group homomorphism.
The following proposition is proven in \cite[Proposition 2.5, Corollary 2.7]{nar_2009}.

\begin{prop} With the hypotheses above, the following properties hold:
\begin{enumerate}
\item[(1)] The image of $\Sigma_m$ is contained in  $\EXP_m\left(\gr \diff_{A/k}\right) $.
\item[(2)] For any $D\in \HS_k(A;m)$ and any $a\in A$ we have $\Sigma_m(a \bullet D) =
a\Sigma_m(D)$.
\item[(3)] The map $\Sigma_m$ 
induces an $A$-linear map $\chi_m:\Ider_k(A;m) \to \EXP_m(\gr
\diff_{A/k})$.
\end{enumerate}
\end{prop}

It is clear that, for $1\leq m \leq q\leq \infty$, the following
diagram is commutative:
\begin{equation*}
\begin{tikzcd}
\Ider_k(A;q) \ar[r,"\chi_q"] \ar[d,hook,"\text{inc.}"'] & \EXP_q(\gr
\diff_{A/k}) \ar[d,"\text{trunc.}"]\\
\Ider_k(A;m) \ar[r,"\chi_m"] & \EXP_m(\gr
\diff_{A/k}).
\end{tikzcd}
\end{equation*}
By taking the inverse limit of the $\chi_m$ for $1\leq m < \infty$ we obtain an $A$-linear map $\chi^f: \Ider^f_k(A) \to \EXP(\gr
\diff_{A/k})$. Explicitly, if $\delta\in \Ider^f_k(A)$, then: 
$$\chi^f(\delta) = \sum_{m=0}^\infty \sigma_m\left(D_m^m\right) t^m
$$
where $D^m = \left(D^m_j\right)_{0\leq j\leq m} \equiv \sum_{j=0}^m D^m_j t^j \in \HS_k(A;m)$ is any $m$-integral of $\delta$ for each integer $m\geq 1$ ($D^0 = \mathbb{I}$).
\medskip

From the universal property of power divided algebras (see
Proposition \ref{prop:PU-PD}), we obtain a canonical homomorphism of graded
$A$-algebras:
\begin{equation} \label{eq:mor-canon}
\vartheta^f_{A/k}: \Gamma \Ider^f_k(A) \to \gr \diff_{A/k}.
\end{equation}
It is clear that for each integer $m\geq 1$, the following diagram is commutative:
$$
\begin{tikzcd}
\Gamma\Ider_k(A;\infty) \ar[r,"\text{nat.}"] \ar[dr,"\vartheta_{A/k,\infty}"'] & \Gamma\Ider^f_k(A) \ar[r,"\text{nat.}"] 
\ar[d,"\vartheta^f_{A/k}"] & \Gamma\Ider_k(A;m) \ar[dl,"\vartheta_{A/k,m}"]\\
 & \gr \diff_{A/k}, &
\end{tikzcd}
$$
where the $\vartheta_{A/k,m}$ and $\vartheta_{A/k,\infty}$ have been defined in  \cite[(2.6)]{nar_2009}.
The following two theorems are proven in \cite{nar_2009}, Theorem (2.8) and Theorem (2.14), for $\Ider_k(A;\infty)$,  $\vartheta_{A/k,\infty}$ instead of $\Ider^f_k(A)$, $\vartheta^f_{A/k}$, but the proofs remain essentially the same.

\begin{thm} \label{teo:commut} With the above notations, there are canonical maps $\theta_{A/k}$ and $\phi$ such that the following diagram of graded $A$-algebras is commutative:
\begin{equation*}
\begin{tikzcd}
\gr \diff_{A/k} \ar[r,"\theta_{A/k}"] & \left(\Sim_A \Omega_{A/k}\right)^*_{gr}\\
\Gamma \Ider^f_k(A) \ar[u,"\vartheta^f_{A/k}"] \ar[r,"\text{\rm nat.}"]
& \Gamma \Der_k(A). \ar[u,"\phi"']
\end{tikzcd}
\end{equation*}
\end{thm}

\begin{thm} \label{teo:impor} Assume that $\Der_k(A)$ is a projective $A$-module of finite rank.
The following properties are equivalent:
\begin{enumerate}
\item[(a)] The homomorphism of graded $A$-algebras
$\theta_{A/k}: \gr \diff_{A/k} \xrightarrow{} \left(\Sim_A \Omega_{A/k}\right)^*_{gr} $ 
 is an isomorphism.
\item[(b)] The homomorphism of graded $A$-algebras
$\vartheta^f_{A/k}: \Gamma \Ider^f_k(A) \xrightarrow{} \gr \diff_{A/k}$
is an isomorphism.
\item[(c)] $\Ider^f_k(A)=\Der_k(A)$.
\end{enumerate}
\end{thm}

\begin{nota} After Theorem (2.14) in \cite{nar_2009} or Proposition \ref{prop:N-infini}, the equivalent properties in Theorem \ref{teo:impor} are also equivalent to:
\begin{enumerate}
\item[(b')] The homomorphism of graded $A$-algebras
$$\vartheta_{A/k,\infty}: \Gamma \Ider_k(A;\infty) \xrightarrow{} \gr \diff_{A/k}
$$
is an isomorphism.
\item[(c')] $\Ider_k(A;\infty)=\Der_k(A)$.
\end{enumerate}
\end{nota}

\begin{defi} \label{defi:HS_smooth}
We say that a $k$-algebra $A$ is {\em HS-smooth}  if $\Der_k(A)$ is a projective $A$-module of finite rank and the equivalent properties (a), (b), (c) of Theorem \ref{teo:impor} hold.
\end{defi}

Let us recall the following result (\cite[Corollary (2.16)]{nar_2009}).

\begin{cor} Assume that $\Omega_{A/k}$ is a projective $A$-module of finite rank and
that $A$ is differentially smooth over $k$ (in the sense of \cite[16.10]{ega_iv_4}).
Then, $A$ is a HS-smooth $k$-algebra.
\end{cor}

In particular, after \cite[Proposition 17.12.4]{ega_iv_4}, if $A$ is a smooth finitely presented $k$-algebra, then 
$A$ is a HS-smooth $k$-algebra.

\section{Main results}  \label{sect:main}

\subsection{Hasse--Schmidt modules}

\begin{defi} \label{def:HS-structure}
Let $R$ be a $k$-algebra over $A$. A {\em pre-HS-structure} \index{pre-HS-structure} on $R$ over $A/k$ is a system of maps
$$ \Uppsi= \left\{\Uppsi^p_\Delta: \HS^p_k(A;\Delta) \longrightarrow \U^p(R;\Delta),\ p\in \N, \Delta\in\coide{\N^p}\right\}
$$
such that\footnote{Actually, from (\ref{eq:U-inv-limit-finite}) and (\ref{eq:HS-inv-limit-finite}) we could restrict ourselves to non-empty \underline{finite} co-ideals.}:
\begin{enumerate}
\item[(i)] The $\Uppsi^p_\Delta$ are group homomorphisms.
\item[(ii)] (Leibniz rule) For any $D\in \HS^p_k(A;\Delta)$, $\Uppsi^p_\Delta(D)$ is a $D$-element, i.e. $\Uppsi^p_\Delta(D)\, a = \widetilde{D}(a) \Uppsi^p_\Delta(D) $ for all $a\in A$ (see Lemma \ref{lemma:charact-D-elements}).
\item[(iii)] For any substitution map $\varphi\in\Sub_k(p,q;\Delta,\nabla)$ and for any $D\in \HS^p_k(A;\Delta)$ we have $\Uppsi^q_\nabla(\varphi\sbullet D) = \varphi\sbullet \Uppsi^p_\Delta(D)$.
\end{enumerate}
We say that a pre-HS-structure $\Uppsi$ on $R$ over $A/k$ is a {\em HS-structure} \index{HS-structure} if property (iii) above holds for any substitution map $\varphi\in\Sub_A(p,q;\Delta,\nabla)$.
\medskip

\noindent If $R'$ is another $k$-algebra over $A$ and $f:R\to R'$ is a map of $k$-algebras over $A$, then any (pre-)HS-structure $\Uppsi$ on $R$ over $A/k$ gives rise to a (pre-)HS-structure $f\pcirc \Uppsi$ on $R'$ over $A/k$ defined as
$$\left( f\pcirc \Uppsi \right)^p_\Delta := \overline{f} \pcirc \Uppsi^p_\Delta,\quad p\in \N, \Delta\in\coide{\N^p}.
$$

\noindent If $R$ is filtered, we will say that a (pre-)HS-structure $\Uppsi$ on $R$ over $A/k$ is {\em filtered} \index{filtered (pre-)HS-structure} if 
$$\Uppsi^p_\Delta(\HS^p_k(A;\Delta) ) \subset \Ufil^p(R;\Delta)$$ for all $p\in\N$ and all $\Delta\in\coide{\N^p}$.
\end{defi}

Let us notice that if $\Uppsi$ is a pre-HS-structure on $R$ over $A/k$, then the system of maps $ \Upgamma= 
\{\Upgamma^p_\Delta: \HS^p_k(A;\Delta) \longrightarrow \U^p(R^{\text{\rm opp}};\Delta),\ p\in \N, 
\Delta\in\coide{\N^p}\}
$ defined as $\Upgamma^p_\Delta(D) = \Uppsi^p_\Delta(D^*)$ for $D\in \HS^p_k(A;\Delta)$ is a pre-structure on $R^{\text{\rm opp}}$ over $A/k$. However, if $\Uppsi$ is a HS-structure on $R$ over $A/k$, the system $\Upgamma$ defined above is not in general HS-structure on $R^{\text{\rm opp}}$. More precisely, we have the following proposition.

\begin{prop} \label{prop:carac-dual-HS-structures}
Let $\Uppsi$ be a pre-HS-structure on $R$ over $A/k$ and let us consider the system of maps $ \Upgamma= \{\Upgamma^p_\Delta: \HS^p_k(A;\Delta) \longrightarrow \U^p(R^{\text{\rm opp}};\Delta),\ p\in \N, \Delta\in\coide{\N^p}\}
$ defined as $\Upgamma^p_\Delta(D) = \Uppsi^p_\Delta(D^*)$ for $D\in \HS^p_k(A;\Delta)$. The following properties are equivalent:
\begin{enumerate}
\item[(1)] $\Upgamma$ is a HS-structure on $R^{\text{\rm opp}}$ over $A/k$.
\item[(2)] For each $p,q\in \N$, for each $\Delta\in\coide{\N^p}, \nabla\in\coide{\N^q}$, for each substitution map $\varphi\in\Sub_A(p,q;\Delta,\nabla)$ and for each $D\in \HS^p_k(A;\Delta)$ we have $\Uppsi^q_\nabla (\varphi \sbullet D) = \Uppsi^p_\Delta (D) \sbullet \varphi^D$ (see Proposition \ref{prop:varphi-D-main}).
\end{enumerate}
\end{prop}

\begin{proof}  (1) $\Rightarrow$ (2): We know that for each $E\in \HS^p_k(A;\Delta)$ and each $\psi \in \Sub_A(p,q;\Delta,\nabla)$ we have $\Upgamma^q_\nabla (\psi \sbullet E) = \psi \stackrel{\text{\tiny\rm opp}}{\sbullet} \Upgamma^p_\Delta(E)$, i.e. $\Uppsi^q_\nabla \left((\psi \sbullet E)^*\right) = \Uppsi^p_\Delta(E^*) \sbullet \psi$, and we conclude by taking $E=D^*$ and $\psi = \varphi^D$ (see Proposition \ref{prop:varphi-D-main}):
$$ \Uppsi^q_\nabla (\varphi \sbullet D) = \Uppsi^q_\nabla\left( \psi^E \sbullet E^*\right) = \Uppsi^q_\nabla \left((\psi \sbullet E)^*\right) = \Uppsi^p_\Delta(E^*) \sbullet \psi = \Uppsi^p_\Delta(D) \sbullet \varphi^D.
$$
(2) $\Rightarrow$ (1): Properties (i) and (ii) are clear. For property (iii) we proceed as in (1) $\Rightarrow$ (2).
\end{proof}

\begin{exam} The inclusions 
$$\HS^p_k(A;\Delta) \hookrightarrow \U^p(\DD_{A/k};\Delta) \subset \U^p(\End_k(A);\Delta)
$$ give rise to the ``tautological'' HS-structures on $\DD_{A/k}$ and on $\End_k(A)$ over $A/k$.
\end{exam}

\begin{defi} (1) A {\em left (pre-)HS-module} \index{HS-module (pre-, left or right)} (resp. a  {\em right (pre-)HS-module}) over $A/k$ is an $A$-module $E$ endowed with a (pre-)HS-structure on $\End_k(E)$ (resp. on the opposed ring $\End_k(E)^{\text{\rm opp}}$) over $A/k$.\\
(2) A HS-map from a left (resp. a right) (pre-)HS-module $(E,\Phi)$ to a left (resp. to a right) (pre-)HS-module $(F,\Uppsi)$ is an $A$-linear map $f:E\to F$ such that $\overline{f} \pcirc \Phi^p_\Delta(D) = \Uppsi^p_\Delta(D) \pcirc \overline{f}$
 for all $p\in\N$,  for all $\Delta\in\coide{\N^p}$, for all $\alpha\in\Delta$ and for all $D\in \HS^p_k(A;\Delta)$.
\end{defi}

\begin{nota} \label{nota:Hsmod-equiv}
Let $E$ be an $A$-module and $R=\End_k(E)$. By using the canonical isomorphisms (\ref{eq:U-iso-pcirc}), we have the following:
\\

\noindent (1) For each left (pre-)HS-module $(E,\Uppsi)$, the (pre-)HS-structure $\Uppsi$ may be considered as a system of maps $ \Uppsi= \{\Uppsi^p_\Delta: \HS^p_k(A;\Delta) \longrightarrow \Aut^{\circ}_{k[[\bfs]]_\Delta}(E[[\bfs]]_\Delta),\ p\in \N, \Delta\in\coide{\N^p}\}
$, with $\bfs = \{s_1,\dots,s_p\}$, 
such that:
\begin{enumerate}
\item[(i)] The $\Uppsi^p_\Delta$ are group homomorphisms.
\item[(ii)] For any $D\in \HS^p_k(A;\Delta)$ and any $a\in A[[\bfs]]_\Delta$,  $\Uppsi^p_\Delta(D)\, a = \widetilde{D}(a) \Uppsi^p_\Delta(D)$.
\item[(iii)] For any substitution map $\varphi\in\Sub_A(p,q;\Delta,\nabla)$ (resp. for any substitution map $\varphi\in\Sub_k(p,q;\Delta,\nabla)$) and for any $D\in \HS^p_k(A;\Delta)$ we have $\Uppsi^q_\nabla(\varphi\sbullet D) = \varphi\sbullet \Uppsi^p_\Delta(D)$.
\end{enumerate}
Moreover, property (ii) above is equivalent to:
\begin{enumerate}
\item[(ii')] For any $D\in \HS^p_k(A;\Delta)$ and any $a\in A[[\bfs]]_\Delta$,  $a \, \Uppsi^p_\Delta(D)  =  \Uppsi^p_\Delta(D) \, \widetilde{D^*}(a)$.
\end{enumerate}

\noindent (2) For each right (pre-)HS-module $(E,\Uppsi)$, the (pre-)HS-structure $\Uppsi$ may be considered as a system of maps $ \Uppsi= \{\Uppsi^p_\Delta: \HS^p_k(A;\Delta) \longrightarrow \Aut^{\circ}_{k[[\bfs]]_\Delta}(E[[\bfs]]_\Delta),\ p\in \N, \Delta\in\coide{\N^p}\}
$
such that:
\begin{enumerate}
\item[(i)] The $\Uppsi^p_\Delta$ are group anti-homomorphisms.
\item[(ii)] For any $D\in \HS^p_k(A;\Delta)$ and any $a\in A[[\bfs]]_\Delta$,  $a \,\Uppsi^p_\Delta(D) =  \Uppsi^p_\Delta(D)\, \widetilde{D}(a)$.
\item[(iii)] For any substitution map $\varphi\in\Sub_A(p,q;\Delta,\nabla)$ (resp. for any substitution map $\varphi\in\Sub_k(p,q;\Delta,\nabla)$) and for any $D\in \HS^p_k(A;\Delta)$ we have $\Uppsi^q_\nabla(\varphi\sbullet D) = \Uppsi^p_\Delta(D) \sbullet \varphi$.
\end{enumerate}
Moreover, property (ii) above is equivalent to:
\begin{enumerate}
\item[(ii')] For any $D\in \HS^p_k(A;\Delta)$ and any $a\in A[[\bfs]]_\Delta$,  $\Uppsi^p_\Delta(D)\, a  = \widetilde{D^*}(a) \Uppsi^p_\Delta(D)$.
\end{enumerate}
\end{nota}

\begin{exam} The underlying $A$-module of any left (resp. right) $\DD_{A/k}$-module $E$ carries an obvious left (resp. right) HS-module structure, namely $ \Uppsi= \{\Uppsi^p_\Delta: \HS^p_k(A;\Delta) \longrightarrow \Aut^{\circ}_{k[[\bfs]]_\Delta}(E[[\bfs]]_\Delta),\ p\in \N, \Delta\in\coide{\N^p}\}
$
given by:
$$ \Uppsi^p_\Delta (D)(e) := \sum_{\scriptscriptstyle \alpha\in\Delta} \left(\sum_{\scriptscriptstyle \beta+\gamma= \alpha} D_\beta \cdot e_\gamma\right)\bfs^\alpha\quad 
\left(\text{resp.\ }\ \Uppsi^p_\Delta (D)(e) := \sum_{\scriptscriptstyle \alpha\in\Delta}\left(\sum_{\scriptscriptstyle \beta+\gamma= \alpha} e_\gamma \cdot D_\beta  \right)\bfs^\alpha\right)
$$
for all $D\in \HS^p_k(A;\Delta)$ and for all $e=\sum e_\gamma \bfs^\gamma \in E[[\bfs]]_\Delta$.
\medskip

When we consider the left $\DD_{A/k}$-module $E=A$, then its left HS-module structure is simply given by the injective group homomorphisms
$$
D\in \HS^p_k(A;\Delta) \longmapsto \widetilde{D} \in \Aut^{\circ}_{k[[\bfs]]_\Delta}(A[[\bfs]]_\Delta).
$$
\end{exam}

\begin{prop} \label{prop:pre-HS-Omega}
Under the above hypotheses, the $A$-module $\Omega_{A/k}$ has a unique left pre-HS-module structure over $A/k$ for which the differential $d:A \longrightarrow \Omega_{A/k}$ is a HS-map.
\end{prop}

\begin{proof} For each $p\in \N$, each $\Delta\in\coide{\N^p}$ and each $D\in\HS^\bfs_k(A;\Delta)$, let us consider $\Omega_{A/k}[[\bfs]]_\Delta$ as an $A$-module through the $k$-algebra map $\Phi_D: A \to A[[\bfs]]_\Delta$ (see (\ref{eq:HS-iso-Hom(A,A[[s]])})). It is clear that the map
$$ \overline{d} \pcirc \Phi_D: x\	\in A \longmapsto  \sum_{\alpha} d(D_\alpha(x)) \bfs^\alpha \in \Omega_{A/k}[[\bfs]]_\Delta
$$
is a $k$-derivation. So, there is a unique $A$-linear map $\fLie^p_\Delta(D): \Omega_{A/k} \longrightarrow \Omega_{A/k}[[\bfs]]_\Delta$ such that the following diagram is commutative:
$$
\begin{tikzcd}
A \ar[r,"d"] \ar[d,"\Phi_D"'] & \Omega_{A/k} \ar[d, "\fLie^p_\Delta (D)"]\\
A[[\bfs]]_\Delta \ar[r,"\overline{d}"] & \Omega_{A/k}[[\bfs]]_\Delta.
\end{tikzcd}
$$
If write $\fLie^p_\Delta (D) = \sum_{\alpha} \fLie^p_\Delta (D)_\alpha \bfs^\alpha$, each $\fLie^p_\Delta(D)_\alpha$ is $k$-linear, $\fLie^p_\Delta(D)_\alpha \pcirc d = d \pcirc D_\alpha$ for all $\alpha\in \Delta$
and the $A$-linearity of $\fLie^p_\Delta(D)$ means that
\begin{equation}  \label{eq:D-element-fLie}
\fLie^p_\Delta(D)_\alpha (a \omega) = \sum_{\scriptscriptstyle \alpha'+\alpha''=\alpha} D_{\alpha'}(a) \fLie^p_\Delta(D)_{\alpha''} (\omega)\  \forall a\in A, \forall \omega \in \Omega_{A/k}, \forall \alpha\in \Delta.
\end{equation}
In particular, $\fLie^p_\Delta(D)_0 = \Id$. In order to simplify, the canonical $k[[\bfs]]_\Delta$-linear extension of $\fLie^p_\Delta(D)$ to $\Omega_{A/k}[[\bfs]]_\Delta$ (see (\ref{eq:g^e})) will be also denoted by $\fLie^p_\Delta(D)$. We have then a commutative diagram:
$$
\begin{tikzcd}
A[[\bfs]]_\Delta \ar[r,"\overline{d}"] \ar[d,"\widetilde{D}"'] & \Omega_{A/k}[[\bfs]]_\Delta \ar[d, "\fLie^p_\Delta(D)"]\\
A[[\bfs]]_\Delta \ar[r,"\overline{d}"] & \Omega_{A/k}[[\bfs]]_\Delta.
\end{tikzcd}
$$
Let us see that the system:
$$\fLie := \{\fLie^p_\Delta:\HS^\bfs_k(A;\Delta) \to \Aut^{\circ}_{k[[\bfs]]_\Delta}(\Omega_{A/k}[[\bfs]]_\Delta),\ p\in \N, \Delta\in\coide{\N^p}
\}
$$ 
is a left pre-HS-module structure on $\Omega_{A/k}$ over $A/k$:
\medskip

\noindent (i) The uniqueness property defining $\fLie^p_\Delta(D)$ implies that the $\fLie^p_\Delta$ are group homomorphisms.
\medskip

\noindent (ii) Property (\ref{eq:D-element-fLie}) can be translated into $\fLie^p_\Delta(D) a = \widetilde{D}(a) \fLie^p_\Delta(D)$.
\medskip

\noindent (iii) Let $\varphi\in\Sub_k(p,q;\Delta,\nabla)$ be a substitution map with constant coefficients and $D\in \HS^p_k(A;\Delta)$. To prove the equality $\fLie^q_\nabla(\varphi\sbullet D) = \varphi\sbullet \fLie^p_\Delta(D)$, it is enough to prove that the restrictions to $\Omega_{A/k}$ of both terms coincide (see Lemma \ref{lema:tilde-map}), and this is a consequence of the identity 
$$ \left( \varphi\sbullet \fLie^p_\Delta(D) \right)|_{\Omega_{A/k}} = \varphi_\Omega \pcirc \fLie^p_\Delta(D),
$$
where $\varphi_\Omega =\varphi \widehat{\otimes} \Id_{\Omega_{A/k}}:\Omega_{A/k}[[\bfs]]_\Delta \to \Omega_{A/k}[[\bft]]_\nabla$ is the $\varphi$-linear map induced by $\varphi$ (see \ref{sec:action-substi} and (\ref{eq:diag-funda})), the identity $\Phi_{\varphi\sbullet D} = \varphi \pcirc \Phi_D$ (see (\ref{eq:diag-funda-HS})),
and the commutativity of the following diagram:
$$
\begin{tikzcd}
A \ar[r,"d"] \ar[d,"\Phi_D"'] & \Omega_{A/k} \ar[d, "\fLie^p_\Delta(D)"]\\
A[[\bfs]]_\Delta \ar[r,"\overline{d}"] \ar[d,"\varphi"'] & \Omega_{A/k}[[\bfs]]_\Delta \ar[d, "\varphi_\Omega"]\\
A[[\bft]]_\nabla \ar[r,"\overline{d}"] & \Omega_{A/k}[[\bft]]_\nabla.
\end{tikzcd}
$$
Let us notice that the commutativity of the bottom square depends on $\varphi$ being with constant coefficients.
\end{proof}

\begin{nota} With the notations of the above proposition, for each $\alpha \in\Delta$ with $|\alpha|=1$, the map 
$\fLie^p_\Delta(D)_\alpha:\Omega_{A/k}\to \Omega_{A/k}$ coincides with the classical Lie derivative
$\Lie_{D_\alpha}: \Omega_{A/k}\to \Omega_{A/k}$ with respect to the derivation $D_\alpha$.
\end{nota}

\begin{prop} The following properties hold:
\begin{enumerate}
\item[1)]  For each $p\in \N$, each $\Delta \in\coide{\N^p}$, each $D\in\HS^p_k(A;\Delta)$ and each $\delta\in\Der_k(A)[[\bfs]]_\Delta$  we have
$ D\, \delta\, D^* \in \Der_k(A)[[\bfs]]_\Delta$.
\item[2)] The system $\fAd :=\{\fAd^p_\Delta:\HS_k^p(A;\Delta) \to 
 \Aut^{\circ}_{k[[\bfs]]_\Delta}(\Der_k(A)[[\bfs]]_\Delta),\ p\in \N, \Delta \in\coide{\N^p}\}$, defined as
$$\fAd^p_\Delta(D)(\delta):=D\, \delta\, D^*\quad \forall D\in\HS^p_k(A;\Delta),\ \forall \delta\in  \Der_k(A)[[\bfs]]_\Delta,
$$
is a left pre-HS-module structure on $\Der_k(A)$ over $A/k$.
\end{enumerate}
\end{prop}

\begin{proof} 1)  
For each $a\in A[[\bfs]]_\Delta$ we have
\begin{eqnarray*}
& \displaystyle
[\widetilde{D\, \delta\, D^*}, a] = \widetilde{D}\, \widetilde{\delta}\, \widetilde{D^*}\, a - a\, \widetilde{D}\, \widetilde{\delta}\, \widetilde{D^*} = 
&\\
& \displaystyle
\widetilde{D}\, \widetilde{\delta}\, \widetilde{D^*}(a)\, \widetilde{D^*} - a\, \widetilde{D}\, \widetilde{\delta}\, \widetilde{D^*} =
\widetilde{D}\,  \widetilde{D^*}(a)\, \widetilde{\delta}\, \widetilde{D^*} + \widetilde{D}\, \widetilde{\delta}(\widetilde{D^*}(a))\,  \widetilde{D^*} - a\, \widetilde{D}\, \widetilde{\delta}\, \widetilde{D^*} =
&\\
& \displaystyle
\widetilde{D}(\widetilde{D^*}(a))\, \widetilde{D}\, \widetilde{\delta}\, \widetilde{D^*} + \widetilde{D}(\widetilde{\delta}(\widetilde{D^*}(a)))\,\widetilde{D} \,\widetilde{D^*} - a\, \widetilde{D}\, \widetilde{\delta} \widetilde{D^*} =
&\\
& \displaystyle
a\, \widetilde{D}\, \widetilde{\delta}\, \widetilde{D^*} + \widetilde{D \delta D^*}(a) - a\, \widetilde{D}\, \widetilde{\delta}\, \widetilde{D^*}  = \widetilde{D \delta D^*}(a)
\end{eqnarray*}
and so by Lemma \ref{lema:tilde-map}, c), we deduce that $D\, \delta\, D^* \in \Der_k(A)[[\bfs]]_\Delta$. Actually, this result can be simply understood as the fact that the conjugation of any $k[[\bfs]]_\Delta$-derivation of $A[[\bfs]]_\Delta$ by any automorphism of the $k[[\bfs]]_\Delta$-algebra $A[[\bfs]]_\Delta$ is again a $k[[\bfs]]_\Delta$-derivation.
\medskip

\noindent 2) For each $\delta\in\Der_k(A)$ we have $\fAd^p_\Delta(D)(\delta) = \sum_\alpha \fAd^p_\Delta(D)_\alpha(\delta)\, \bfs^\alpha$ with 
$$\fAd^p_\Delta(D)_\alpha(\delta) =\sum_{\scriptscriptstyle \alpha'+\alpha''=\alpha} D_{\alpha'}\, \delta\, D^*_{\alpha''},
$$
and so $\fAd^p_\Delta(D)_0 =\Id$ and $\fAd^p_\Delta(D) \in \Aut^{\circ}_{k[[\bfs]]_\Delta}(\Der_k(A)[[\bfs]]_\Delta)$.
\medskip

\noindent (i) Since the $\fAd^p_\Delta$ are defined as a conjugation, they are group homomorphisms.
\medskip

\noindent (ii) For any $D\in\HS^p_k(A;\Delta)$, for any 
$a\in A[[\bfs]]_\Delta$  and for any $\delta\in\Der_k(A)[[\bfs]]_\Delta$  we have
$$\left(\fAd^p_\Delta(D)\, a\right)(\delta) = D\, a\, \delta\, D^* = \widetilde{D}(a)\, D\, \delta\, D^*= \widetilde{D}(a) \fAd^p_\Delta(D)(\delta).
$$

\noindent (iii) Let $\varphi\in\Sub_k(p,q;\Delta,\nabla)$ be a substitution map with constant coefficients and $D\in \HS^p_k(A;\Delta)$ a HS-derivation. Let us denote $E:= \varphi \bullet D$. We know from  \ref{nume:def-sbullet-HS} that:
$$
 E_e = \sum_{\substack{\scriptscriptstyle \alpha\in\Delta\\ \scriptscriptstyle  |\alpha|\leq |e| }} {\bf C}_e(\varphi,\alpha) D_\alpha,\quad \forall e\in \N^q, e \neq 0\quad (E_0=\Id)
$$
and $E^* = \varphi \bullet D^*$. So, for each $\varepsilon\in \nabla$ and for each $\delta\in\Der_k(A)$ we have:
\begin{eqnarray*}
&\displaystyle
\fAd^p_\Delta(\varphi \bullet D)_\varepsilon(\delta) = \sum_{\scriptscriptstyle e + f = \varepsilon} E_e \, \delta \, E^*_f = \sum_{\substack{\scriptscriptstyle e + f = \varepsilon\\ \scriptscriptstyle \alpha\in \Delta, \gamma\in \Delta \\ \scriptscriptstyle |\alpha|\leq |e|, |\gamma|\leq |f| }} {\bf C}_e(\varphi,\alpha) {\bf C}_f(\varphi,\gamma) D_\alpha\, \delta\, D^*_\gamma =&\\
&\displaystyle \sum_{\substack{\scriptscriptstyle a\in\Delta\\ \scriptscriptstyle |a|\leq |\varepsilon|}}\sum_{\substack{\scriptscriptstyle \alpha, \gamma \in\Delta\\ \scriptscriptstyle \alpha+\gamma=a}} \sum_{\substack{\scriptscriptstyle e + f = \varepsilon\\  \scriptscriptstyle |\alpha|\leq |e|, |\gamma|\leq |f| }} {\bf C}_e(\varphi,\alpha) {\bf C}_f(\varphi,\gamma) D_\alpha\, \delta\, D^*_\gamma 
\stackrel{(\star)}{=}
&\\
&\displaystyle
 \sum_{\substack{\scriptscriptstyle a\in\Delta\\ \scriptscriptstyle |a|\leq |\varepsilon|}}\sum_{\substack{\scriptscriptstyle \alpha, \gamma \in\Delta\\ \scriptscriptstyle \alpha+\gamma=a}}  {\bf C}_\varepsilon(\varphi,a) D_\alpha\, \delta\, D^*_\gamma = \sum_{\substack{\scriptscriptstyle a\in\Delta\\ \scriptscriptstyle |a|\leq |\varepsilon|}} {\bf C}_\varepsilon(\varphi,a) \left(
 \sum_{\substack{\scriptscriptstyle \alpha, \gamma \in\Delta\\ \scriptscriptstyle \alpha+\gamma=a}}  D_\alpha\, \delta\, D^*_\gamma\right) =
&\\
&\displaystyle 
  \sum_{\substack{\scriptscriptstyle a\in\Delta\\ \scriptscriptstyle |a|\leq |\varepsilon|}} {\bf C}_\varepsilon(\varphi,a) \fAd^p_\Delta(D)_a(\delta) = \left( \varphi\sbullet \fAd^p_\Delta(D)\right)_\varepsilon (\delta),
\end{eqnarray*}
where the equality $(\star)$ comes from the fact that $\varphi$ is an $A$-algebra map (see \cite[Proposition 3]{nar_subst_2018}).
\end{proof}

\begin{nota} With the notations of the above proposition, for each $\alpha \in\Delta$ with $|\alpha|=1$, the map 
$\fAd^p_\Delta(D)_\alpha:\Der_k(A)\to \Der_k(A)$ coincides with the classical adjoint representation
$$\Ad_{D_\alpha}: \delta\in \Der_k(A)\longmapsto [D_\alpha,\delta]\in \Der_k(A)$$ associated with the derivation $D_\alpha$.
\end{nota}

It is clear that left (resp. right) (pre-)HS-modules with HS-maps form an abelian category admitting a conservative additive exact functor (the forgetful functor) to the category of $A$-modules.
\medskip

\subsection{Operations on Hasse--Schmidt modules}

In this section, starting with two left (pre-)HS-modules  $(E,\Uppsib)$, $(F,\Uppsibb)$ over $A/k$ and two right (pre-)HS-modules $(P,\Upgammab)$, $(Q,\Upgammabb)$  over $A/k$, we will see how to construct natural left \mbox{(pre-)}HS-modules structures on $E\otimes_A F$, $\Hom_A(E,F)$, $\Hom_A(P,Q)$ and right (pre-)HS-modules structures on $P\otimes_A E$, $\Hom_A(E,P)$. Let us notice that similar constructions have been studied in \cite[\S 2.2]{matzat_vdput_2003} in the particular case of iterative uni-variate Hasse--Schmidt derivations over a field.

\begin{prop} \label{prop:tensor-left-HS-mod}
Under the above hypotheses, the following properties hold:
\begin{enumerate}
\item[(1)]
For any $p\in \N$, for any $\Delta\in\coide{\N^p}$ and for any $D\in \HS^p_k(A;\Delta)$ there is a unique $\Uppsi^p_\Delta(D) \in \Aut^{\circ}_{k[[\bfs]]_\Delta}((E\otimes_A F)[[\bfs]]_\Delta)$ such
that the following diagram is commutative:
$$
\begin{tikzcd}
E[[\bfs]]_\Delta \otimes_{k[[\bfs]]_\Delta} F[[\bfs]]_\Delta \ar[r, "\mu"]
\ar[d,"\Uppsib^p_\Delta(D)\otimes \Uppsibb^p_\Delta(D)"'] &
(E\otimes_A F)[[\bfs]]_\Delta \ar[d, "\Uppsi^p_\Delta(D)"]\\
E[[\bfs]]_\Delta \otimes_{k[[\bfs]]_\Delta} F[[\bfs]]_\Delta \ar[r, "\mu"] &
(E\otimes_A F)[[\bfs]]_\Delta,
\end{tikzcd}
$$
where $\mu$ is the natural $(A[[\bfs]]_\Delta;A[[\bfs]]_\Delta)$-linear map
$$ \mu\left( \left(\sum_\alpha e_\alpha \bfs^\alpha\right) \otimes \left(\sum_\alpha f_\alpha \bfs^\alpha\right)\right) = \sum_\alpha \left( \sum_{\alpha'+\alpha''=\alpha} e_{\alpha'} \otimes f_{\alpha''} \right) \bfs^\alpha.
$$
\item[(2)] The system $\Uppsi = \{\Uppsi^p_\Delta, p\in \N, \Delta\in\coide{\N^p}\}$ defines a left (pre-)HS-module structure over $A/k$ on $E\otimes_A F$.
\end{enumerate}
\end{prop}

\begin{proof} (1) Since we have canonical isomorphisms $E[[\bfs]]_\Delta \otimes_{A[[\bfs]]_\Delta} F[[\bfs]]_\Delta \simeq (E\otimes_A F)[[\bfs]]_\Delta$, the result comes from the following equality:
\begin{eqnarray*}
& \displaystyle 
\mu \left( \left( \Uppsib^p_\Delta(D)\otimes \Uppsibb^p_\Delta(D) \right) ( (ae) \otimes f)   \right) =
\mu \left( \Uppsib^p_\Delta(D)(ae)\otimes \Uppsibb^p_\Delta(D)(f) \right)  = 
&\\
& \displaystyle 
\mu \left( \left(\widetilde{D}(a)\Uppsib^p_\Delta(D)(e)\right)\otimes \Uppsibb^p_\Delta(D)(f) \right) =
\mu \left( \Uppsib^p_\Delta(D)(e)\otimes \left(\widetilde{D}(a)\Uppsibb^p_\Delta(D)(f)\right) \right) =
&\\
& \displaystyle 
\mu \left( \Uppsib^p_\Delta(D)(e)\otimes \Uppsibb^p_\Delta(D)(af)\right) =
\mu \left( \left( \Uppsib^p_\Delta(D)\otimes \Uppsibb^p_\Delta(D) \right) ( e \otimes (af))   \right)
\end{eqnarray*}
for all $e\in E[[\bfs]]_\Delta $, for all $f\in F[[\bfs]]_\Delta $ and for all $a\in A[[\bfs]]_\Delta$.
\medskip

\noindent (2) We have to check properties (i), (ii) and (iii) of Remark \ref{nota:Hsmod-equiv} (1). Property (i) is clear from the uniqueness of $\Uppsi^p_\Delta(D)$ in part (1). Property (ii) follows from 
\begin{eqnarray*}
& \displaystyle \left(\Uppsi^p_\Delta(D)\, a\right)(\mu(e\otimes f)) = \Uppsi^p_\Delta(D)(\mu((ae)\otimes f)) =
&
\\
& \displaystyle 
\mu \left( \Uppsib^p_\Delta(D)(ae)\otimes \Uppsibb^p_\Delta(D)(f) \right)  = 
\mu \left( \left(\widetilde{D}(a)\Uppsib^p_\Delta(D)(e)\right)\otimes \Uppsibb^p_\Delta(D)(f) \right) =
&
\\
& \displaystyle 
\widetilde{D}(a)\, \mu \left( \Uppsib^p_\Delta(D)(e)\otimes \Uppsibb^p_\Delta(D)(f) \right) =
\widetilde{D}(a)\, \Uppsi^p_\Delta(D)(\mu(e\otimes f))
\end{eqnarray*}
for all $e\in E[[\bfs]]_\Delta$, for all $f\in F[[\bfs]]_\Delta$ and for all $a\in A[[\bfs]]_\Delta$.
Property (iii) follows from (\ref{eq:commut-CD-varphi-bullet}) and the commutativity of the following diagram:
$$
\begin{tikzcd}
E[[\bfs]]_\Delta \otimes_{k[[\bfs]]_\Delta} F[[\bfs]]_\Delta \ar[r, "\mu"]
\ar[d,"\varphi_E \otimes \varphi_F"'] &
(E\otimes_A F)[[\bfs]]_\Delta \ar[d, "\varphi_{E\otimes_A F}"]\\
E[[\bft]]_\nabla \otimes_{k[[\bft]]_\nabla} F[[\bft]]_\nabla \ar[r, "\mu"] &
(E\otimes_A F)[[\bft]]_\nabla
\end{tikzcd}
$$
for each substitution map $\varphi \in \Sub_A(p,q;\Delta,\nabla)$ (resp. $\varphi\in\Sub_k(p,q;\Delta,\nabla)$).
\end{proof}

For any maps $f:E[[\bfs]]_\Delta \to E[[\bfs]]_\Delta$, $g:F[[\bfs]]_\Delta \to F[[\bfs]]_\Delta$ and $h:E[[\bfs]]_\Delta \to F[[\bfs]]_\Delta$, let us denote:
$$ f^\star(h) := h\pcirc f,\quad g_\star(h) := g\pcirc h.
$$

\begin{prop} \label{prop:hom-left-HS-mod} 
Under the above hypotheses, the following properties hold:
\begin{enumerate}
\item[(1)]
For any $p\in \N$, for any  $\Delta\in\coide{\N^p}$ and for any $D\in \HS^p_k(A;\Delta)$ there is a unique $\Uppsi^p_\Delta(D) \in \Aut^{\circ}_{k[[\bfs]]_\Delta}\left(\Hom_A(E,F)[[\bfs]]_\Delta\right)$ such
that the following diagram is commutative:
$$
\begin{tikzcd}
\Hom_A(E,F)[[\bfs]]_\Delta
\ar[r, "\nu"]  \ar[d, "\Uppsi^p_\Delta(D)"'] & \Hom_{k[[\bfs]]_\Delta}(E[[\bfs]]_\Delta,F[[\bfs]]_\Delta)
\ar[d, "\Uppsibb^p_\Delta(D)_\star \pcirc \Uppsib^p_\Delta(D^*)^\star"]
 \\
\Hom_A(E,F)[[\bfs]]_\Delta \ar[r, "\nu"] & 
\Hom_{k[[\bfs]]_\Delta}(E[[\bfs]]_\Delta,F[[\bfs]]_\Delta),
\end{tikzcd}
$$
where $\nu$ is the natural $(A[[\bfs]]_\Delta;A[[\bfs]]_\Delta)$-linear map defined as $\nu(h) = \widetilde{h}$ (see (\ref{eq:tilde-map})).
\item[(2)] The system $\Uppsi = \{\Uppsi^p_\Delta, p\in \N, \Delta\in\coide{\N^p}\}$ defines a left (pre-)HS-module structure over $A/k$ on $\Hom_A(E,F)$.
\end{enumerate}
\end{prop}

\begin{proof} (1) Since we have canonical isomorphisms 
$$h \in \Hom_A(E,F)[[\bfs]]_\Delta \stackrel{\sim}{\longmapsto} \widetilde{h} \in \Hom_{A[[\bfs]]_\Delta}(E[[\bfs]]_\Delta,F[[\bfs]]_\Delta),
$$
the result comes from the fact that $\left(\Uppsibb^p_\Delta(D)_\star \pcirc \Uppsib^p_\Delta(D^*)^\star\right)(h')$ is $A[[\bfs]]_\Delta$-linear for each $h'\in \Hom_{A[[\bfs]]_\Delta}(E[[\bfs]]_\Delta,F[[\bfs]]_\Delta)$, namely:
\begin{eqnarray*}
& \displaystyle \left(\Uppsibb^p_\Delta(D)_\star \pcirc \Uppsib^p_\Delta(D^*)^\star\right)(h')(am) = \left( \Uppsibb^p_\Delta(D) \pcirc h' \pcirc \Uppsib^p_\Delta(D^*)\right)(am)=
&
\\
& \displaystyle  \Uppsibb^p_\Delta(D) \left( h' \left( \widetilde{D^*}(a)\, \Uppsib^p_\Delta(D^*)(m)\right)   \right)   =
\Uppsibb^p_\Delta(D) \left(\widetilde{D^*}(a)\, h' \left(  \Uppsib^p_\Delta(D^*)(m)\right)   \right) =
&
\\
& \displaystyle 
\widetilde{D}(\widetilde{D^*}(a)) \Uppsibb^p_\Delta(D) \left(h' \left(  \Uppsib^p_\Delta(D^*)(m)\right)   \right) =
a \left(\Uppsibb^p_\Delta(D)_\star \pcirc \Uppsib^p_\Delta(D^*)^\star\right)(h')(m)
\end{eqnarray*}
for all $m\in E[[\bfs]]_\Delta $ and for all $a\in A[[\bfs]]_\Delta$.
\medskip

\noindent (2) As in Proposition \ref{prop:tensor-left-HS-mod}, we have to check properties (i), (ii) and (iii) of Remark \ref{nota:Hsmod-equiv} (1). Property (i) comes from the fact that the map
\begin{eqnarray*}
& \displaystyle 
D\in \HS^p_k(A;\Delta) \longmapsto 
&
\\
& \displaystyle 
\Uppsibb^p_\Delta(D)_\star \pcirc \Uppsib^p_\Delta(D^*)^\star \in
\Aut_{k[[\bfs]]_\Delta}\left( \Hom_{k[[\bfs]]_\Delta}(E[[\bfs]]_\Delta,F[[\bfs]]_\Delta) \right)
\end{eqnarray*}
is a group homomorphism:
\begin{eqnarray*}
& \displaystyle 
 \Uppsibb^p_\Delta(D\pcirc E)_\star \pcirc \Uppsib^p_\Delta((D\pcirc E)^*)^\star = \cdots = 
&
\\
& \displaystyle 
\Uppsibb^p_\Delta(D)_\star \pcirc \Uppsibb^p_\Delta(E)_\star \pcirc \Uppsib^p_\Delta(D^*)^\star \pcirc \Uppsib^p_\Delta(E^*)^\star =
&
\\
& \displaystyle  \Uppsibb^p_\Delta(D)_\star \pcirc \Uppsib^p_\Delta(D^*)^\star \pcirc \Uppsibb^p_\Delta(E)_\star  \pcirc \Uppsib^p_\Delta(E^*)^\star.
\end{eqnarray*}
Property (ii) follows from the following equality:
\begin{eqnarray*}
& \displaystyle 
 \left( \Uppsibb^p_\Delta(D)_\star \pcirc \Uppsib^p_\Delta(D^*)^\star \right) (ah') = \Uppsibb^p_\Delta(D) \pcirc (ah') \pcirc \Uppsib^p_\Delta(D^*) = 
 &
\\
& \displaystyle 
 \left( \Uppsibb^p_\Delta(D) \, a\right) \pcirc h' \pcirc \Uppsib^p_\Delta(D^*) =
\left( \widetilde{D}(a) \Uppsibb^p_\Delta(D)\right) \pcirc h' \pcirc \Uppsib^p_\Delta(D^*) = 
&
\\
& \displaystyle 
\widetilde{D}(a) \left( \Uppsibb^p_\Delta(D)_\star \pcirc \Uppsib^p_\Delta(D^*)^\star \right) (h')
&
\end{eqnarray*}
for all $h'\in \Hom_{A[[\bfs]]_\Delta}(E[[\bfs]]_\Delta,F[[\bfs]]_\Delta)$ and for all $a\in A[[\bfs]]_\Delta$.
\medskip

To finish, let us prove property (iii). Let us write $M=\Hom_A(E,F)$.
It is enough to prove that $\Uppsi^q_\nabla(\varphi\sbullet D)|_M = \left(\varphi\sbullet  \Uppsi^p_\Delta(D)\right)|_M$ for all $p,q\in \N$, for all $\Delta\subset\N^p, \nabla\in\coide{\N^q}$, for all substitution map   $\varphi\in\Sub_A(p,q;\Delta,\nabla)$ (resp. $\varphi\in\Sub_k(p,q;\Delta,\nabla)$) and for all HS-derivation $D\in\HS_k^p(A;\Delta)$. For each $h\in M$ we have $\nu(h)=\widetilde{h} = \overline{h}$ with $\overline{h}\left(\sum_\beta e_\alpha \bft^\beta\right) = \sum_\beta h(e_\beta) \bft^\beta$ for each $\sum_\beta e_\beta \bft^\beta \in E[[\bft]]_\nabla$. So:
\begin{eqnarray*}
& \displaystyle  \left( \nu \pcirc \Uppsi^q_\nabla(\varphi\sbullet D)\right)(h)|_E = 
\left[ \Uppsibb^q_\nabla(\varphi\sbullet D) \pcirc \nu(h) \pcirc \Uppsib^q_\nabla((\varphi\sbullet D)^*)\right]|_E  \stackrel{\text{\tiny (1)}}{=}
&
\\
& \displaystyle 
\left( \varphi\sbullet\Uppsibb^p_\Delta(D) \right)\pcirc \widetilde{h} \pcirc \left[\Uppsib^q_\nabla(\varphi^D\sbullet D^*)|_E\right]= 
\left( \varphi\sbullet\Uppsibb^p_\Delta(D) \right)\pcirc \widetilde{h} \pcirc \left[\left(\varphi^D\sbullet \Uppsib^p_\Delta(D^*)\right)|_E\right] \stackrel{\text{\tiny (2)}}{=}
&
\\
& \displaystyle 
\left( \varphi\sbullet\Uppsibb^p_\Delta(D) \right)\pcirc \overline{h} \pcirc 
\left[\left(\varphi^D\right)_E\pcirc \left(\Uppsib^p_\Delta(D^*)|_E\right)\right]=
&
\\
& \displaystyle 
\left( \varphi\sbullet\Uppsibb^p_\Delta(D) \right)\pcirc 
\left(\varphi^D\right)_F\pcirc \overline{h}\pcirc \left(\Uppsib^p_\Delta(D^*)|_E\right) \stackrel{\text{\tiny (3)}}{=}
\varphi_F\pcirc \Uppsibb^p_\Delta(D) \pcirc \nu(h) \pcirc \left(\Uppsib^p_\Delta(D^*)|_E \right)=
&
\\
& \displaystyle 
\varphi_F\pcirc 
\left[  
\left(\nu \pcirc \Uppsi^p_\Delta(D)\right)(h)|_E
\right] =
\varphi_F\pcirc 
\left[  
\nu \left(\Uppsi^p_\Delta(D)(h)\right)|_E
\right] \stackrel{\text{\tiny (4)}}{=}
\nu\left( \varphi_M \left(\Uppsi^p_\Delta(D)(h)\right)\right)|_E =
&
\\
& \displaystyle 
\nu \left( \left(\varphi_M  \pcirc \Uppsi^p_\Delta(D)\right)(h)
\right)|_E = 
\nu \left( \left(\varphi \sbullet \Uppsi^p_\Delta(D)\right)(h)
\right)|_E = 
\left( \nu \pcirc \left(\varphi \sbullet \Uppsi^p_\Delta(D)\right) \right)(h)|_E,
\end{eqnarray*}
where equality (1) comes from Proposition \ref{prop:varphi-D-main}, equality (2) comes from (\ref{eq:atencion-bullet}), equality (3) comes from Proposition \ref{prop:bullet-D-elements}, (c), and equality (4) comes from (\ref{eq:tilde-varphi-M-varphi-F}). 
We first deduce that $\left( \nu \pcirc \Uppsi^q_\nabla(\varphi\sbullet D)\right)(h) = \left( \nu \pcirc \left(\varphi \sbullet \Uppsi^p_\Delta(D)\right) \right)(h)$ for all $h\in M$, i.e. 
$$  \nu \pcirc \left(\Uppsi^q_\nabla(\varphi\sbullet D)|_M \right)= \nu \pcirc \left( \left(\varphi \sbullet \Uppsi^p_\Delta(D)\right)|_M \right),
$$
second, from the injectivity of $\nu$,  that $ \Uppsi^q_\nabla(\varphi\sbullet D)|_M =   \left(\varphi \sbullet \Uppsi^p_\Delta(D)\right)|_M$, and we conclude that $ \Uppsi^q_\nabla(\varphi\sbullet D) =   \varphi \sbullet \Uppsi^p_\Delta(D)$.
\end{proof}

The proofs of the following three propositions are completely similar to the proofs of Propositions \ref{prop:hom-left-HS-mod} and \ref{prop:tensor-left-HS-mod}.

\begin{prop} \label{prop:tensor-right-left-HS-mod}
Under the above hypotheses, the following properties hold:
\begin{enumerate}
\item[(1)]
For any $p\in \N$, for any $\Delta\in\coide{\N^p}$ and for any $D\in \HS^p_k(A;\Delta)$ there is a unique $\Upgamma^p_\Delta(D) \in \Aut^{\circ}_{k[[\bfs]]_\Delta}((P\otimes_A E)[[\bfs]]_\Delta)$ such
that the following diagram is commutative:
$$
\begin{tikzcd}
P[[\bfs]]_\Delta \otimes_{k[[\bfs]]_\Delta} E[[\bfs]]_\Delta \ar[r, "\mu"]
\ar[d,"\Upgammab^p_\Delta(D)\otimes \Uppsib^p_\Delta(D^*)"'] &
(P\otimes_A E)[[\bfs]]_\Delta \ar[d, "\Upgamma^p_\Delta(D)"]\\
P[[\bfs]]_\Delta \otimes_{k[[\bfs]]_\Delta} E[[\bfs]]_\Delta \ar[r, "\mu"] &
(P\otimes_A E)[[\bfs]]_\Delta,
\end{tikzcd}
$$
where $\mu$ is the natural $(A[[\bfs]]_\Delta;A[[\bfs]]_\Delta)$-linear map
$$ \mu\left( \left(\sum_\alpha p_\alpha \bfs^\alpha\right) \otimes \left(\sum_\alpha e_\alpha \bfs^\alpha\right)\right) = \sum_\alpha \left( \sum_{\alpha'+\alpha''=\alpha} p_{\alpha'} \otimes e_{\alpha''} \right) \bfs^\alpha.
$$
\item[(2)] The system $\Upgamma = \{\Upgamma^p_\Delta, p\in \N, \Delta\in\coide{\N^p}\}$ defines a right (pre-)HS-module structure over $A/k$ on $P\otimes_A E$.
\end{enumerate}
\end{prop}

\begin{prop} \label{prop:hom-right-right-HS-mod}
Under the above hypotheses, the following properties hold:
\begin{enumerate}
\item[(1)]
For any $p\in \N$, for any $\Delta\in\coide{\N^p}$ and for any $D\in \HS^p_k(A;\Delta)$ there is a unique $\Uppsi^p_\Delta(D) \in \Aut^{\circ}_{k[[\bfs]]_\Delta}\left(\Hom_A(P,Q)[[\bfs]]_\Delta\right)$ such
that the following diagram is commutative:
$$
\begin{tikzcd}
\Hom_A(P,Q)[[\bfs]]_\Delta
\ar[r, "\nu"]  \ar[d, "\Uppsi^p_\Delta(D)"'] & \Hom_{k[[\bfs]]_\Delta}(P[[\bfs]]_\Delta,Q[[\bfs]]_\Delta)
\ar[d, "\Upgammabb^p_\Delta(D^*)_\star \pcirc \Upgammab^p_\Delta(D)^\star"]
 \\
\Hom_A(P,Q)[[\bfs]]_\Delta \ar[r, "\nu"] & 
\Hom_{k[[\bfs]]_\Delta}(P[[\bfs]]_\Delta,Q[[\bfs]]_\Delta),
\end{tikzcd}
$$
where $\nu$ is the natural $(A[[\bfs]]_\Delta;A[[\bfs]]_\Delta)$-linear map defined as $\nu(h) = \widetilde{h}$ (see (\ref{eq:tilde-map})).
\item[(2)] The system $\Uppsi = \{\Uppsi^p_\Delta, p\in \N, \Delta\in\coide{\N^p}\}$ defines a left (pre-)HS-module structure over $A/k$ on $\Hom_A(P,Q)$.
\end{enumerate}
\end{prop}

\begin{prop} \label{prop:hom-left-right-HS-mod}
Under the above hypotheses, the following properties hold:
\begin{enumerate}
\item[(1)]
For any $p\in \N$, for any $\Delta\in\coide{\N^p}$ and for any $D\in \HS^p_k(A;\Delta)$ there is a unique $\Upgamma^p_\Delta(D) \in \Aut^{\circ}_{k[[\bfs]]_\Delta}\left(\Hom_A(E,P)[[\bfs]]_\Delta\right)$ such
that the following diagram is commutative:
$$
\begin{tikzcd}
\Hom_A(E,P)[[\bfs]]_\Delta
\ar[r, "\nu"]  \ar[d, "\Upgamma^p_\Delta(D)"'] & \Hom_{k[[\bfs]]_\Delta}(E[[\bfs]]_\Delta,P[[\bfs]]_\Delta)
\ar[d, "\Upgammab^p_\Delta(D)_* \pcirc \Uppsib^p_\Delta(D)^\star = \Uppsib^p_\Delta(D)^\star \pcirc \Upgammab^p_\Delta(D)_*"]
 \\
\Hom_A(E,P)[[\bfs]]_\Delta \ar[r, "\nu"] & 
\Hom_{k[[\bfs]]_\Delta}(E[[\bfs]]_\Delta,P[[\bfs]]_\Delta),
\end{tikzcd}
$$
where $\nu$ is the natural $(A[[\bfs]]_\Delta;A[[\bfs]]_\Delta)$-linear map defined as $\nu(h) = \widetilde{h}$ (see (\ref{eq:tilde-map})).
\item[(2)] The system $\Upgamma = \{\Upgamma^p_\Delta, p\in \N, \Delta\in\coide{\N^p}\}$ defines a right (pre-)HS-module structure over $A/k$ on $\Hom_A(E,P)$.
\end{enumerate}
\end{prop}

The following proposition easily follows from Proposition \ref{prop:tensor-left-HS-mod} and its proof is left to the reader.

\begin{prop} \label{prop:HS-rep-sym-altern}
Under the above hypotheses, the left (pre-)HS-module structure over $A/k$ on $E^{\otimes d} = E\otimes_A E\otimes_A \cdots \otimes_A E$ defined in Proposition \ref{prop:tensor-left-HS-mod} induces:
\begin{enumerate}
\item[1)] A unique (pre-)HS-module structure over $A/k$ on $\Sim_A^d E$ such that the natural map $E^{\otimes d} \to \Sim_A^d E$ is a HS-map.
\item[2)] A unique (pre-)HS-module structure over $A/k$ on $\bigwedge_A^d E$ such that the natural map $E^{\otimes d} \to \bigwedge_A^d E$ is a HS-map.
\end{enumerate}
\end{prop}

\subsection{The enveloping algebra of Hasse--Schmidt derivations}

Let $\dT_{A/k}$ be the free $k$-algebra
$$
\dT_{A/k} := k \langle S_a, T_{p,\Delta,D,\alpha}; a\in A,
p \in \N, \Delta\in\coide{\N^p},
 \alpha \in \Delta,
D\in \HS^p_k(A;\Delta) \rangle
$$ 
and let us consider the two-sided ideal $\dI \subset \dT_{A/k}$ with generators:
\begin{itemize}
\item[(0)] $S_{c1} -c$, $S_{a+a'}-S_a-S_{a'}$, $S_{aa'} - S_a S_{a'}$,
\item[(i)] $ T_{p,\{0\},\mathbb{I},0} - 1$,
\item[(ii)] $ T_{p,\Delta,\mathbb{I},\alpha}$ for $|\alpha|> 0$\footnote{Actually, generators (ii) can be avoided since they are deduced from generators (i) and (iii).},
\item[(iii)] $\displaystyle T_{p,\Delta,D\, \circ\, E,\alpha} - \sum_{\scriptscriptstyle \beta + \gamma = \alpha}
T_{p,\Delta,D,\beta}\, T_{p,\Delta,E,\gamma}$,
\item[(iv)] $\displaystyle T_{p,\Delta,D,\alpha}\, S_a - \sum_{\scriptscriptstyle \beta+\gamma=\alpha} S_{D_\beta(a)} T_{p,\Delta,D,\gamma}$,
\item[(v)] $\displaystyle T_{q,\nabla,\varphi\bullet D,\beta} - \sum_{\substack{\scriptscriptstyle \alpha \in\Delta  \\ \scriptscriptstyle |\alpha|\leq |\beta|  }} S_{{\bf C}_\beta(\varphi,\alpha)} T_{p,\Delta,D,\alpha}$,

\end{itemize}
for 
$c\in k$, $a,a'\in A$, $p,q \in \N$, $\Delta\subset \N^p, \nabla\in\coide{\N^q}$, $\alpha \in\Delta$, $\beta\in\nabla$, $D,E\in \HS^p_k(A;\Delta)$ and
 $\varphi \in \Sub_A(p,q;\Delta,\nabla) $.
\medskip

\noindent
We consider the $\N$-grading in $\dT_{A/k}$ given by (see Definition \ref{def:ell-new-alpha}): 
$$\textstyle
\deg (k) = 0,\ \deg (S_a) = 0,\ \deg \left(T_{p,\Delta,D,\alpha}\right) = \lfloor \frac{|\alpha|}{ \ell_\alpha(D)}\rfloor
$$
for $a\in A$, $p\in \N$, $\Delta\in\coide{\N^p}$,  $\alpha \in\Delta$ and $D\in \HS^p_k(A;\Delta)$. This grading is motivated by Proposition \ref{prop:ell-new-alpha}. 
Let us notice that
$$ \deg \left(T_{p,\Delta,D,\alpha}\right) = \deg \left(T_{p,\fn_\alpha,\tau_{\Delta,\fn_\alpha}(D),\alpha}\right).
$$
We will denote $\dT_{A/k}^d$ the homogeneous component of degree $d$ and $\dT_{A/k}^{\leq d}:= \bigoplus_{e\leq d} \dT_{A/k}^e$.
\medskip

Let us call $\dU_{A/k} := \dT_{A/k}/\dI$ and write ${\bf S}_a := S_a + \dI, {\bf T}_{p,\Delta,D,\alpha} := T_{p,\Delta,D,\alpha} + \dI$ for the generators of the $k$-algebra $\dU_{A/k}$.
The grading in $\dT_{A/k}$ induces a filtration on $\dU_{A/k}$ and let us also call $\deg: \dU_{A/k} \to \N$ the corresponding map:
$$ \deg (P) := \min \{\deg(p)\ |\ p\in \dT_{A/k}, P= p + \dI\}\quad \text{for\ }\ P \in \dU_{A/k}, P\neq 0,
$$
and $\deg(0)=-\infty$, 
with 
$ \dU_{A/k}^d = \{ P \in \dU_{A/k}\ |\ \deg(P)\leq d\} = \dT_{A/k}^{\leq d}/\left(\dI\cap \dT_{A/k}^{\leq d}\right)$.
\medskip

The generators of type (0) of $\dI$ give rise to a natural $k$-algebra map $a\in A \mapsto {\bf S}_a \in \dU_{A/k}$ and so $\dU_{A/k}$ is a $k$-algebra over $A$.
\bigskip

\numero \label{nume:collect}
We first collect some direct consequences of the above definitions. For $p\in\N$, $\bfs=\{s_1,\dots,s_p\}$, $\Delta\in\coide{\N^p}$, $\alpha\in\Delta$ and $D\in \HS^p_k(A;\Delta)$ we have:
\begin{enumerate}
\item[(a)] Since the quotient map $\pi:A[[\bfs]]_\Delta \to A[[\bfs]]_{\fn_\alpha}$ is a substitution map (actually, a truncation map) and the action
$$ \pi\sbullet (-): \HS^p_k(A;\Delta) \longrightarrow  \HS^p_k(A;\fn_\alpha)
$$
coincides with the truncation $\tau_{\Delta,\fn_\alpha}$ (see Lemma \ref{lemma:truncations-are-substitutions}), by using the generators of type (v) and the fact that ${\bf C}_\beta(\pi,\alpha) = \delta_{\alpha \beta}$, we obtain ${\bf T}_{p,\Delta,D,\alpha} = {\bf T}_{p,\fn_\alpha,\tau_{\Delta,\fn_\alpha}(D),\alpha}$ (remember that $\deg \left(T_{p,\Delta,D,\alpha}\right) = \deg \left(T_{p,\fn_\alpha,\tau_{\Delta,\fn_\alpha}(D),\alpha}\right)$).
\item[(b)] From (a) and the generators of type (i) of $\dI$ we deduce:
${\bf T}_{p,\Delta,D,0} = {\bf T}_{p,\{0\},\tau_{\Delta,\{0\}}(D),0} = 1$.
\item[(c)]  If $0<|\alpha| < \ell_\alpha(D)$, then $\tau_{\Delta,\fn_\alpha}(D)=\mathbb{I}$ and so from (a) and the generators of type (ii) of $\dI$ we have
${\bf T}_{p,\Delta,D,\alpha} = {\bf T}_{p,\fn_\alpha,\mathbb{I},\alpha} = 0$.
\end{enumerate}

\begin{lemma} The term $\dU_{A/k}^0$ is the $k$-module generated by the ${\bf S}_a$, $a\in A$, and coincides with the image of the natural map $A \to \dU_{A/k}$.
\end{lemma}

\begin{proof} By definition, $\dU_{A/k}^0$ is the $k$-module generated by the monomials in the ${\bf S}_a$, $a\in A$, and the ${\bf T}_{p,\Delta,D,\alpha}$ with 
$$ \textstyle
\deg \left(T_{p,\Delta,D,\alpha}\right) = \lfloor \frac{|\alpha|}{\ell_\alpha(D)}\rfloor =0,
$$
i.e. $|\alpha| < \ell_\alpha(D)$. So, by (b) and (c) and the generators of type (0) of $\dI$ we deduce that
 $\dU_{A/k}^0$ is the $k$-module generated by the ${\bf S}_a$ and coincides with the image of $A \to \dU_{A/k}$.
\end{proof}

The proof of the following proposition is clear (see Proposition \ref{prop:ell-new}).

\begin{prop} \label{prop:bupupsilon}
There is a unique $k$-algebra map $\bupupsilon:\dU_{A/k} \longrightarrow \DD_{A/k}$ sending
$$ {\bf S}_a \longmapsto a,\quad {\bf T}_{p,\Delta,D,\alpha} \longmapsto D_\alpha.
$$
Moreover, it is filtered.
\end{prop}

\begin{cor} The natural map $A \to \dU_{A/k}$ is injective and $A\simeq \dU_{A/k}^0$.
\end{cor}

\begin{prop} \label{prop:HS-struc-dU}
The $k$-algebra $\dU_{A/k}$ over $A$ is endowed with a natural HS-structure $\Upupsilon$ over $A/k$. Moreover, the pair $(\dU_{A/k},\Upupsilon)$ is universal among HS-structures, i.e. for any $k$-algebra $R$ over $A$ and any HS-structure $\Uppsi$ on $R$ over $A/k$, there is a unique map $f:\dU_{A/k} \to R$ of $k$-algebras over $A$ such that $f\pcirc \Upupsilon = \Uppsi$.
\end{prop}

\begin{proof} We consider the system of maps $\Upupsilon$ given by:
$$ \Upupsilon^p_{\Delta}: D\in \HS^p_k(A;\Delta) \longmapsto \sum_{\scriptscriptstyle \alpha \in\Delta} {\bf T}_{p,\Delta,D,\alpha} \bfs^\alpha \in \U^p(\dU_{A/k};\Delta)
$$
for $p\in\N$, $\Delta\in\coide{\N^p}$. 
It is straightforward to see that properties in Definition \ref{def:HS-structure} hold for $\Upupsilon$. Namely, property 1) follows from the generators of type (i), (ii) and (iii) of $\dI$, property 2) follows from the generators of type (iv) of $\dI$, and finally the generators of type (v) of $\dI$ guarantee 
property 3).
\medskip

For the universal property, let $f_0: \dT_{A/k} \to R$ be the $k$-algebra map determined by
$$ f_0(S_a) = a 1,\ f_0(T_{p,\Delta,D,\alpha}) = \Uppsi^p_{\Delta}(D)_\alpha$$
for all $a\in A$, for all $p\in\N$, for all $\Delta\in\coide{\N^p}$, for all $\alpha\in\Delta$ and for all $D\in\HS^p_k(A;\Delta)$. 
It is clear that $f_0$ vanishes on $\dI$ and gives rise to our wanted map $f:\dU_{A/k} \to R$ of $k$-algebras over $A$. The uniqueness of $f$ is clear.
\end{proof}

Let us notice that the HS-structure $\Upupsilon$ in the above proposition is filtered.

\begin{cor} The abelian category of left (resp. right) HS-modules over $A/k$ is isomorphic to the category of left (resp. right) $\dU_{A/k}$-modules.
\end{cor}

\begin{defi}  \label{defi:enveloping_HS} The {\em enveloping algebra} of the Hasse--Schmidt derivations of $A$ over $k$ is the $k$-algebra $\dU_{A/k}= \dT_{A/k}/\dI$ defined above. It is a filtered $k$-algebra over $A$. 
\end{defi}

\begin{thm} The graded ring $\gr \dU_{A/k}$ is commutative.
\end{thm}

\begin{proof} We need to prove that the degree of the bracket of the classes in $\dU_{A/k}$ of any two variables generating $\dT_{A/k}$ is strictly less than the sum of the degrees of these variables. 
\medskip

\noindent -) 
For the variables $S_a$ the result is clear since ${\bf S}_a {\bf S}_{a'} - {\bf S}_{a'} {\bf S}_a = {\bf S}_{aa'} - {\bf S}_{a'a}=0$.
\medskip

\noindent -) 
Let us see the case of one variable $S_a$ and one variable $T_{p,\Delta,D,\alpha}$, with $a\in A$, $p\in\N$, $\Delta\in\coide{\N^p}$, $\alpha\in\Delta$ and $D\in\HS_k^p(A;\Delta)$, and set 
$\ell = \ell_\alpha(D)$.
\medskip

We know from (b) that ${\bf T}_{p,\Delta,D,0}=1$, and from (c) that whenever $0<|\alpha|<\ell$, then  ${\bf T}_{p,\Delta,D,\alpha}=0$, and of course $D_\alpha=0$. 
So, if $|\alpha|<\ell$ then ${\bf T}_{p,\Delta,D,\alpha} {\bf S}_a - {\bf S}_a {\bf T}_{p,\Delta,D,\alpha}=0$. Otherwise $|\alpha|\geq \ell$ and,
by using the generators of type (iv) of $\dI$, we have:
$$ 
 {\bf T}_{p,\Delta,D,\alpha}\, {\bf S}_a - {\bf S}_a \, {\bf T}_{p,\Delta,D,\alpha} = \sum_{\substack{\scriptscriptstyle \beta+\gamma=\alpha\\ \scriptscriptstyle |\beta|>0 }} {\bf S}_{D_\beta(a)} {\bf T}_{p,\Delta,D,\gamma} =
 \sum_{\substack{\scriptscriptstyle \beta+\gamma=\alpha\\ \scriptscriptstyle |\beta|\geq \ell }} {\bf S}_{D_\beta(a)} {\bf T}_{p,\Delta,D,\gamma}.
$$
We conclude that:
\begin{eqnarray*}
&  \deg \left(  {\bf T}_{p,\Delta,D,\alpha}\, {\bf S}_a - {\bf S}_a \, {\bf T}_{p,\Delta,D,\alpha} \right) \leq 
\max \left\{  \deg \left( T_{p,\Delta,D,\gamma}  \right) \ |\  \beta+\gamma=\alpha, |\beta|\geq \ell \right\}
\leq \\
&  
\max \left\{  \lfloor \frac{\scriptstyle |\gamma|}{\scriptstyle \ell_\gamma(D)}\rfloor\ |\  \gamma \leq \alpha, |\gamma| \leq |\alpha|-\ell \right\} \leq  \max \left\{  \lfloor \frac{\scriptstyle |\gamma|}{\scriptstyle \ell_\alpha(D)}\rfloor\ |\  \gamma \leq \alpha, |\gamma| \leq |\alpha|-\ell \right\} < &
 \\
 & 
{\textstyle \lfloor \frac{\scriptstyle |\alpha|}{\scriptstyle \ell}\rfloor} =
 \deg \left(  T_{p,\Delta,D,\alpha} \right) = \deg \left(  T_{p,\Delta,D,\alpha} \right) + \deg(S_a).
\end{eqnarray*}

\noindent -) 
It remains to treat the case of two variables $T_{p,\Delta,D,\alpha}$ and $T_{q,\nabla,E,\beta}$. 
We need to prove that:
\begin{equation} \label{eq:goal-commut-T}
\deg \left(  {\bf T}_{p,\Delta,D,\alpha}\, {\bf T}_{q,\nabla,E,\beta} - {\bf T}_{q,\nabla,E,\beta} \, {\bf T}_{p,\Delta,D,\alpha} \right) < \deg \left(  T_{p,\Delta,D,\alpha}\right) + \deg \left(  T_{q,\nabla,E,\beta}\right).
\end{equation}
From (b), we may assume $\alpha, \beta\neq 0$; by taking into account generators of $\dI$ of type (ii), we may assume $D,E\neq\mathbb{I}$; from (c), we may assume $\ell_\alpha(D)\leq |\alpha|$ and $\ell_\beta(E)\leq |\beta|$; and finally, from (a), we may assume that $\Delta=\fn_\alpha$ and $\nabla=\fn_\beta$.
Let us denote $\bfs=\{s_1,\dots,s_p\}$, $\bft=\{t_1,\dots,t_q\}$,
$$\iota:A[[\bfs]]_{\fn_\alpha} \to A[[\bfs\sqcup\bft]]_{\fn_\alpha\times \fn_\beta} = A[[\bfs\sqcup\bft]]_{\fn_{(\alpha,\beta)}},\  \kappa: A[[\bft]]_\nabla \to A[[\bfs\sqcup\bft]]_{\fn_{(\alpha,\beta)}}
$$ 
the combinatorial substitution maps given by the inclusions $\bfs,\bft\hookrightarrow \bfs\sqcup\bft$, 
 $F:=\iota\sbullet D$, $G:=\kappa\sbullet E$, $\ell_1:=\ell(D) =\ell_\alpha(D)$, $\ell_2:=\ell(E)=\ell_\beta(E)$. From Proposition \ref{prop:varphi-D-main} we have
$F^* =  \iota\sbullet D^*$ and $G^* = \kappa\sbullet E^*$.
\medskip

\noindent
We will proceed in several steps.
First, by using the generators of type (v) of $\dI$ and the fact that:
\begin{eqnarray*}
& \displaystyle 
{\bf C}_{(\gamma,\sigma)}(\iota,\alpha') = \left\{
\begin{array}{cl}
1 & \text{if\ }\  \gamma = \alpha'\  \text{\ and\ }\ \sigma = 0\\
0 & \text{otherwise},
\end{array}
\right.
&
\\
&
{\bf C}_{(\gamma,\sigma)}(\kappa,\beta') = \left\{
\begin{array}{cl}
1 & \text{if\ }\  \gamma = 0\  \text{\ and\ }\ \sigma = \beta'\\
0 & \text{otherwise},
\end{array}
\right.
\end{eqnarray*}
we deduce that:
\medskip

\item[(1)] $ {\bf T}_{p+q,\fn_{(\alpha,\beta)},F,(\alpha',0)} = {\bf T}_{p,\fn_\alpha,D,\alpha'}$,
${\bf T}_{p+q,\fn_{(\alpha,\beta)},G,(0,\beta')} = {\bf T}_{q,\fn_\beta,E,\beta'}$.
\medskip

\item[(2)]$ {\bf T}_{p+q,\fn_{(\alpha,\beta)},F,(\alpha',\beta')} = 0$ for $\beta'\neq 0$ and ${\bf T}_{p+q,\fn_{(\alpha,\beta)},G,(\alpha',\beta')}=0$ for $\alpha'\neq 0$.
\medskip

\item[(3)] $\ell_{(\alpha',0)}(F) = \ell_{\alpha'}(D)$, $\ell_{(0,\beta')}(G) = \ell_{\beta'}(E)$ (in particular,
$\ell(F)=\ell_{(\alpha,0)}(F) = \ell_{\alpha}(D)=\ell(D)=\ell_1$, $\ell(G)=\ell_{(0,\beta)}(G) = \ell_\beta(E)=\ell(E)=\ell_2$)
and 
\begin{eqnarray*}
& 
\deg \left( T_{p+q,\fn_{(\alpha,\beta)},F,(\alpha',0)}\right) = 
\lfloor \frac{|(\alpha',0)|}{\ell_{(\alpha',0)}(F)}\rfloor = 
\lfloor \frac{|\alpha'|}{\ell_{\alpha'}(D)}\rfloor = 
 \deg \left( T_{p,\fn_\alpha,D,\alpha'} \right),
&\\
& 
\deg \left( T_{p+q,\fn_{(\alpha,\beta)},G,(0,\beta')}\right) = 
\lfloor \frac{|(0,\beta')|}{\ell_{(0,\beta')}(G)}\rfloor = 
\lfloor \frac{|\beta'|}{\ell_{\beta'}(E)}\rfloor = 
 \deg \left( T_{q,\fn_\beta,E,\beta'} \right).
\end{eqnarray*} 

\item[(4)]  
From \ref{nume:properties-external-x} and the generators of type (iii) and (v) of $\dI$ we have:
\begin{eqnarray*}
& \displaystyle 
{\bf T}_{p+q,\fn_{(\alpha,\beta)},D\, \boxtimes\, E,(\alpha',\beta')} = {\bf T}_{p+q,\fn_{(\alpha,\beta)},F\, \circ\, G,(\alpha',\beta')} =
{\bf T}_{p,\fn_\alpha,D,\alpha'}\, {\bf T}_{q,\fn_\beta,E,\beta'},&
\\
& \displaystyle 
 {\bf T}_{p+q,\fn_{(\alpha,\beta)},E\, \boxtimes\, D,(\alpha',\beta')} =
{\bf T}_{p+q,\fn_{(\alpha,\beta)},G\, \circ\, F,(\alpha',\beta')} =
{\bf T}_{q,\fn_\beta,E,\beta'}\, {\bf T}_{p,\fn_\alpha,D,\alpha'}.
\end{eqnarray*}
Let us write
$H=[F,G]=F\pcirc G\pcirc F^* \pcirc G^*$. From Lemma \ref{lemma:ell-corchete} we know that $\ell(H) \geq \ell_1 + \ell_2$. 
Let us prove that:
\medskip

\item[(5)] ${\bf T}_{p+q,\fn_{(\alpha,\beta)},H,(\mu,\lambda)}=0$ whenever $(\mu,\lambda)\neq (0,0)$ and $|\mu|<\ell_1$ or $|\lambda|<\ell_2$.
\medskip

By using (1), (2)  and the generators of type (iii) of $\dI$ again, we obtain:
\begin{eqnarray*} 
&\displaystyle 
{\bf T}_{p+q,\fn_{(\alpha,\beta)},H,(\mu,\lambda)} =\cdots=
&\\
&\displaystyle 
\sum {\bf T}_{p+q,\fn_{(\alpha,\beta)},F,(\mu',0)}\, {\bf T}_{p+q,\fn_{(\alpha,\beta)},G,(0,\lambda')}\, {\bf T}_{p+q,\fn_{(\alpha,\beta)},F^*,(\mu'',0)} \, {\bf T}_{p+q,\fn_{(\alpha,\beta)},G^*,(0,\lambda'')}=
\end{eqnarray*}
\begin{eqnarray} \label{eq:aux-H}
&\displaystyle 
\sum {\bf T}_{p,\fn_\alpha,D,\mu'}\, {\bf T}_{q,\fn_\beta,E,\lambda'}\, {\bf T}_{p,\fn_\alpha,D^*,\mu''} \, {\bf T}_{q,\fn_\beta,E^*,\lambda''},
\end{eqnarray}
where both sums are indexed by the $(\mu',\mu'',\lambda',\lambda'')$ such that $\mu'+\mu''=\mu$ and $\lambda'+\lambda''=\lambda$. 
If $\mu=0$ and $0<|\lambda|$ then
\begin{eqnarray*}
&\displaystyle 
 {\bf T}_{p+q,\fn_{(\alpha,\beta)},H,(0,\lambda)} =\cdots= 
&
\\
& \displaystyle 
\sum_{\scriptscriptstyle \lambda'+\lambda''=\lambda}  {\bf T}_{q,\fn_\beta,E,\lambda'}\,  {\bf T}_{q,\fn_\beta,E^*,\lambda''} = {\bf T}_{q,\fn_\beta,E \circ E^*,\lambda} = {\bf T}_{q,\fn_\beta,\mathbb{I},\lambda} = 0,
\end{eqnarray*}
by using generators of type (iii), (ii) of $\dI$.
In a similar way, we have that ${\bf T}_{p+q,\fn_{(\alpha,\beta)},H,(\mu,0)}=0$ whenever $0<|\mu|$. Assume now that $\mu\neq 0$ and $\lambda\neq 0$. If $|\mu|<\ell_1$ or $|\lambda|<\ell_2$, then all the summands in (\ref{eq:aux-H}) vanish by (c) (remember that $\ell(D^*) =\ell(D)$ and $\ell(E^*) =\ell(E)$) and so ${\bf T}_{p+q,\fn_{(\alpha,\beta)},H,(\mu,\lambda)}=0$.
\medskip

\item[(6)] By using $F\pcirc G  = H \pcirc (G \pcirc F)$ and the generators of type (iii) of $\dI$ we have:
\begin{eqnarray*}
&\displaystyle
{\bf T}_{p+q,\fn_{(\alpha,\beta)},F\, \circ\, G,(\alpha,\beta)} = 
\sum_{\substack{\scriptscriptstyle  \alpha'+\alpha''=\alpha\\ \scriptscriptstyle  
\beta'+\beta''=\beta}}
{\bf T}_{p+q,\fn_{(\alpha,\beta)},H,(\alpha',\beta')}\, {\bf T}_{p+q,\fn_{(\alpha,\beta)},G\, \circ\, F,(\alpha'',\beta'')}.
&
\end{eqnarray*}
Hence:
\begin{eqnarray*}
&\displaystyle
{\bf T}_{p+q,\fn_{(\alpha,\beta)},F\, \circ\, G,(\alpha,\beta)} - {\bf T}_{p+q,\fn_{(\alpha,\beta)},G\, \circ\, F,(\alpha,\beta)}  = 
&\\
&\displaystyle
\sum_{\scriptscriptstyle  |\alpha'|+|\beta'|>0}
{\bf T}_{p+q,\fn_{(\alpha,\beta)},H,(\alpha',\beta')}\, {\bf T}_{p+q,\fn_{(\alpha,\beta)},G\, \circ\, F,(\alpha'',\beta'')} \stackrel{\text{(c)}}{=}
&\\
&\displaystyle
\sum_{\scriptscriptstyle   |\alpha'|+|\beta'|\geq \ell(H)}
{\bf T}_{p+q,\fn_{(\alpha,\beta)},H,(\alpha',\beta')}\, {\bf T}_{p+q,\fn_{(\alpha,\beta)},G\, \circ\, F,(\alpha'',\beta'')} \stackrel{\text{(4),(5)}}{=}
&\\
&\displaystyle
\sum_{\scriptscriptstyle  |\alpha'|\geq\ell_1,|\beta'|\geq \ell_2}
{\bf T}_{p+q,\fn_{(\alpha,\beta)},H,(\alpha',\beta')}\, 
{\bf T}_{q,\fn_\beta,E,\beta''} \,{\bf T}_{p,\fn_\alpha,D,\alpha''},
\end{eqnarray*}
where all the indexes $(\alpha',\alpha'',\beta',\beta'')$ in the above sums satisfy $\alpha'+\alpha''=\alpha$ and $\beta'+\beta''=\beta$,
and so, by (4):
\begin{eqnarray*}
&
\deg \left( {\bf T}_{p,\fn_\alpha,D,\alpha}\, {\bf T}_{q,\fn_\beta,E,\beta} - 
{\bf T}_{q,\fn_\beta,E,\beta}\,    {\bf T}_{p,\fn_\alpha,D,\alpha} \right) =
&\\
&
 \deg\left( 
{\bf T}_{p+q,\fn_{(\alpha,\beta)},F\, \circ\, G,(\alpha,\beta)} - {\bf T}_{p+q,\fn_{(\alpha,\beta)},G\, \circ\, F,(\alpha,\beta)} \right)\leq
&\\
&
\max \left\{ \deg\left(T_{p+q,\fn_{(\alpha,\beta)},H,(\alpha',\beta')}\right) + 
\deg\left(T_{q,\fn_\beta,E,\beta''} \right) + \deg\left(T_{p,\fn_\alpha,D,\alpha''} \right) \right\} =
&\\
& 
\max \left\{ \lfloor \frac{|\alpha'|+|\beta'|}{\ell_{(\alpha',\beta')}(H)}\rfloor + 
\lfloor \frac{|\beta''|}{\ell_{\beta''}(E)}\rfloor + \lfloor \frac{|\alpha''|}{\ell_{\alpha''}(D)}\rfloor \right\}\leq
&\\
& 
\max \left\{ \lfloor \frac{|\alpha'|+|\beta'|}{\ell(H)}\rfloor + 
\lfloor \frac{|\beta''|}{\ell(E)}\rfloor + \lfloor \frac{|\alpha''|}{\ell(D)}\rfloor \right\}
\leq
&
\\
& 
\max \left\{ \lfloor \frac{|\alpha'|+|\beta'|}{\ell_1+\ell_2}\rfloor + 
\lfloor \frac{|\beta''|}{\ell_2}\rfloor + \lfloor \frac{|\alpha''|}{\ell_1}\rfloor \right\}
\leq
\max \left\{ \lfloor \frac{|\alpha'|+|\beta'|}{\ell_1+\ell_2}\rfloor + 
\lfloor \frac{|\beta''|}{\ell_2}\rfloor + \lfloor \frac{|\alpha''|}{\ell_1}\rfloor \right\} <
&\\
& 
\lfloor \frac{|\alpha'+\alpha''|}{\ell_1}\rfloor + \frac{|\beta'+\beta''|}{\ell_2}\rfloor =
\lfloor \frac{|\alpha|}{\ell_1}\rfloor + \frac{|\beta|}{\ell_2}\rfloor =
\deg \left(  T_{p,\fn_\alpha,D,\alpha}\right) + \deg \left(  T_{q,\fn_\beta,E,\beta}\right),
\end{eqnarray*} 
where the $\max$'s are taken over the $\alpha',\alpha''\in \N^p$ and $\beta',\beta''\in\N^q$ such that $\alpha'+\alpha''=\alpha$, $\beta'+\beta''=\beta$, 
$  |\alpha'|\geq\ell_1$ and $|\beta'|\geq \ell_2$, and the last (strict) inequality comes from Lemma \ref{lemma:aux-parte-entera}.
\end{proof}

\begin{lemma} \label{lemma:aux-parte-entera}
Let $\ell_1,\ell_2 \geq 1$ be integers. For any integers $a',b',a'',b''\geq 0$ with $a'\geq \ell_1$, $b'\geq \ell_2$  we have:
$$\textstyle
\lfloor \frac{a'+b'}{\ell_1+\ell_2}\rfloor +
  \lfloor \frac{a''}{\ell_1}\rfloor + \lfloor \frac{b''}{\ell_2}\rfloor < \lfloor \frac{a'+a''}{\ell_1}\rfloor +
  \lfloor \frac{b'+b''}{\ell_2}\rfloor.
$$
\end{lemma}

\begin{proof}  We have
\begin{eqnarray*}
& 
\lfloor \frac{a'+b'}{\ell_1+\ell_2}\rfloor +
  \lfloor \frac{a''}{\ell_1}\rfloor + \lfloor \frac{b''}{\ell_2}\rfloor \leq 
\max\left\{ \lfloor \frac{a'}{\ell_1}\rfloor, \lfloor \frac{b'}{\ell_2}\rfloor \right\} 
+ \lfloor \frac{a''}{\ell_1}\rfloor + \lfloor \frac{b''}{\ell_2}\rfloor < 
&\\
& 
\lfloor \frac{a'}{\ell_1}\rfloor + \lfloor \frac{b'}{\ell_2}\rfloor 
+ \lfloor \frac{a''}{\ell_1}\rfloor + \lfloor \frac{b''}{\ell_2}\rfloor \leq
\lfloor \frac{a'+a''}{\ell_1}\rfloor + \lfloor \frac{b'+b''}{\ell_2}\rfloor.
\end{eqnarray*}
\end{proof}

\subsection{The case of HS-smooth algebras}

Our first goal is to define a canonical map of graded $A$-algebras from the divided power algebra of the module of {\em f-integrable $k$-derivations} (see Definitions  \ref{defi:power-divided-algebra} and \ref{def:HS-integ}) of $A$ to the graded ring of $\dU_{A/k}$. We will closely follow the procedure in \cite[\S 2.2]{nar_2009} (see also section \ref{sec:int-der-HS-smooth}).
\medskip

\begin{prop} \label{prop:bsigma-upsilon}
For each integer $m\geq 1$ the group homomorphism 
$$ \bupsigma \pcirc \Upupsilon^1_m: \HS_k(A;m) \longrightarrow \Ugr(\gr \dU_{A/k};m)
$$
vanishes on $\ker \tau_{m,1}$ and its image is contained in $\EXP_m(\gr \dU_{A/k})$.
\end{prop}

\begin{proof} Let us consider the combinatorial substitution maps  $\iota_1,\iota_2: A[[s]]_m \to A[[s_1,s_2]]_{(m,m)}$ given by $\iota_i(s) = s_i$, $i=1,2$, and the substitution map $\varphi: A[[s]]_m \to A[[s_1,s_2]]_m$ given by $\varphi(s) = s_1 + s_2$. Notice that $\init \iota_i = \iota_i$ and $\init \varphi = \varphi$ (see Proposition \ref{prop:init}). An element $r\in \U(\gr \dU_{A/k};m)$ belongs to $\EXP_m(\gr \dU_{A/k})$ if and only if $(\iota_1 \sbullet r) (\iota_2 \sbullet r) = \varphi\sbullet r$ (see \ref{defi:exponential-series}).
\medskip

Let $D\in \HS_k(A;m)$ be a HS-derivation, and let us denote $r=(\bupsigma \pcirc \Upupsilon^1_m)(D)$,
$E=\varphi\sbullet D$, $F=(\iota_1\sbullet D)\pcirc (\iota_2\sbullet D)$ and $H=E\pcirc F^*$. 
It is clear that $H_{(1,0)} = H_{(0,1)}=0$ and so $\ell(H) > 1$.
Then,
$$ \textstyle
\deg \left( {\bf T}_{1,\ft_m,H,(i,j)} \right) \leq \deg \left(T_{1,\ft_m,H,(i,j)} \right) = \lfloor \frac{i+j}{\ell_{(i,j)}(H)}  \rfloor \leq \lfloor \frac{i+j}{\ell(H)}  \rfloor < i+j
$$
for all $(i,j)$ with $0< i+j \leq m$, and so
\begin{equation} \label{eq:341}
  (\bupsigma \pcirc \Upupsilon^1_m)(H) = \bupsigma \left(\sum_{\scriptscriptstyle i+j\leq m} {\bf T}_{1,\ft_m,H,(i,j)} \, s_1^i s_2^j\right) =
\sum_{\scriptscriptstyle i+j\leq m} \sigma_{i+j}\left({\bf T}_{1,\ft_m,H,(i,j)}\right) s_1^i s_2^j = 1.
\end{equation}
We deduce that:
\begin{eqnarray*}
& \displaystyle  \varphi\sbullet r = (\init \varphi)\sbullet \left(\bupsigma \left( \Upupsilon^1_m \left(D\right) \right) \right) \stackrel{(\star)}{=} \bupsigma \left( \varphi\sbullet \Upupsilon^1_m \left(D\right) \right) =
\bupsigma \left( \Upupsilon^2_m (E) \right)=\bupsigma \left( \Upupsilon^2_m  (H \pcirc F) \right)=
&
\\ 
& \displaystyle 
   \bupsigma \left( \Upupsilon^2_m  (H) \, \Upupsilon^2_m  (F) \right) \stackrel{\text{(\ref{eq:341})}}{=}
 \bupsigma \left(  \Upupsilon^2_m  (F)\right) = 
 \bupsigma \left(  \Upupsilon^2_m  (\iota_1\sbullet D)\,\, \Upupsilon^2_m  (\iota_2\sbullet D)\right) =
&
\\
& \displaystyle  
\bupsigma \left( (\iota_1 \sbullet \Upupsilon^1_m (D))\,  (\iota_2 \sbullet \Upupsilon^1_m (D)) \right) =
\bupsigma \left( (\iota_1 \sbullet \Upupsilon^1_m (D)) \right)\,
\bupsigma \left( (\iota_2 \sbullet \Upupsilon^1_m (D)) \right) \stackrel{(\star)}{=}
&
\\
& \displaystyle  
\left( (\init \iota_1) \sbullet r \right) \, \left( (\init \iota_2) \sbullet r \right)=
( \iota_1 \sbullet r ) ( \iota_2 \sbullet r ),
 &
\end{eqnarray*}
where equalities $(\star)$ come from Proposition \ref{prop:sigma-varphi-bullet}, and so $r=(\bupsigma \pcirc \Upupsilon^1_m)(D) \in \EXP_m(\gr \dU_{A/k})$.
\medskip

On the other hand, if $D\in \ker \tau_{m,1}$, then $\ell(D)> 1$ and we can proceed as before with $H$ and deduce that $(\bupsigma \pcirc \Upupsilon^1_m)(D) = 1$.
\end{proof}

\begin{cor}
 There is a natural system of $A$-linear maps 
$$ \bupchi_m: \Ider_k(A;m) \longrightarrow \EXP_m(\gr \dU_{A/k}),\quad m\geq 1,
$$
such that for $m'\geq m$ the following diagram is commutative:

\begin{equation} \label{eq:diag-chi-m}
\begin{tikzcd}
\Ider_k(A;m') \ar[r,"\bupchi_{m'}"]  \ar[d,"\text{\rm incl.}"']  & \EXP_{m'}(\gr \dU_{A/k})
 \ar[d,"\text{\rm trunc.}"]\\
\Ider_k(A;m) \ar[r,"\bupchi_m"]  & \EXP_m(\gr \dU_{A/k}).
\end{tikzcd}
\end{equation}
Moreover, the system above induces a natural $A$-linear map
$\bupchi: \Ider^f_k(A) \longrightarrow \EXP(\gr \dU_{A/k})$.
\end{cor} 

\begin{proof} Since $\Ider_k(A;m)$ is by definition the image of the group homomorphism 
$$\tau_{m,1}:\HS_k(A;m) \to \HS_k(A;1) \equiv \Der_k(A),
$$ 
we deduce from Proposition \ref{prop:bsigma-upsilon} that the group homomorphism $\bupsigma \pcirc \Upupsilon^1_m$ induces a natural group homomorphism $\bupchi_m: \Ider_k(A;m) \longrightarrow \EXP_m(\gr \dU_{A/k})$. 
If $\delta \in \Ider_k(A;m)$, then $\bupchi_m(\delta) = \sum_{i=0}^m \sigma_i\left({\bf T}_{1,m,D,i} \right) s^i$ where
$D\in \HS_k(A;m)$ is any $m$-integral of $\delta$, i.e. $D_1=\delta$. Then, for each $a\in A$, $a\sbullet D$ is an $m$-integral of $a\delta$ and 
\begin{eqnarray*}
& \displaystyle 
\bupchi_m(a\delta) = \sum_{i=0}^m \sigma_i\left({\bf T}_{1,m,a\bullet D,i} \right) s^i \stackrel{(\star)}{=}
\sum_{i=0}^m \sigma_i\left(
\sum_{j=0}^i a^j {\bf T}_{1,m,D,j} \right)s^i=
&
\\
& \displaystyle 
= \sum_{i=0}^m  \sigma_i\left(a^i {\bf T}_{1,m,D,i}\right) s^i =
\sum_{i=0}^m  \sigma_i\left( {\bf T}_{1,m,D,i}\right) (as)^i = a \bupchi_m(\delta),
\end{eqnarray*}
where equality $(\star)$ comes from generators of type (v) of $\dI$, 
and so $\bupchi_m$ is $A$-linear (remember that the $A$-action on exponential type series is given by substitutions $s\mapsto as$, $a\in A$, see (\ref{eq:mod_action_on_EXP})). 
The commutativity of (\ref{eq:diag-chi-m}) comes from the commutativity of the following diagram ($\bupsigma$ and the $\Upupsilon^p_\Delta$ are compatible with truncations):
$$
\begin{tikzcd}
\HS_k(A;m') \ar[r,"\bupsigma \pcirc \Upupsilon^1_{m'}"] \ar[d,"\text{\rm trunc.}"']  & \EXP_{m'}(\gr \dU_{A/k})
\ar[d,"\text{\rm trunc.}"] \\
\HS_k(A;m)  \ar[r,"\bupsigma \pcirc \Upupsilon^1_m"] & \EXP_m(\gr \dU_{A/k}).
\end{tikzcd}
$$
The map $\bupchi$ is simply the inverse limit of the $\bupchi_m$.
\end{proof}

\begin{cor} \label{cor:varthetas-cd}
There is a natural map 
$\bupvartheta:\Gamma_A \Ider^f_k(A) \longrightarrow \gr \dU_{A/k}$ of graded $A$-algebras such 
that the following diagram is commutative:
\begin{equation}  \label{eq:CD-cor:varthetas-cd}
\begin{tikzcd}
\Gamma_A \Ider^f_k(A)  \ar[r,"\bupvartheta"]  \ar[rd,"\vartheta^f_{A/k}"'] &   \gr \dU_{A/k} \ar[d,"\gr \bupupsilon"]\\
 &  \gr \DD_{A/k},
\end{tikzcd}
\end{equation}
where $\vartheta^f_{A/k}$ is the map defined in (\ref{eq:mor-canon}) and $\bupupsilon$ is defined in Proposition \ref{prop:bupupsilon}.
\end{cor}

\begin{proof}  Let us denote 
$$\gamma:\delta\in\Ider^f_k(A) \longmapsto \sum_{n=0}^\infty \gamma_n(\delta) s^n \in \EXP (\Gamma_A \Ider^f_k(A))
$$
the canonical map (see \ref{defi:power-divided-algebra}). 
The existence of $\bupvartheta$ comes from the universal property of $\gamma$. Namely, there is a unique map of $A$-algebras $\bupvartheta:\Gamma_A \Ider^f_k(A) \longrightarrow \gr \dU_{A/k}$ such that $\bupchi = \EXP(\bupvartheta) \pcirc \gamma$. More explicitly, for each $\delta \in \Ider^f_k(A)$ and for each $D\in\HS_k(A;m)$ such that $D_1=\delta$, we have $\bupvartheta(\gamma_m(\delta)) = 
\sigma_m\left( {\bf T}_{1,m,D,m}\right)$. In particular, $\bupvartheta$ is graded. 
\medskip

The commutativity of the diagram (\ref{eq:CD-cor:varthetas-cd}) is a consequence of the commutativity of the diagram
\begin{equation*}  
\begin{tikzcd}
\Ider^f_k(A)  \ar[r,"\bupchi"]  \ar[rd,"\chi"'] &   \EXP(\gr \dU_{A/k}) \ar[d,"\EXP(\gr \bupupsilon)"]\\
 &   \EXP(\gr \DD_{A/k}),
\end{tikzcd}
\end{equation*}
where $\chi$ is the inverse limit of the maps $\chi_m: \Ider_k(A;m) \to \EXP_m(\gr \DD_{A/k})$, $m\geq 1$,  defined in  \cite[Corollary (2.7)]{nar_2009}.
\end{proof}

\begin{prop} \label{prop:bvartheta-surjective}
Assume that $\Ider^f_k(A)=\Der_k(A)$. Then, the map $$\bupvartheta:\Gamma_A \Ider^f_k(A) \longrightarrow \gr \dU_{A/k}$$ is surjective.
\end{prop}

\begin{proof} 
The $A$-algebra $\gr \dU_{A/k} $ is generated by the $\sigma_d\left({\bf T}_{q,\nabla,E,\beta}\right)$ for $q \geq 1$, $\nabla\in\coide{\N^q}$, $\beta\in\nabla$, $E\in\HS^q_k(A;\nabla)$, $E\neq\mathbb{I}$, $d=\lfloor \frac{|\beta|}{\ell_\beta(E)} \rfloor$. After \ref{nume:collect}, we may assume that $\nabla= \fn_\beta$ and so $\ell_\beta(E) = \ell(E)$. Let us call $m=\height(\nabla)$.
\medskip

Let $\{\delta_s, s\in \bfs\}$ be a system of generators of the $A$-module $\Der_k(A)$. Since $\Ider_k(A;m)=\Der_k(A)$,
for each $s\in\bfs$ there exists $D^s\in \HS_k(A;m)$ which is an $m$-integral of $\delta_s$. By considering some total ordering $<$ on $\bfs$, we can define $D\in \HS^\bfs_k(A;m)$ as the external product (see Definition \ref{defi:external-x}) of the ordered family $\{D^s, s\in\bfs\}$, i.e. $D_0=\Id$ and for each $\alpha \in\N^{(\bfs)}$, $\alpha\neq 0$,
$$  D_\alpha = D^{s_1}_{\alpha_{s_1}} \pcirc \cdots \pcirc D^{s_e}_{\alpha_{s_e}}\quad \text{with\ }\ 
\supp \alpha = \{s_1 <  \cdots < s_e\}.
$$
After \cite[Theorem 1]{nar_subst_2018}, there exists a substitution map $\varphi_0:A[[\bfs]]_m \to A[[t_1,\dots,t_q]]_\nabla$ such that $E=\varphi_0 \sbullet D$. Moreover, it is clear
that we can take $\ord(\varphi_0) = \ell(E)$.
\medskip

Since $\nabla$ is finite, condition (17) in \cite[Proposition 2]{nar_subst_2018} implies that the set $\{s\in\bfs\ |\ \varphi_0(s)\neq 0\}$ is finite. Let us call $\{s_1 < \dots< s_p\}$ this set. We have a factorization of substitution maps:
$$
\begin{tikzcd}
A[[\bfs]]_m  \ar[rr,"\varphi_0"] \ar[dr,"\varphi_1"'] & & A[[t_1,\dots,t_q]]_\nabla \\
 &  A[[s_1,\dots,s_p]]_m \ar[ur,"\varphi"] & 
\end{tikzcd}
$$
where $\varphi_1(s) = 0$ if $s\neq s_i$, $\varphi_1(s_i)=s_i$ and $\varphi(s_i) = \varphi_0(s_i)$. Then we have
$E=\varphi_0 \sbullet D = \varphi \sbullet F$ with $F=\varphi_1\sbullet D =
D^{s_1}\boxtimes \cdots \boxtimes D^{s_p}\in \HS^p_k(A;(m,\dots,m))$. 
\medskip

We obviously have $\ord(\varphi)=\ord(\varphi_0) = \ell(E)$ and so ${\bf C}_\beta(\varphi,\alpha)=0$ whenever $|\alpha|\ell(E)> |\beta|$.
So,
\begin{eqnarray*}
& \displaystyle
 {\bf T}_{q,\nabla,E,\beta} =  \sum_{\substack{\scriptscriptstyle |\alpha| \leq  m\\ \scriptscriptstyle |\alpha|\leq |\beta|  }} {\bf C}_\beta(\varphi,\alpha) {\bf T}_{p,\underline{m},F,\alpha} =
&
\\
& \displaystyle 
\sum_{\substack{\scriptscriptstyle |\alpha| \leq  m\\ \scriptscriptstyle |\alpha|\ell(E)\leq |\beta|  }} {\bf C}_\beta(\varphi,\alpha) {\bf T}_{1,m,D^{s_1},\alpha_1} {\bf T}_{1,m,D^{s_2},\alpha_2} \cdots {\bf T}_{1,m,D^{s_p},\alpha_p},
&
\\
& \displaystyle
\sigma_d\left({\bf T}_{q,\nabla,E,\beta}\right) = \sum_{\scriptscriptstyle 
|\alpha|=d}
{\bf C}_\beta(\varphi,\alpha)\, \prod_{j=1}^p \sigma_{\alpha_j}\left({\bf T}_{1,m,D^{s_j},\alpha_j}\right) =
\bvartheta\left(
\sum_{|\alpha|=d}
{\bf C}_\beta(\varphi,\alpha)\, \prod_{j=1}^p \gamma_{\alpha_j}(\delta_j)\right)
\end{eqnarray*}
and we deduce that $\bvartheta$ is surjective.
\end{proof}

\begin{nota} In the proof of the above proposition we have used the Axiom of Choice in order to consider a total ordering on $\bfs$. This could be avoided when $\Der_k(A)$ is a finitely generated $A$-module. In general, we could also avoid the Axiom of Choice by proving directly a convenient variant of Theorem 1 of   \cite{nar_subst_2018}.
\end{nota}

\begin{thm}
If $A$ is a HS-smooth $k$-algebra, then the natural map $\bupupsilon: \dU_{A/k} \longrightarrow \DD_{A/k}$ is an isomorphism of filtered $k$-algebras.
\end{thm}

\begin{proof} It is enough to prove that $\gr \bupupsilon: \gr \dU_{A/k} \longrightarrow \gr \DD_{A/k}$ is an isomorphism of graded $A$-algebras. Since $A$ is a HS-smooth $k$-algebra, we have $\vartheta^f_{A/k}: \Gamma_A \Ider^f_k(A)  \xrightarrow{\sim} \gr \DD_{A/k}$ and from Corollary \ref{cor:varthetas-cd} we deduce that $\bupvartheta$ is injective. The surjectivity of $\bupvartheta$  comes from Proposition
\ref{prop:bvartheta-surjective}.
\end{proof}

\begin{cor} If $A$ is a HS-smooth $k$-algebra, then the category of left (resp. right) HS-modules over $A/k$ is isomorphic to the category of left (resp. right) $\DD_{A/k}$-modules.
\end{cor}

\subsection{Further developments and questions}

\begin{question}
With the hypotheses of the preceding section, it is easy to see that the map
$$ \Upupsilon^1_1: \HS_k(A;1)\equiv \Der_k(A) \longrightarrow \U(\dU_{A/k};1) \equiv \dU_{A/k}
$$
is $k$-linear, compatible with Lie brackets and satisfies Leibniz rule. So, it induces a $k$-algebra map from the enveloping algebra of the Lie-Rinehart algebra $\Der_k(A)$ (\cite{rine-63}) to $\dU_{A/k}$. The paper \cite{nar_HSmod_vs_IC} 
is devoted to prove that this map is an isomorphism whenever $\QQ \subset k$, and so HS-modules 
 and classical integrable connections coincide in characteristic $0$.
\end{question}

\begin{question} Assume that $A$ is a HS-smooth $k$-algebra and $\Omega_{A/k}$ is a projective $A$-module of rank $d$. In an article in preparation we study how the operations in Proposition \ref{prop:HS-rep-sym-altern}, the 
pre-HS-module structure on $\Omega_{A/k}$ (see Proposition \ref{prop:pre-HS-Omega}) and Proposition \ref{prop:carac-dual-HS-structures} give rise to a right HS-module structure on the dualizing module $\omega_{A/k} = \Omega^d_{A/k}$.
\end{question}

\bigskip

\noindent {\small \href{http://personal.us.es/narvaez/}{Luis Narv\'aez Macarro}\\
\noindent \href{http://departamento.us.es/da/}{Departamento de \'Algebra} \&\
\href{http://www.imus.us.es}{Instituto de Matem\'aticas (IMUS)}\\
\href{http://matematicas.us.es}{Facultad de Matem\'aticas}, \href{http://www.us.es}{Universidad de Sevilla}\\
Calle Tarfia s/n, 41012  Sevilla, Spain} \\
{\small {\it E-mail}\ : narvaez@us.es
 }

\end{document}